\documentclass{article}
\usepackage{amsmath}
\usepackage{amssymb}
\usepackage{amsthm}

\newcommand{\A}{\mathfrak{A}}
\newcommand{\B}{B_n}

\newcommand{\tX}{\tilde{X}}
\newcommand{\Fgp}{\tX^{\geq p}}
\newcommand{\Xgp}{\tX^{\geq p}}
\newcommand{\Xp}{\tX^p}

\newcommand{\delK}{\partial_K}

\newcommand{\delA}{\partial_A}

\newcommand{\Q}{\mathbb{Q}}

\newcommand{\Pri}{\mathcal{P}}

\newcommand{\N}{\mathbb{N}}

\newcommand{\IK}{I_K}
\newcommand{\IF}{I_F}

\newcommand{\IKh}{(I_K/I_K^2)}
\newcommand{\IFh}{(I_F/I_F^2)}

\newcommand{\vB}{v \negthinspace B_n}

\newcommand{\grK}{gr_I K}

\newcommand{\grPV}{gr_I \Q \PV}

\newcommand{\QG}{\Q G}

\newcommand{\Rg}{\mathfrak{R}}

\newcommand{\Kipq}{\Rg_{p,q}}
\newcommand{\Kip}{\Rg_p}

\newcommand{\PV}{P \negthinspace v \negthinspace B_n}
\newcommand{\PfB}{P \negthinspace f \negthinspace B_n}
\newcommand{\pf}{\mathfrak{pfb}_n}
\newcommand{\pfq}{\mathfrak{pfb}_n^!}

\newcommand{\PB}{P \negthinspace B_n}
\newcommand{\pb}{\mathfrak{pb}_n}
\newcommand{\pbq}{\mathfrak{pb}_n^!}

\newcommand{\pv}{\mathfrak{pvb}_n}
\newcommand{\pvq}{\mathfrak{pvb}_n^!}

\newcommand{\grpv}{grP \negthinspace v \negthinspace B_n}
\newcommand{\qpv}{\Q P \negthinspace v \negthinspace B_n}
\newcommand{\rij}{r_{ij}}
\newcommand{\rik}{r_{ik}}
\newcommand{\rjk}{r_{jk}}
\newcommand{\rji}{r_{ji}}
\newcommand{\rki}{r_{ki}}
\newcommand{\rkj}{r_{kj}}
\newcommand{\rkl}{r_{kl}}
\newcommand{\rlk}{r_{lk}}
\newcommand{\ril}{r_{il}}
\newcommand{\rli}{r_{li}}
\newcommand{\rjl}{r_{jl}}
\newcommand{\rlj}{r_{lj}}
\newcommand{\rst}{r_{st}}
\newcommand{\rijq}{r_{ij}^*}
\newcommand{\rikq}{r_{ik}^*}
\newcommand{\rjkq}{r_{jk}^*}
\newcommand{\rjiq}{r_{ji}^*}
\newcommand{\rkiq}{r_{ki}^*}
\newcommand{\rkjq}{r_{kj}^*}
\newcommand{\rklq}{r_{kl}^*}
\newcommand{\rlkq}{r_{lk}^*}
\newcommand{\rilq}{r_{il}^*}
\newcommand{\rliq}{r_{li}^*}
\newcommand{\rjlq}{r_{jl}^*}
\newcommand{\rljq}{r_{lj}^*}
\newcommand{\Yijk}{Y_{ijk}}
\newcommand{\Cijkl}{C_{ij}^{kl}}
\newcommand{\cijkl}{c_{ij}^{kl}}
\newcommand{\cikjl}{c_{ik}^{jl}}
\newcommand{\ciljk}{c_{il}^{jk}}
\newcommand{\yijk}{y_{ijk}}
\newcommand{\yijl}{y_{ijl}}
\newcommand{\yikl}{y_{ikl}}
\newcommand{\yjkl}{y_{jkl}}
\newcommand{\pvqiii}{\mathfrak{pvb}_n^{!3}}
\newcommand{\pvqii}{\mathfrak{pvb}_n^{!2}}
\newcommand{\xp}{x_p}
\newcommand{\bxp}{\bar{x}_p}
\newcommand{\bx}{\bar{x}}

\newcommand{\Rq}{R}

\newcommand{\Rqp}{R^{\perp}}

\newcommand{\xmi}{R_{m,i}}
\newcommand{\xmj}{R_{m,j}}

\newcommand{\RK}{\Rg}

\newcommand{\RKp}{\Rg_p}
\newcommand{\RKgp}{\Rg_{\geq p}}

\newcommand{\RF}{\Rg^F}

\newcommand{\YF}{\mathcal{Y}}
\newcommand{\w}{\wedge}
\newcommand{\sw}{\overset{*}{\wedge}}

\newcommand{\delSyz}{\partial_{Syz}}
\newcommand{\Diiti}{\Delta^{!}_{2,1}}
\newcommand{\Ditii}{\Delta^{!}_{1,2}}
\newcommand{\Diti}{\Delta^{!}_{1,1}}
\newcommand{\bR}{\overline{R}}
\newcommand{\si}{\sigma_i}
\newcommand{\sii}{\sigma_{i+1}}
\newcommand{\sj}{\sigma_j}

\newcommand{\Eii}{S^{(2)}}
\newcommand{\FSyz}{\pi^{Syz}}
\newcommand{\Gi}{G^{(1)}}
\newcommand{\Gii}{G^{(2)}}
\newcommand{\Giii}{G^{(3)}}
\newcommand{\Gj}{G^{(j)}}
\newcommand{\GI}{\Gi \negthinspace / \Gii}
\newcommand{\GII}{\Gii \negthinspace / \Giii}
\newcommand{\Gp}{G^{(p)}}
\newcommand{\Gpi}{G^{(p+1)}}
\newcommand{\GP}{\Gp \negthinspace / \Gpi}
\newcommand{\tx}{\tilde{x}}
\newcommand{\tz}{\tilde{z}}
\newcommand{\mkl}{m_{kl}}

\newcommand{\ra}{\rightarrow}

\newcommand{\thra}{\twoheadrightarrow}
\newcommand{\Sla}{S_{\lambda,a}}
\newcommand{\ola}{(\overline{\lambda a})^{-1}}
\newcommand{\Aij}{A_{ij}}
\newcommand{\aij}{a_{ij}}
\newcommand{\aik}{a_{ik}}
\newcommand{\ajk}{a_{jk}}
\newcommand{\akl}{a_{kl}}

\newtheorem{proposition}{Proposition}
\newtheorem{remark}{Remark}
\newtheorem{theorem}{Theorem}
\newtheorem{lemma}{Lemma}
\newtheorem{corollary}{Corollary}

\usepackage[all,knot,poly]{xy}

\begin{document}
\title{The Pure Virtual Braid Group is Quadratic}
\author{Peter Lee}
\maketitle

\begin{abstract}
If an augmented algebra $K$ over $\Q$ is filtered by powers of its augmentation ideal $I$, the associated graded algebra $gr_I K$ need not in general be quadratic:  although it is generated in degree 1, its relations may not be generated by homogeneous relations of degree 2.   In this paper we give a sufficient criterion (called the PVH Criterion) for $gr_I K$ to be quadratic.  When $K$ is the group algebra of a group $G$, quadraticity is known to be equivalent to the existence of a (not necessarily homomorphic) universal finite type invariant for $G$.  Thus the PVH Criterion also implies the existence of such a universal finite type invariant for the group $G$.  We apply the PVH Criterion to the group algebra of the pure virtual braid group (also known as the quasi-triangular group), and show that the corresponding associated graded algebra is quadratic, and hence that these groups have a universal finite type invariant.
\paragraph{MSC subject class:} Primary: 16S34; Secondary: 16S37, 20F38, 20F40
\paragraph{Keywords:} Group rings; Quadratic and Koszul algebras; Pure Virtual Braid Group
\end{abstract}

\tableofcontents

\newpage

\section{Introduction}
\label{Intro}

This paper will ultimately be concerned with the pure virtual braid groups $\pmb{\PV}$, for all $n \in \N$, generated by symbols $R_{ij}$, $1 \leq i \neq j \leq n$, with relations the Reidemeister III moves (or quantum Yang-Baxter relations) and certain commutativities:

\begin{align}
R_{ij} R_{ik} R_{jk} & = R_{jk} R_{ik} R_{ij} \label{relations1} \\
R_{ij} R_{kl} & = R_{kl} R_{ij}, \quad \quad
\label{relations2}
\end{align}
with $i,j,k,l$ distinct.  This group is referred to as the quasi-triangular group $\pmb{QTr_n}$ in \cite{BEER}.

We will also be concerned with the related algebra $\pv$, generated by symbols $r_{ij}$, $1 \leq i \neq j \leq n$, with relations the `6-term' (or `classical Yang-Baxter') relations, and related commutativities:

\begin{align}
\yijk:=  [r_{ij},r_{ik}] + & [r_{ij},r_{jk}] + [r_{ik},r_{jk}] = 0,  \label{6Trels}
\\
& \cijkl:= [r_{ij},r_{kl}] = 0 \quad \quad \quad  \label{AlgCommutRels}
\end{align}
with $i,j,k,l$ distinct.  This algebra is the universal enveloping algebra of the quasi-triangular Lie algebra $\mathfrak{qtr_n}$ in \cite{BEER}.

The group of pure flat braids $\pmb{\PfB}$ is given by the same generators and relations as $\PV$, but with additional relations $R_{ij}R_{ji}=1$.  Associated to it is the algebra $\pmb{\pf}$ with the same generators and relations as $\pv$, but with the additional relations $r_{ij}+r_{ji}=0$.  $\PfB$ is referred to as $T \negthinspace r_n$ in \cite{BEER}, while $\pf$ is the universal enveloping algebra of the Lie algebra $\mathfrak{tr_n}$ of \cite{BEER}.

We will show that $\PV$ is a `quadratic group', in the sense that if its rational group ring $\qpv$ is filtered by powers of the augmentation ideal $I$, the associated graded ring $\grpv$ is a quadratic algebra:  i.e., a graded algebra generated in degree 1, with relations generated by homogeneous relations of degree 2.  We note that, in different language, this is the statement that $\PV$ has a universal finite-type invariant, which takes values in the algebra $\pv$.  We will also prove the corresponding statements for $\PfB$.

In \cite{Hutchings}, a criterion was given for the quadraticity of the pure braid group.  The proof relied on the geometry of braids embedded in $\mathbf{R}^3$.  In order to generalize this criterion to all finitely presented groups, we developed an algebraic proof of the criterion.  This proof turned out not to rely on the existence of an underlying group, and applies instead to algebras over $\Q$, filtered by powers of an augmentation ideal $I$.  Indeed, this criterion arguably lives naturally in an even broader context, such as perhaps augmented algebras over an operad (or the related `circuit algebras' of \cite{DBN-WKO}), although we do not investigate this broader context here.

Our criterion may be summarized as follows in the case of an augmented algebra $K$ with augmentation ideal $\IK$.  We denote by $\grK = \oplus_{m\geq 0} \IK^m/\IK^{m+1}$ the associated graded algebra of $K$ with respect to the filtration by powers of the augmentation ideal.  Let $A$ be the `quadratic approximation' of $\grK$, namely the graded algebra with the same generators and with ideal of relations generated by the degree 2 relations of $\grK$.  We will see that, in fact, we can view the generators of $K$ as also generating $A$, and interpret a certain space $\RF$ of free generators of the relations in $K$ as generating the relations in $A$.  It thus makes sense to ask whether the relations among the elements of $\RF$, when viewed as relations in $A$, also hold when these are viewed as relations in $K$ - i.e., informally, whether the syzygies in $A$ also hold in $K$.  We show that if the syzygies of $A$ do also hold in $K$, then $\grK$ is quadratic.  Furthermore, if $A$ is Koszul, we show that it is sufficient to check this criterion in degrees 2 and 3.  The precise statement of this criterion, which we call the PVH Criterion, appears in Theorem \ref{HutchingsCriterion}.

We note that it is easy to construct algebras $K$ such that $\grK$ is not quadratic (and in particular do not satisfy the above criterion).  For instance we could take $K=\Q\langle x\rangle / \langle x^3\rangle$, a graded algebra with $\langle x\rangle / \langle x^3\rangle$ as the augmentation ideal (so that $\grK = K$).  If we define $\RF$ to be the 2-sided $\Q\langle x\rangle$-module generated by the symbol set $\{r\}$, and $\delK: \RF \ra \Q\langle x\rangle$ as the $\Q\langle x\rangle$-module homomorphism such that $\delK r = x^3$, then $\RF$ generates the relations in $K$ via this map.  Since $x^3$ is cubic, the quadratic approximation $A$ is just $\Q\langle x\rangle$, but again $\RF$ generates the (single, trivial) relation in $A$ via the zero map $\delA: \RF \overset{0}{\ra} \Q\langle x\rangle$.  The kernel of $\delK$ and $\delA$ are what we call the syzygies of $K$ and $A$, respectively.  If we agree that the symbol $r$ has degree 2, then $\RF$ and $ker\ \delA \subseteq \RF$ inherit a grading from $\Q\langle x\rangle$.  We will see that there are always induced maps $\pi^{Syz}_p: ker\ \delK \ra [ker\ \delA]_p$ (where the latter refers to the degree $p$ component of $ker\ \delA$), and the PVH Criterion in degree $p\geq 2$ is that this map should be surjective.  Since, in our example, $ker\ \delK =\{0\}$ while $ker\ \delA = \RF \ne \{0\}$, we see that there can be no such surjection and $K$ will not satisfy the criterion (and, indeed, $\grK$ is not quadratic).

In Section 2 of this paper we give some useful background, both for the concept of quadraticity and for the pure virtual (and pure flat) braid groups.  We first briefly review two concepts that are perhaps better known than quadraticity, namely that of a group expansion, and that of 1-formality of a group.  In particular we show that quadraticity is equivalent to the existence of a particular type of group expansion, specifically a (non-homomorphic) universal finite type invariant for the group $G$.  We then introduce the weaker, graded versions of expansion and 1-formality.  We show that for a reasonably broad class of groups (including those with finite type rational homology, and in particular $\PfB$ and $\PV$) quadraticity is also equivalent to graded 1-formality.  Although this material is not needed in the sequel, it should help to situate the concept of quadraticity in terms of concepts that may be more broadly familiar.  We then give a more leisurely introduction to $\PV$ and $\PfB$, and discuss the notion of an expansion and quadraticity as applied to those groups.  In particular, we recall that in \cite{BEER} it was shown that $\PfB$ and $\PV$ are not 1-formal, which is the motivation for proving that that they are, nonetheless, quadratic.

In Section 3 we set the stage for and give a precise statement of the quadraticity criterion (see Theorem \ref{HutchingsCriterion}).  In Section 4 we supply details of proofs that were omitted in Section 3.  In Section 5 we specialize to $\PV$.  We present a basis for the quadratic dual algebra $\pvq$, and use this basis to compute the syzygies of $\pv$ and prove that $\PV$ satisfies the quadraticity criterion.  It follows that $\PV$ is quadratic (see Theorem \ref{ThmPvBQuadratic}).  Because $\PfB$ is a split quotient of $\PV$, it also follows that $\PfB$ is quadratic (see Corollary \ref{CorPfBQuadratic}), which confirms a result originally announced in \cite{BEER}.

Although Koszulness of the algebra $\pv$ was originally established in \cite{BEER}, we give a somewhat different proof in Subsection \ref{ProofOfKoszulness}, by exhibiting a quadratic Gr\"obner basis for $\pvq$.

Finally, in Section 6 we point out some possible future avenues of research.

We note that the quadraticity of $\PV$ was conjectured in \cite{BEER}.  As pointed out in section 8.5 of that paper, the quadraticity of $\PV$ implies that $H^*(\PV) \cong \mathfrak{pvb}_n^{!}$ as algebras (this is Conjecture 8.6 of \cite{BEER}).  Similarly, the quadraticity of $\PfB$ implies that $H^*(\PfB) \cong \mathfrak{\pf}^{!}$ as algebras (which is Theorem 8.5 of \cite{BEER}).  See Corollary \ref{CorQTrCohomology} and Corollary \ref{CorTrCohomology}.

After this paper was substantially completed, the result was communicated to Alexander Polishchuk, who pointed out that a theorem similar to Theorem \ref{HutchingsCriterion} was obtained in \cite{PosVish} in the context of the cohomology algebra of a nilpotent augmented coalgebra, albeit by different methods.  For this reason we have referred to the criterion in Theorem \ref{HutchingsCriterion} as the PVH Criterion (with reference to Positselski, Vishik and Hutchings).

\subsection{Acknowledgements}

Many thanks to A. Polishchuk, L. Positselski, P. Etingof and E. Rains for their feedback on the paper; and to the referee of the published version of this paper,\footnote{\emph{Sel. Math. New Ser.} (October 2012).  DOI: 10.1007/s00029-012-0107-1.} for pointing out a number of inaccuracies, and for making some helpful suggestions to improve the presentation.  And special thanks to my Ph.D. supervisor, D. Bar-Natan, for his long-term support, encouragement, insights and guidance.

\section{Background}

In this section, which is not needed in the sequel, we will give a fairly concise summary of two concepts, that of group expansion and that of 1-formality, and how they relate to quadraticity.  In particular we show that quadraticity is equivalent to the existence of a (non-homomorphic) expansion, also known as a (non-homomorphic) universal finite type invariant for the group $G$.  Although this material is not needed later in the paper, it should help to situate the significance of the concept of quadraticity in terms of concepts that may be more broadly familiar.  We also give additional background on the pure virtual braid groups $\PV$, emphasizing their relation to the better-known braids groups, and use this to give a pictorial interpretation of the $\PV$.  The material in this Chapter is mostly well known and we have attempted to give suitable references.

\subsection{Group Expansions}
\label{GroupExpSubsection}

In this subsection we suppose $K$ is an augmented algebra over $\Q$ generated by a set $X$, with augmentation ideal $\IK$ given as the kernel of the augmentation map which sends $x \mapsto 1$ for $x \in X$.  $K$ is filtered by powers of $\IK$.  In typical applications, $K$ will be the group algebra of some group $G$.

We recall that the quadratic approximation $q(\grK)$ is the graded algebra with the same generators as $\grK = \oplus_{p\geq 0} \IK^p/\IK^{p+1}$, but with only the quadratic relations of $\grK$. Thus if we put $V:= \IK/\IK^2$, and denote $\langle\partial\Rg\rangle$ the ideal generated in the tensor algebra $TV$ by the degree 2 relations $\partial\Rg$ (the notation will be explained below), then $q(\grK) = \frac{TV}{\langle\partial\Rg\rangle}$.  Moreover, there is a canonical projection $\mu: q(\grK) \twoheadrightarrow \grK$ which is induced by the identity in degree 1.

Following \cite{Lin} and \cite{DBN-WKO}, we define a (homomorphic, quadratic) expansion for $K$ to be a homomorphism of filtered algebras $Z: K \rightarrow \widehat{q(\grK)}$ (where the hat denotes completion) whose associated graded is the identity map in degree 1.  The latter requirement implies in particular that, on the generators $x\in X$ of $K$, $Z$ takes the form:

\begin{equation}
x \mapsto 1 + \tx + \text{ h.d.} \label{GradedIs1}
\end{equation}
where $\tx:= (x-1) \mod \IK^2 \in \IK/\IK^2$ is a generator in $q(\grK)$, and `h.d.' refers to higher degree terms in the $\tz,\ z\in X$.  Indeed, we then see that:

\begin{equation*}
(x -1) \mapsto \tx + \text{ h.d.}
\end{equation*}
and thus the associated graded map $grZ$ is the identity in degree 1.

Since $grZ \circ \mu = id$ in degree 1, the composition $grZ \circ \mu$ must be the identity in all degrees.  Consequently, $\mu$ must be injective, hence an isomorphism.

One can show that one gets an equivalent definition of expansion if one requires that $Z$ be a filtered algebra isomorphism $\widehat{K} \overset{\sim}{\rightarrow} \widehat{q(\grK)}$ (where again the hats denote completion), instead of a filtered homomorphism $K \rightarrow \widehat{q(\grK)}$ (but with the definition otherwise the same).

Thus the situation may be summarized by the following diagram:

\[
\xy
(-15,0)*+{\widehat{K}} = "1";
(15,0)*+{\widehat{q(\grK)}} = "2";
(0,-20)*+{\widehat{\grK}} = "3";
{\ar@{->}^{Z} "1"; "2"};
{\ar@{.>}_{Z'=\mu\circ Z} "1"; "3"};
{\ar@{->}^{gr Z} "3"; "2"};
{\ar@{->>}^{\mu} (15,-4)*{}; (5,-17)*{}}
\endxy
\]

\newpage

Then the following theorem is clear:

\begin{theorem}
The existence of a (homomorphic, quadratic) expansion $Z$ is equivalent to:
\begin{enumerate}
\item the existence of a filtered isomorphism $Z': \widehat{K} \overset{\sim}{\longrightarrow} \widehat{\grK}$ whose associated graded is the identity in degree 1; plus
\item the natural surjection $\mu$ being an isomorphism - i.e. $\grK$ is quadratic.
\end{enumerate}
\end{theorem}

We thus see that quadraticity can be viewed as a `graded' version of the existence of a homomorphic, quadratic expansion. It is a necessary, but not sufficient, condition for the existence of such an expansion.

\subsection{Hopf Algebra Expansions and 1-Formality}
\label{Sec1Formality}

In this Subsection \ref{Sec1Formality}, we take all our filtered or graded spaces (including $K$, $\grK$ and others) to be completed, whether or not this is indicated with a hat.

When $K=\widehat{\Q G}$ is a group algebra, one can step up the requirements on an expansion and require that the map $Z$ respect the Hopf structure -- one then speaks of a Hopf algebra expansion.  In this subsection, we will explain the Hopf algebra structures on $K$ and $q(\grK)$, and briefly discuss Hopf algebra expansions and the related concept of 1-formality.

One may view $K$ as a Hopf algebra using the co-product induced from $G$ (i.e. $\Delta(g) = g\otimes g$ for $g\in G$) and the augmentation map as the co-unit.  Thus $K$ has a Lie algebra of primitives $\Pri(K)$, which is the `Malcev Lie algebra' $M_G$ of $G$ \cite{Quillen2}.  Also, one can show that the degree 2 relations $\partial\Rg$ of $q(\grK)$ are all Lie elements of $TV$.  Hence the Hopf algebra structure on $\widehat{TV}$ (the completed tensor algebra) induces a Hopf algebra structure on ${q(\grK)}$, and it is fairly easy to show that the Lie algebra of primitives of ${q(\grK)}$ is just $\frac{\widehat{LieV}}{\langle\partial\Rg\rangle}$, the quotient of the completed free Lie algebra on $V$ by the Lie ideal generated by $\partial\Rg$.

It thus makes sense to ask whether there is a filtered isomorphism between the respective Lie algebras of primitives $Z: M_G \rightarrow \Pri({q(\grK)})$ whose associated graded is the identity in degree 1 (both Lie algebras have the filtrations induced from the lower central series).

Actually, 1-formality is commonly defined as the existence of a filtered isomorphism of Lie algebras $M_G \rightarrow hol(G)$ whose associated graded is the identity in degree 1, where $hol(G)$ is the `holonomy' Lie algebra of $G$ (defined below), which is in particular a graded Lie algebra completed with respect to the lower central series filtration.\footnote{See also the summary of 1-formality in \cite{PapSu}.}  However, if $G$ is a group whose rational homology is of finite type, then (as explained below) $hol(G)\cong \Pri({q(\grK))}$, so in these cases 1-formality is exactly the Lie algebra analogue of the existence of an expansion.

In particular, when $G$ has finite type rational homology, the existence of a Hopf algebra expansion implies 1-formality, since the expansion induces a filtered Lie algebra homomorphism on the primitives.\footnote{\label{ABCKT}In fact, by a theorem in \cite{ABCKT} (originally due to Morgan), 1-formality for a finitely presentable group is equivalent to the existence of a filtered isomorphism between $M_G$ and any quadratic Lie algebra (with the filtration induced by the lower central series), although their proof uses real coefficients. When this theorem applies, the existence of a Hopf algebra expansion would always imply 1-formality.}

One can also consider a weaker, graded version of 1-formality, i.e. we consider a group to be `graded 1-formal' if $grM_G \cong gr(hol(G))$ ($\cong hol(G)$), with the isomorphism being induced from the identity map in degree 1.

If the filtered isomorphism $Z$ exists, then its associated graded $grZ$ must be a graded Lie algebra isomorphism.  Thus, as with expansions, one has the following diagram:

\[
\xy
(-15,0)*+{M_G} = "1";
(15,0)*+{hol(G)} = "2";
(0,-20)*+{grM_G} = "3";
{\ar@{->}^{Z} "1"; "2"};
{\ar@{.>}_{Z'} "1"; "3"};
{\ar@{->}^{gr Z} "3"; "2"}
\endxy
\]

And again we have the theorem:

\begin{theorem}
1-formality is equivalent to:
\begin{enumerate}
\item the existence of a filtered Lie algebra isomorphism $Z': M_G \overset{\sim}{\longrightarrow} grM_G$ whose associated graded is the identity in degree 1; plus
\item the existence of a graded Lie algebra isomorphism $hol(G) \cong grM_G$ induced by the identity in degree 1 - - i.e. graded 1-formality.
\end{enumerate}
\end{theorem}

Thus graded 1-formality is a necessary, but not sufficient, condition for 1-formality.

We will now amplify briefly on graded 1-formality, and in particular explain how, for groups with finite type rational homology, graded 1-formality is equivalent to quadraticity.

\subsubsection{The Associated Graded of $M_G$}

For purposes of graded 1-formality we need a good description of $grM_G$.  By work of Quillen,\footnote{\cite{Quillen} and \cite{Quillen2}.  See also the summary of the construction of $M_G$ and $grM_G$ in \cite{PapSu}.} we have $grM_G \cong grG \otimes \Q$ as graded Lie algebras, where $grG:= \hat{\oplus}_{p\geq0}\GP$ (completed direct sum) and $\Gj$ is the $j$-th term in the lower central series of $G$.\footnote{See \cite{MKS}, section 5.7, for the Lie algebra structure on $grG$.}

As discussed in \cite{Quillen}, $\grK$ has the structure of a graded Hopf algebra which is primitively generated (specifically by $gr_1K$, all of which is primitive).  It follows by the Milnor-Moore theorem that $\grK \cong U(\Pri(\grK))$, where $\Pri(\grK)$ denotes the Lie algebra of primitives of $\grK$ and $U(-)$ denotes the universal enveloping algebra.

Again by \cite{Quillen}, there is an isomorphism of graded Lie algebras $grG \otimes \Q \cong \Pri(\grK)$ and hence $U(grG \otimes \Q) \cong \grK$.  In particular, we have a Lie algebra isomorphism $grM_G \cong \Pri(\grK)$.

\subsubsection{The Holonomy Lie Algebra and Its Associated Graded}

We now define the holonomy Lie algebra.  Let $X$ be an Eilenberg-MacLane $K(G,1)$ space for $G$, and put $H^i:=H^i(X,\Q)$ for $i\geq 0$.  Let $\cup: H^1 \wedge H^1 \rightarrow H^2$ be the cup product, and $\partial: H^{2*} \rightarrow H^{1*} \wedge H^{1*}$ be the dual map.  Then:

\begin{equation*}
hol(G) := \frac{\widehat{Lie(H^{1*})}}{\langle Im\ \partial \rangle}
\end{equation*}
where $\widehat{Lie(H^{1*})}$ is the free, graded Lie algebra generated by $H^{1*}$ and $\langle Im\partial \rangle$ is the ideal in $\widehat{Lie(H^{1*})}$ generated by $Im\ \partial$.  Thus we have by definition an exact sequence:

\begin{equation*}
H^{2*} \overset{\partial}{\rightarrow} H^{1*} \wedge H^{1*} \rightarrow hol(G)^2 \rightarrow 0
\end{equation*}
where $hol(G)^2$ is the degree 2 component of $hol(G)$.  Dually, $hol(G)^{2*} \cong ker\ \cup$.

In many cases, we have explicit information about $ker\ \cup$:\footnote{See \cite{Lambe}.  This theorem extends part of Theorem (8.1)' of \cite{Sullivan}.  See also \cite{Cenkl}.}

\begin{theorem}
\label{ThmSull}
Let $G$ be a group such that $H_n(G,\Q)$ is finite dimensional over $\Q$ for all $n$.  Then the following sequence is exact:
\begin{equation*}
0 \rightarrow (\GII \otimes \Q)^* \overset{d}{\longrightarrow} H^1\wedge H^1 \overset{\cup}{\longrightarrow} H^2
\end{equation*}
where $d$ is dual to the commutator map $[-,-]: \GI \wedge \GI \rightarrow \GII$.\footnote{We recall that $\GI \cong H^{1*}$ (see \cite{HilStam}).}
\end{theorem}

Thus, when this theorem applies, $hol(G)^2 \cong \GII \otimes \Q$.

Now note that $U(hol(G))$ is generated by $H^{1*} \cong \GI \cong \IK/\IK^2$, so it has the same generators as $\grK$.  Also, from the preceding remark, if Theorem \ref{ThmSull} applies then $U(hol(G))$ must have the same quadratic relations as $U(grG \otimes \Q) \cong \grK$, but has no further defining relations.  Hence it follows that:

\begin{equation*}
U(hol(G)) \cong q(\grK)
\end{equation*}
whenever the result in Theorem \ref{ThmSull} (or similar theorems) applies.

\subsubsection{Relation Between Quadraticity and Graded 1-Formality}
\label{SecExp1Form}

We are now in a position to relate quadraticity with graded 1-formality, in cases where the results of Theorem \ref{ThmSull} or similar theorems hold.  Recall that quadraticity means that $\grK \cong q(\grK)$. Pulling together the various isomorphisms identified above, this is equivalent to $U(\Pri(\grK)) \cong U(hol(G))$ as associative algebras.

Both sides of the latter isomorphism have generators (as associative algebras) of degree 1, all of which are primitive elements, and the isomorphism (when it exists) comes from identifying generators in the respective algebras.  Hence the isomorphism respects the Hopf algebra structure.
It follows that the isomorphism descends to the respective Lie algebras of primitives, i.e. quadraticity is equivalent to $\Pri(\grK) \cong hol(G)$.

Recalling the result of Quillen mentioned earlier, i.e. $\Pri(\grK) \cong grM_G$, we find that quadraticity is equivalent to $grM_G \cong hol(G)$, which is just graded 1-formality.

In summary, we have shown that (when the results of Theorem \ref{ThmSull} or similar theorems hold), quadraticity is equivalent to graded 1-formality.\footnote{Alternatively, whenever the result in footnote \ref{ABCKT} applies, quadraticity is equivalent to graded 1-formality, whether or not $hol(G) \cong \Pri(q(\grK))$.}

\subsection{Universal Finite Type Invariants}

In this subsection, we give some background on the theory of finite type invariants, and explain how quadraticity is equivalent to the existence of a universal finite type invariant.  The existence of such an invariant is perhaps the central question of the theory of finite type invariants for any group.  Much of the material in this subsection comes from \cite{DBN-WKO}, to which the reader is referred for more on finite type theory.

A universal finite type invariant (`UFTI') for a group $G$ is just a not necessarily homomorphic (but still quadratic) expansion.  More precisely, an UFTI is a filtered linear map $Z: K \rightarrow \widehat{q(\grK)}$ whose associated graded is a left inverse for $\mu$, i.e. satisfies $grZ \circ \mu = id$ (in all degrees).  This is described by the following diagram:

\[
\xy
(-15,0)*+{K} = "1";
(15,0)*+{\widehat{q(\grK)}} = "2";
(0,-20)*+{\widehat{\grK}} = "3";
{\ar@{->}^{Z} "1"; "2"};
{\ar@{->}^{gr Z} "3"; "2"};
{\ar@{->>}^{\mu} (15,-4)*{}; (5,-17)*{}}
\endxy
\]

By an argument analogous to that given in Subsection \ref{GroupExpSubsection} the existence of a not necessarily homomorphic expansion implies quadraticity.  In fact, we will show that the existence of an UFTI is equivalent to quadraticity.  First, though, we need some background on finite type invariants.

We begin with the following sequence, which is clearly exact:

\begin{equation*}
0 \longrightarrow \IK^p/\IK^{p+1} \longrightarrow K / \IK^{p+1} \longrightarrow K / \IK^{p} \longrightarrow 0
\end{equation*}

We combine this with the surjection $q(\grK)^p \twoheadrightarrow \IK^p/\IK^{p+1}$ to get:

\[
\xy
(-20,0)*+{0} = "1"; (0,0)*+{\IK^p/\IK^{p+1}} = "2";
(25,0)*+{K / \IK^{p+1}} = "3"; (45,0)*+{K / \IK^{p}} = "4";
(60,0)*+{0} = "5"; (0,15)*+{q(\grK)^p} = "6";
{\ar "1"; "2"}; {\ar "2"; "3"}; {\ar "3"; "4"}; {\ar "4"; "5"}; {\ar@{->>}^{\mu} "6"; "2"}
\endxy
\]

Now we take linear duals - since the spaces are all vector spaces over $\Q$, the sequence remains exact:

\[
\xy
(40,0)*+{0} = "5"; (20,0)*++{(\IK^p/\IK^{p+1})^*} = "4";
(-5,0)*+{(K / \IK^{p+1})^*} = "3"; (-30,0)*+{(K / \IK^{p})^*} = "2";
(-50,0)*+{0} = "1"; (20,-15)*+{q(\grK)^{p*}} = "6";
{\ar "1"; "2"}; {\ar "2"; "3"}; {\ar^{\pi} "3"; "4"}; {\ar "4"; "5"}; {\ar@{>->}^{\mu^*} "4"; "6"}; {\ar@{.>}_{\pi'} "3"; "6"}
\endxy
\]

The space $q(\grK)^{p*}$ is referred to as the space of `weight systems of degree $p$'.  The space $(K / \IK^{p+1})^*$ is referred to as the space of `finite type invariants' (of type $p$).  A central question in the theory of finite type invariants of the group $G$ is whether every weight system of degree $p$ is induced by an invariant of type $p$ - i.e. whether the map $\pi'$ is surjective.

It is clear that surjectivity of $\pi'$ is equivalent to surjectivity of $\mu^*$, which in turn is equivalent to $\mu^*$ being an isomorphism (since $\mu^*$ is already injective).  Thus surjectivity of $\pi'$ is equivalent to quadraticity.  We will now show that quadraticity implies the existence of an UFTI, so that all three concepts are equivalent.

Indeed, suppose $K$ is quadratic.  Then for $p\in \N$, let $\{x_i\}_{i\in J_p}$ be a basis for $[q(\grK)]_p$, and let $\{x_i^*\}_{i\in J_p}$ be a dual basis pulled back to $(\grK)^*$ via $(\mu^*)^{-1}$.  Then, for $x_i^* \in (\IK^p/\IK^{p+1})^*$, let $v_i\in (K/\IK^{p+1})^*$ s.t. $\pi(v_i)=x_i^*$.  Finally, let $\rho_p: K \twoheadrightarrow K/\IK^{p+1}$ be the projection.  Now define:

\begin{equation*}
Z(-):= \sum_p \sum_{i\in J_p} x_i. (v_i\circ \rho_p) (-)
\end{equation*}

We need to check that $grZ\circ \mu = id$.  So let $x_i, i\in J_p$ be one of our chosen basis elements of $[q(\grK)]_p$, and let $\xi_i \in \IK^p$ be such that $\rho_p(\xi_i)=\mu(x_i) \in \IK^p/\IK^{p+1}$.  Then
\begin{align*}
grZ\circ \mu(x_i) & = \sum_q \sum_{j\in J_q} x_j.(v_j\circ \rho_q)(\xi_i) \mod [q(\grK)]_{\geq p+1} \\
& = \sum_{q\geq p} \sum_{j \in J_q} x_j v_j(\rho_q(\xi_i)) \mod [q(\grK)]_{\geq p+1} \\
& = \sum_{j\in J_p} x_j x_j^*(\mu(x_i)) = \sum_{j\in J_p} x_j \delta_{ij} = x_i
\end{align*}
as required.

\subsection{The Pure Virtual Braid Groups}

\subsubsection{Definition}

Recall that the braid group $\B$ is generated by the symbols $\{\si: i=1, \dots, (n-1)\}$, each corresponding to a braid with $n$ strands with the strand in position $i$ crossing over the adjacent strand to the `right' (i.e. the strand in position $(i+1)$), in a `positive' fashion:

\[
\xy
\vtwist~{(-3,3)}{(3,3)}{(-3,-3)}{(3,-3)};
(-12,-3)*{}; (-12,3)*{} **\dir{-};
(12,-3)*{}; (12,3)*{} **\dir{-};
(-7,0)*{\dots}; (7,0)*{\dots}; (-3,-5)*{i}; (3,-5)*{i \negthinspace + \negthinspace 1}
\endxy
\]
(where by convention all strands are oriented upwards).

The relations in $\B$ are the well-known Reidemeister III move and obvious commutativities:

\begin{align*}
\si\sii\si & = \sii\si\sii \\
\si\sj & = \sj\si \quad \quad \text{ for } |i-j|>1
\end{align*}

The (non-pure) \emph{virtual} braid group $\vB$ is generated by the symbols $\{\si: i=1, \dots, (n-1)\}$, referred to as (ordinary) crossings, and $\{s_i: i=1, \dots, (n-1)\}$, referred to as virtual crossings.  The $\{\si\}$ are subject to the same relations (Reidemeister III and commutativities) as above, while the $\{s_i\}$ are subject to the symmetric group relations:

\begin{align*}
s_i s_{i+1} s_i & = s_{i+1} s_i s_{i+1} \\
s_i s_j & = s_j s_i \quad \quad \text{ for } |i-j|>1 \\
s_i^2 &= 1
\end{align*}
which are like the braid group relations but also include the involutory relations $s_i^2=1$.

\newpage

There are also certain `mixed' relations between the $\si$ and the $s_i$:

\begin{align*}
s_i\sii^{\pm 1} s_i & = s_{i+1} \si^{\pm 1} s_{i+1} \\
s_i\sj & = \sj s_i \quad \quad \text{ for } |i-j|>1
\end{align*}

In pictures, ordinary crossings are depicted in the usual way, while virtual crossings $s_i$ are depicted as

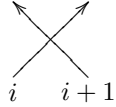
\[
\xy
(0,0)*{} = "1";
(10,0)*{} = "2";
(0,-10)*{} = "3";
(10,-10)*{} = "4";
(0,-12)*{i};
(10,-12)*{i+1};
{\ar@{->} "3"; "2"};
{\ar@{->} "4"; "1"};
\endxy
\]

Thus the mixed relations may be depicted as:

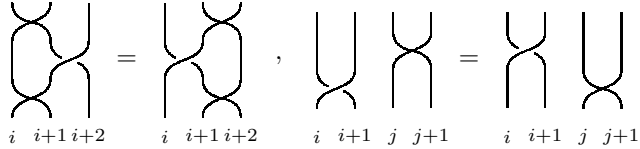
\[
\xy
0;/r.12pc/:
(-45,0)*{\xy
(-20,0)*{\xy
(-5,10)*{\xy (-5,-5)*{}; (5,5)*{} **\crv{(-5,1)&(5,-1)}
\POS?(.5)*{}="x"; (-5,5)*{}; "x" **\crv{(-5,1)}; "x"; (5,-5)*{}
**\crv{(5,-1)} \endxy};
(5,0)*{\xy (-5,-5)*{}; (5,5)*{} **\crv{(-5,1)&(5,-1)}
\POS?(.5)*{\hole}="x"; (-5,5)*{}; "x" **\crv{(-5,1)}; "x"; (5,-5)*{}
**\crv{(5,-1)} \endxy};
(-5,-10)*{\xy (-5,-5)*{}; (5,5)*{} **\crv{(-5,1)&(5,-1)}
\POS?(.5)*{}="x"; (-5,5)*{}; "x" **\crv{(-5,1)}; "x"; (5,-5)*{}
**\crv{(5,-1)} \endxy};
(10,-15)*{}; (10,-5)*{} **\dir{-}; (-10,-5)*{}; (-10,5)*{}
**\dir{-}; (10,5)*{}; (10,15)*{} **\dir{-}
\endxy};
(0,0)*{=};
(20,0)*{\xy
(5,10)*{\xy (-5,-5)*{}; (5,5)*{} **\crv{(-5,1)&(5,-1)}
\POS?(.5)*{}="x"; (-5,5)*{}; "x" **\crv{(-5,1)}; "x"; (5,-5)*{}
**\crv{(5,-1)} \endxy};
(-5,0)*{\xy (-5,-5)*{}; (5,5)*{} **\crv{(-5,1)&(5,-1)}
\POS?(.5)*{\hole}="x"; (-5,5)*{}; "x" **\crv{(-5,1)}; "x"; (5,-5)*{}
**\crv{(5,-1)} \endxy};
(5,-10)*{\xy (-5,-5)*{}; (5,5)*{} **\crv{(-5,1)&(5,-1)}
\POS?(.5)*{}="x"; (-5,5)*{}; "x" **\crv{(-5,1)}; "x"; (5,-5)*{}
**\crv{(5,-1)} \endxy};
(-10,-15)*{}; (-10,-5)*{} **\dir{-}; (10,-5)*{}; (10,5)*{}
**\dir{-}; (-10,5)*{}; (-10,15)*{} **\dir{-};
\endxy}
\endxy};
(-5,0)*{,};
(-75,-20)*{\scriptstyle{i}};(-65,-20)*{\scriptstyle{i+1}};(-55,-20)*{\scriptstyle{i+2}};
(-35,-20)*{\scriptstyle{i}};(-25,-20)*{\scriptstyle{i+1}};(-15,-20)*{\scriptstyle{i+2}};
(5,-20)*{\scriptstyle{i}};(15,-20)*{\scriptstyle{i+1}};(25,-20)*{\scriptstyle{j}};(35,-20)*{\scriptstyle{j+1}};
(55,-20)*{\scriptstyle{i}};(65,-20)*{\scriptstyle{i+1}};(75,-20)*{\scriptstyle{j}};(85,-20)*{\scriptstyle{j+1}};
(45,0)*{\xy
(-25,0)*{\xy
(15,0)*{\xy (-5,-5)*{}; (5,5)*{} **\crv{(-5,1)&(5,-1)}
\POS?(.5)*{}="x"; (-5,5)*{}; "x" **\crv{(-5,1)}; "x"; (5,-5)*{}
**\crv{(5,-1)} \endxy};
(-5,-10)*{\xy (-5,-5)*{}; (5,5)*{} **\crv{(-5,1)&(5,-1)}
\POS?(.5)*{\hole}="x"; (-5,5)*{}; "x" **\crv{(-5,1)}; "x"; (5,-5)*{}
**\crv{(5,-1)} \endxy};
(10,-15)*{}; (10,-5)*{} **\dir{-}; (-10,-5)*{}; (-10,10)*{}
**\dir{-}; (0,-5)*{}; (0,10)*{}
**\dir{-}; (10,5)*{}; (10,10)*{} **\dir{-}; (20,5)*{}; (20,10)*{}
**\dir{-}; (20,-5)*{}; (20,-15)*{}
**\dir{-}
\endxy};
(0,0)*{=};
(25,0)*{\xy
(-5,0)*{\xy (-5,-5)*{}; (5,5)*{} **\crv{(-5,1)&(5,-1)}
\POS?(.5)*{\hole}="x"; (-5,5)*{}; "x" **\crv{(-5,1)}; "x"; (5,-5)*{}
**\crv{(5,-1)} \endxy};
(15,-10)*{\xy (-5,-5)*{}; (5,5)*{} **\crv{(-5,1)&(5,-1)}
\POS?(.5)*{}="x"; (-5,5)*{}; "x" **\crv{(-5,1)}; "x"; (5,-5)*{}
**\crv{(5,-1)} \endxy};
(10,-15)*{}; (10,-15)*{} **\dir{-}; (-10,-15)*{}; (-10,-5)*{}
**\dir{-}; (-10,5)*{}; (-10,10)*{}
**\dir{-};(0,5)*{}; (0,10)*{}
**\dir{-}; (0,-15)*{}; (0,-5)*{}
**\dir{-};(10,-5)*{}; (10,10)*{} **\dir{-}; (20,-5)*{}; (20,10)*{}
**\dir{-};
\endxy}
\endxy}
\endxy
\]
with similar pictures where the ordinary crossings are negative.

One can show that the subgroup generated by the $\{\si\}$ is just the braid group $\B$, and the subgroup generated by the $\{s_i\}$ is the symmetric group $S_n$.  There is also a surjection $\vB \thra S_n$ onto the symmetric group $S_n$ given by sending each $\si$ and $s_i$ to $s_i$.  The kernel of this map is referred to as the pure virtual braid group, $\PV$.  We very briefly sketch how to obtain a presentation for $\PV$ using the Reidemeister-Schreier method (a more thorough explanation of this method as applied to $\PV$, with proofs, is given in \cite{Bard}, and a general overview of the method appears in \cite{MKS}, Section 2.3).

We first exhibit a set of right coset representatives for $\PV$ in $\vB$.  Let us define $\mkl:= s_{k-1} s_{k-2} \dots s_l$ for $k>l$ and $\mkl=1$ otherwise.  Then

\begin{equation*}
\Lambda_n:= \lbrace \prod_{k=2}^n m_{k,j_k} \text{   s.t.   } 1 \leq j_k \leq k \rbrace
 \end{equation*}
coincides with $S_n$ and gives a `Schreier' set of coset representatives, i.e. if $\lambda \in \Lambda_n$ then also $\lambda'\in \Lambda_n$ whenever $\lambda'$ is an initial segment of the word $\lambda$.

Note that the $\mkl$ (for $k>l$) can be depicted as virtual braids which move the strand which is initially in position $k$ so that it is immediately to the left of the strand in position $l$:

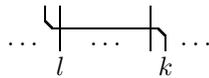
\[
\xy
(-10,-8)*{}; (-10,-2)*{} **\dir{-};
(2,-8)*{}; (2,-2)*{} **\dir{-};
(4,-8)*{}; (4,-6)*{} **\dir{-};
(4,-6)*{}; (3,-5)*{} **\crv{(3.5,-5.5)};
(3,-5)*{}; (-11,-5)*{} **\dir{-};
(-11,-5)*{}; (-12,-4)*{} **\crv{(-11.5,-4.5)};
(-12,-4)*{}; (-12,-2)*{} **\dir{-};
(-4,-7)*{\dots};
(-15,-7)*{\dots};
(8,-7)*{\dots};
(-10,-10)*{l}; (4,-10)*{k}
\endxy
\]

To apply the Reidemeister-Schreier method, we now define a map $\vB \ra \Lambda_n$ sending $x \in \vB$ to its coset representative $\bx \in \Lambda_n$.  Then we have $x \bx^{-1} \in \PV$.  By the Reidemeister-Schreier theorem, $\PV$ is generated by

\begin{equation*}
\big\{ \Sla := \lambda a \ola \text{   s.t.   } \lambda \in \Lambda_n \text{  and  } a \in \{\sigma_i, s_i\} \big\}
\end{equation*}

One can show that the above generating set reduces to precisely the set $\{R_{ij}\}_{1 \leq i \ne j \leq n}$ (or their inverses, or the identity), where (for $i<j$):

\begin{align*}
R_{ij} &= m_{j,i+1} (\si s_i) m_{j,i+1}^{-1}, \text{     and} \\
R_{ji} &= m_{j,i+1} (s_i \si) m_{j,i+1}^{-1}
\end{align*}

Each such generator consists of a sequence of virtual crossings, followed by a single ordinary crossing, followed by additional virtual crossings which restore all strands to their original position.  A typical element $R_{ij}$ may be depicted:

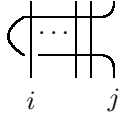
\[
\xy
(0,0)*{}; (-10,0)*{} **\dir{-};
(0,0)*{}; (2,2)*{} **\crv{(2,0)};
(-10,-5)*{}; (-10,0)*{} **\crv{(-14,-2.5)};
(-9,-8)*{}; (-9,2) **\dir{-}  \POS?(.3)*{\hole} = "x";
(-10,-5)*{}; "x" **\dir{-};
"x"; (0,-5)*{} **\dir{-};
(0,-5)*{}; (2,-7)*{} **\crv{(2,-5)};
(2,-7)*{}; (2,-8)*{} **\dir{-};
(-1,-8)*{}; (-1,2)*{} **\dir{-};
(-3,-8)*{}; (-3,2)*{} **\dir{-};
(-6,-2)*{\dots};
(-9,-11)*{i}; (2,-11)*{j}
\endxy
\]
(again with all strands oriented upwards).

As explained in \cite{DBN-WKO}, one effect of the relations satisfied by the $\{s_i\}$ (mixed and unmixed) is that each generator $R_{ij}$ can be thought of as an ordinary (positive) crossing of strand $i$ over strand $j$, with any choice of virtual moves to get the strands `into position' for the ordinary crossing being equally valid.  From this point of view, the above generator $R_{ij}$ could have been depicted:

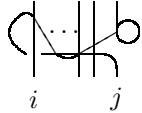
\[
\xy
(2,-2)*{}; (-3,-5)*{} **\dir{-};
(-3,-5)*{}; (-6,-5)*{} **\crv{(-4.5,-6)};
(-6,-5)*{}; (-9,0)*{} **\dir{-};
(2,-2)*{}; (2,2)*{} **\dir{-};
(2,-2)*{}; (5,-2)*{} **\crv{(2,0) & (5,0)};
(2,-2)*{}; (5,-2)*{} **\crv{(2,-4) & (5,-4)};
(-11,0)*{}; (-9,0)*{} **\crv{(-10,1)};
(-10,-5)*{}; (-11,0)*{} **\crv{(-14,-2.5)};
(-9,-8)*{}; (-9,2) **\dir{-}  \POS?(.3)*{\hole} = "x";
(-10,-5)*{}; "x" **\dir{-};
"x"; (0,-5)*{} **\dir{-};
(0,-5)*{}; (2,-7)*{} **\crv{(2,-5)};
(2,-7)*{}; (2,-8)*{} **\dir{-};
(-1,-8)*{}; (-1,2)*{} **\dir{-};
(-3,-8)*{}; (-3,2)*{} **\dir{-};
(-5,-2)*{\dots};
(-9,-11)*{i}; (2,-11)*{j}
\endxy
\]
For this reason, it is rather superfluous when drawing pictures to show any needed virtual crossings, and one usually omits them (as we will do going forward).

The procedure used to obtain the relations in $\PV$ is described in $\cite{Bard}$, and produces the following:

\begin{align*}
R_{ij}R_{ik}R_{jk} & = R_{jk}R_{ik}R_{ij} \text{     (Reid. III)}\\
R_{ij}R_{kl} & = R_{kl}R_{ij} \text{     Commutativities}
\end{align*}

The following pictures give the graphical representation of these relations.  One recognizes the Reidemeister III moves and commutativities which hold in the ordinary braid group (although the interpretation of the pictures in terms of algebraic generators is different).

\[
\xy
0;/r.12pc/:
(-45,0)*{\xy
(-20,0)*{\xy
(-5,10)*{\xy (-5,-5)*{}; (5,5)*{} **\crv{(-5,1)&(5,-1)}
\POS?(.5)*{\hole}="x"; (-5,5)*{}; "x" **\crv{(-5,1)}; "x"; (5,-5)*{}
**\crv{(5,-1)} \endxy};
(5,0)*{\xy (-5,-5)*{}; (5,5)*{} **\crv{(-5,1)&(5,-1)}
\POS?(.5)*{\hole}="x"; (-5,5)*{}; "x" **\crv{(-5,1)}; "x"; (5,-5)*{}
**\crv{(5,-1)} \endxy};
(-5,-10)*{\xy (-5,-5)*{}; (5,5)*{} **\crv{(-5,1)&(5,-1)}
\POS?(.5)*{\hole}="x"; (-5,5)*{}; "x" **\crv{(-5,1)}; "x"; (5,-5)*{}
**\crv{(5,-1)} \endxy};
(10,-15)*{}; (10,-5)*{} **\dir{-}; (-10,-5)*{}; (-10,5)*{}
**\dir{-}; (10,5)*{}; (10,15)*{} **\dir{-}
\endxy};
(0,0)*{=};
(20,0)*{\xy
(5,10)*{\xy (-5,-5)*{}; (5,5)*{} **\crv{(-5,1)&(5,-1)}
\POS?(.5)*{\hole}="x"; (-5,5)*{}; "x" **\crv{(-5,1)}; "x"; (5,-5)*{}
**\crv{(5,-1)} \endxy};
(-5,0)*{\xy (-5,-5)*{}; (5,5)*{} **\crv{(-5,1)&(5,-1)}
\POS?(.5)*{\hole}="x"; (-5,5)*{}; "x" **\crv{(-5,1)}; "x"; (5,-5)*{}
**\crv{(5,-1)} \endxy};
(5,-10)*{\xy (-5,-5)*{}; (5,5)*{} **\crv{(-5,1)&(5,-1)}
\POS?(.5)*{\hole}="x"; (-5,5)*{}; "x" **\crv{(-5,1)}; "x"; (5,-5)*{}
**\crv{(5,-1)} \endxy};
(-10,-15)*{}; (-10,-5)*{} **\dir{-}; (10,-5)*{}; (10,5)*{}
**\dir{-}; (-10,5)*{}; (-10,15)*{} **\dir{-};
\endxy}
\endxy};
(-5,0)*{,};
(-75,-20)*{i};(-65,-20)*{j};(-55,-20)*{k};
(-35,-20)*{i};(-25,-20)*{j};(-15,-20)*{k};
(5,-20)*{i};(15,-20)*{j};(25,-20)*{k};(35,-20)*{l};
(55,-20)*{i};(65,-20)*{j};(75,-20)*{k};(85,-20)*{l};
(45,0)*{\xy
(-25,0)*{\xy
(15,0)*{\xy (-5,-5)*{}; (5,5)*{} **\crv{(-5,1)&(5,-1)}
\POS?(.5)*{\hole}="x"; (-5,5)*{}; "x" **\crv{(-5,1)}; "x"; (5,-5)*{}
**\crv{(5,-1)} \endxy};
(-5,-10)*{\xy (-5,-5)*{}; (5,5)*{} **\crv{(-5,1)&(5,-1)}
\POS?(.5)*{\hole}="x"; (-5,5)*{}; "x" **\crv{(-5,1)}; "x"; (5,-5)*{}
**\crv{(5,-1)} \endxy};
(10,-15)*{}; (10,-5)*{} **\dir{-}; (-10,-5)*{}; (-10,10)*{}
**\dir{-}; (0,-5)*{}; (0,10)*{}
**\dir{-}; (10,5)*{}; (10,10)*{} **\dir{-}; (20,5)*{}; (20,10)*{}
**\dir{-}; (20,-5)*{}; (20,-15)*{}
**\dir{-}
\endxy};
(0,0)*{=};
(25,0)*{\xy
(-5,0)*{\xy (-5,-5)*{}; (5,5)*{} **\crv{(-5,1)&(5,-1)}
\POS?(.5)*{\hole}="x"; (-5,5)*{}; "x" **\crv{(-5,1)}; "x"; (5,-5)*{}
**\crv{(5,-1)} \endxy};
(15,-10)*{\xy (-5,-5)*{}; (5,5)*{} **\crv{(-5,1)&(5,-1)}
\POS?(.5)*{\hole}="x"; (-5,5)*{}; "x" **\crv{(-5,1)}; "x"; (5,-5)*{}
**\crv{(5,-1)} \endxy};
(10,-15)*{}; (10,-15)*{} **\dir{-}; (-10,-15)*{}; (-10,-5)*{}
**\dir{-}; (-10,5)*{}; (-10,10)*{}
**\dir{-};(0,5)*{}; (0,10)*{}
**\dir{-}; (0,-15)*{}; (0,-5)*{}
**\dir{-};(10,-5)*{}; (10,10)*{} **\dir{-}; (20,-5)*{}; (20,10)*{}
**\dir{-};
\endxy}
\endxy}
\endxy
\]

Thus, as mentioned earlier, pure virtual braids are given by the following presentation:

\begin{equation*}
\PV := \langle R_{ij} \rangle_{1\leq i \ne j \leq n} / \{\text{Reid. III, Commutativities} \}
\end{equation*}

As indicated previously, the presentation for $\PfB$ is the same, except that we also impose the relations $R_{ij}R_{ji}=1$.  One could instead retain only those generators $R_{ij}$ with $i<j$ (and keep the $\PV$ relations which involve only such generators, but leave out $R_{ij}R_{ji}=1$).  This gives an isomorphic group known as the group of pure descending braids.  The `descending' refers to the fact that if in all crossings $R_{ij}$ we impose $i<j$, this corresponds to the fact that strand $i$ always passes over strand $j$ so that the `height' of strands decreases as the labels increase.

We end this overview of $\PV$ with a brief explanation of its relation with the pure (ordinary) braid group $\PB$.  As discussed above, we may view the ordinary braid group $\B$ as the subgroup of $\vB$ generated by the $\{\si\}$ and the symmetric group $S_n$ as the subgroup generated by the $\{s_i\}$; then $\PB$ is by definition the kernel of the group homomorphism $\B \ra S_n$ which sends each $\si$ to $s_i$.  It is generated by the symbols $\{\Aij: 1\leq i<j \leq n\}$ (for the relations, see e.g. \cite{MarMc}).  As pointed out in \cite{BEER} (Section 4.3), there is a homomorphism $\Psi_n: \PB \ra \PV$ defined by:

\begin{equation*}
\Aij \mapsto R_{j-1,j} \dots R_{i+1,j} R_{ij} R_{ji} (R_{j-1,j} \dots R_{i+1,j})^{-1}
\end{equation*}

As noted in the Introduction, $\PB$ is also quadratic (as follows from results of \cite{Kohno}, as well as \cite{Hutchings}), and the associated graded of its group ring $gr\Q\PB=\pb$ is generated by symbols $\{\aij: 1\leq i \ne j \leq n; \aij = a_{ji}\}$, subject to the relations

\begin{align*}
[\aij, \aik + \ajk] & = 0 \\
[\aij,\akl] &= 0
\end{align*}

It is fairly easy to see that the induced map $\psi_n: \pb \ra \pv$ is given by $\aij \mapsto \rij + \rji$.  That $\psi_n$ is an algebra homomorphism is also verified directly in \cite{Polyak}, in which it is conjectured that $\psi_n$ is injective.  This conjecture, and the related claim that $\Psi_n$ is injective, were proved in \cite{BEER} (Proposition 4.8).

\subsubsection{Filtration, Associated Graded and Quadratic Approximation}
\label{SubsecPvBQuadApp}

We let $F$ denote the free (unital) algebra with the same generators $\{R_{ij}\}$ as $\PV$; and denote by $\IF$ and $\IK$ the augmentation ideals of $F$ and $K:=\Q \PV$, respectively, given by the kernels of the augmentation maps which send $R_{ij} \mapsto 1$.  Both $\IF$ and $\IK$ are generated by the sets $\{\bR_{ij}:= (R_{ij}-1) \}$.  The $\bR_{ij}$ may be depicted as:

\[
\xy
0;/r.12pc/:
(-20,0)*{\bR_{ij}:=};
(0,0)*{\xy (-5,-5)*{}; (5,5)*{} **\crv{(-5,1)&(5,-1)}
\POS?(.5)*{\bigcirc}="x"; (-5,5)*{}; "x" **\crv{(-5,1)}; "x"; (5,-5)*{}
**\crv{(5,-1)};
(-5,-8)*{i}; (5,-8)*{j}
\endxy};
(15,0)*{:=};
(30,0)*{
\xy (-5,-5)*{}; (5,5)*{} **\crv{(-5,1)&(5,-1)}
\POS?(.5)*{\hole}="x"; (-5,5)*{}; "x" **\crv{(-5,1)}; "x"; (5,-5)*{}
**\crv{(5,-1)};
(-5,-8)*{i}; (5,-8)*{j}
\endxy};
(40,0)*{-};
(50,0)*{\xy (-5,-5)*{}; (5,5)*{} **\crv{(-5,1)&(5,-1)}
\POS?(.5)*{}="x"; (-5,5)*{}; "x" **\crv{(-5,1)}; "x"; (5,-5)*{}
**\crv{(5,-1)};
(-5,-8)*{i}; (5,-8)*{j}
\endxy}
\endxy
\]

Re-expressed in terms of the $\bR_{ij}$, the Reidemeister III relation gives a `topological 8-term' relation:

\[
\xy
0;/r.12pc/:
(-20,0)*{\xy
(-5,10)*{\xy (-5,-5)*{}; (5,5)*{} **\crv{(-5,1)&(5,-1)}
\POS?(.5)*{\hole}="x"; (-5,5)*{}; "x" **\crv{(-5,1)}; "x"; (5,-5)*{}
**\crv{(5,-1)} \endxy};
(5,0)*{\xy (-5,-5)*{}; (5,5)*{} **\crv{(-5,1)&(5,-1)}
\POS?(.5)*{\hole}="x"; (-5,5)*{}; "x" **\crv{(-5,1)}; "x"; (5,-5)*{}
**\crv{(5,-1)} \endxy};
(-5,-10)*{\xy (-5,-5)*{}; (5,5)*{} **\crv{(-5,1)&(5,-1)}
\POS?(.5)*{\hole}="x"; (-5,5)*{}; "x" **\crv{(-5,1)}; "x"; (5,-5)*{}
**\crv{(5,-1)} \endxy};
(10,-15)*{}; (10,-5)*{} **\dir{-}; (-10,-5)*{}; (-10,5)*{}
**\dir{-}; (10,5)*{}; (10,15)*{} **\dir{-}
\endxy};
(0,0)*{-};
(20,0)*{\xy
(5,10)*{\xy (-5,-5)*{}; (5,5)*{} **\crv{(-5,1)&(5,-1)}
\POS?(.5)*{\hole}="x"; (-5,5)*{}; "x" **\crv{(-5,1)}; "x"; (5,-5)*{}
**\crv{(5,-1)} \endxy};
(-5,0)*{\xy (-5,-5)*{}; (5,5)*{} **\crv{(-5,1)&(5,-1)}
\POS?(.5)*{\hole}="x"; (-5,5)*{}; "x" **\crv{(-5,1)}; "x"; (5,-5)*{}
**\crv{(5,-1)} \endxy};
(5,-10)*{\xy (-5,-5)*{}; (5,5)*{} **\crv{(-5,1)&(5,-1)}
\POS?(.5)*{\hole}="x"; (-5,5)*{}; "x" **\crv{(-5,1)}; "x"; (5,-5)*{}
**\crv{(5,-1)} \endxy};
(-10,-15)*{}; (-10,-5)*{} **\dir{-}; (10,-5)*{}; (10,5)*{}
**\dir{-}; (-10,5)*{}; (-10,15)*{} **\dir{-};
\endxy};
(40,0)*{=}
\endxy
\]

\[
\xy
0;/r.12pc/:
(-20,0)*{\xy
(-5,10)*{\xy (-5,-5)*{}; (5,5)*{} **\crv{(-5,1)&(5,-1)}
\POS?(.5)*{\bigcirc}="x"; (-5,5)*{}; "x" **\crv{(-5,1)}; "x"; (5,-5)*{}
**\crv{(5,-1)} \endxy};
(5,0)*{\xy (-5,-5)*{}; (5,5)*{} **\crv{(-5,1)&(5,-1)}
\POS?(.5)*{\bigcirc}="x"; (-5,5)*{}; "x" **\crv{(-5,1)}; "x"; (5,-5)*{}
**\crv{(5,-1)} \endxy};
(-5,-10)*{\xy (-5,-5)*{}; (5,5)*{} **\crv{(-5,1)&(5,-1)}
\POS?(.5)*{\bigcirc}="x"; (-5,5)*{}; "x" **\crv{(-5,1)}; "x"; (5,-5)*{}
**\crv{(5,-1)} \endxy};
(10,-15)*{}; (10,-5)*{} **\dir{-}; (-10,-5)*{}; (-10,5)*{}
**\dir{-}; (10,5)*{}; (10,15)*{} **\dir{-}
\endxy};
(0,0)*{-};
(20,0)*{\xy
(5,10)*{\xy (-5,-5)*{}; (5,5)*{} **\crv{(-5,1)&(5,-1)}
\POS?(.5)*{\bigcirc}="x"; (-5,5)*{}; "x" **\crv{(-5,1)}; "x"; (5,-5)*{}
**\crv{(5,-1)} \endxy};
(-5,0)*{\xy (-5,-5)*{}; (5,5)*{} **\crv{(-5,1)&(5,-1)}
\POS?(.5)*{\bigcirc}="x"; (-5,5)*{}; "x" **\crv{(-5,1)}; "x"; (5,-5)*{}
**\crv{(5,-1)} \endxy};
(5,-10)*{\xy (-5,-5)*{}; (5,5)*{} **\crv{(-5,1)&(5,-1)}
\POS?(.5)*{\bigcirc}="x"; (-5,5)*{}; "x" **\crv{(-5,1)}; "x"; (5,-5)*{}
**\crv{(5,-1)} \endxy};
(-10,-15)*{}; (-10,-5)*{} **\dir{-}; (10,-5)*{}; (10,5)*{}
**\dir{-}; (-10,5)*{}; (-10,15)*{} **\dir{-};
\endxy};
\endxy
\]

\[
\xy
0;/r.12pc/:
(-50,0)*{\xy
(-45,0)*{+};
(-30,0)*{\xy
(-5,10)*{\xy (-5,-5)*{}; (5,5)*{} **\crv{(-5,1)&(5,-1)}
\POS?(.5)*{}="x"; (-5,5)*{}; "x" **\crv{(-5,1)}; "x"; (5,-5)*{}
**\crv{(5,-1)} \endxy};
(5,0)*{\xy (-5,-5)*{}; (5,5)*{} **\crv{(-5,1)&(5,-1)}
\POS?(.5)*{\bigcirc}="x"; (-5,5)*{}; "x" **\crv{(-5,1)}; "x"; (5,-5)*{}
**\crv{(5,-1)} \endxy};
(-5,-10)*{\xy (-5,-5)*{}; (5,5)*{} **\crv{(-5,1)&(5,-1)}
\POS?(.5)*{\bigcirc}="x"; (-5,5)*{}; "x" **\crv{(-5,1)}; "x"; (5,-5)*{}
**\crv{(5,-1)} \endxy};
(10,-15)*{}; (10,-5)*{} **\dir{-}; (-10,-5)*{}; (-10,5)*{}
**\dir{-}; (10,5)*{}; (10,15)*{} **\dir{-}
\endxy};
(-15,0)*{+};
(0,0)*{\xy
(-5,10)*{\xy (-5,-5)*{}; (5,5)*{} **\crv{(-5,1)&(5,-1)}
\POS?(.5)*{\bigcirc}="x"; (-5,5)*{}; "x" **\crv{(-5,1)}; "x"; (5,-5)*{}
**\crv{(5,-1)} \endxy};
(5,0)*{\xy (-5,-5)*{}; (5,5)*{} **\crv{(-5,1)&(5,-1)}
\POS?(.5)*{}="x"; (-5,5)*{}; "x" **\crv{(-5,1)}; "x"; (5,-5)*{}
**\crv{(5,-1)} \endxy};
(-5,-10)*{\xy (-5,-5)*{}; (5,5)*{} **\crv{(-5,1)&(5,-1)}
\POS?(.5)*{\bigcirc}="x"; (-5,5)*{}; "x" **\crv{(-5,1)}; "x"; (5,-5)*{}
**\crv{(5,-1)} \endxy};
(10,-15)*{}; (10,-5)*{} **\dir{-}; (-10,-5)*{}; (-10,5)*{}
**\dir{-}; (10,5)*{}; (10,15)*{} **\dir{-}
\endxy};
(15,0)*{+};
(30,0)*{\xy
(-5,10)*{\xy (-5,-5)*{}; (5,5)*{} **\crv{(-5,1)&(5,-1)}
\POS?(.5)*{\bigcirc}="x"; (-5,5)*{}; "x" **\crv{(-5,1)}; "x"; (5,-5)*{}
**\crv{(5,-1)} \endxy};
(5,0)*{\xy (-5,-5)*{}; (5,5)*{} **\crv{(-5,1)&(5,-1)}
\POS?(.5)*{\bigcirc}="x"; (-5,5)*{}; "x" **\crv{(-5,1)}; "x"; (5,-5)*{}
**\crv{(5,-1)} \endxy};
(-5,-10)*{\xy (-5,-5)*{}; (5,5)*{} **\crv{(-5,1)&(5,-1)}
\POS?(.5)*{}="x"; (-5,5)*{}; "x" **\crv{(-5,1)}; "x"; (5,-5)*{}
**\crv{(5,-1)} \endxy};
(10,-15)*{}; (10,-5)*{} **\dir{-}; (-10,-5)*{}; (-10,5)*{}
**\dir{-}; (10,5)*{}; (10,15)*{} **\dir{-}
\endxy}
\endxy};
(45,0)*{\xy
(-45,0)*{-};
(-30,0)*{\xy
(5,10)*{\xy (-5,-5)*{}; (5,5)*{} **\crv{(-5,1)&(5,-1)}
\POS?(.5)*{}="x"; (-5,5)*{}; "x" **\crv{(-5,1)}; "x"; (5,-5)*{}
**\crv{(5,-1)} \endxy};
(-5,0)*{\xy (-5,-5)*{}; (5,5)*{} **\crv{(-5,1)&(5,-1)}
\POS?(.5)*{\bigcirc}="x"; (-5,5)*{}; "x" **\crv{(-5,1)}; "x"; (5,-5)*{}
**\crv{(5,-1)} \endxy};
(5,-10)*{\xy (-5,-5)*{}; (5,5)*{} **\crv{(-5,1)&(5,-1)}
\POS?(.5)*{\bigcirc}="x"; (-5,5)*{}; "x" **\crv{(-5,1)}; "x"; (5,-5)*{}
**\crv{(5,-1)} \endxy};
(-10,-15)*{}; (-10,-5)*{} **\dir{-}; (10,-5)*{}; (10,5)*{}
**\dir{-}; (-10,5)*{}; (-10,15)*{} **\dir{-};
\endxy};
(-15,0)*{-};
(0,0)*{\xy
(5,10)*{\xy (-5,-5)*{}; (5,5)*{} **\crv{(-5,1)&(5,-1)}
\POS?(.5)*{\bigcirc}="x"; (-5,5)*{}; "x" **\crv{(-5,1)}; "x"; (5,-5)*{}
**\crv{(5,-1)} \endxy};
(-5,0)*{\xy (-5,-5)*{}; (5,5)*{} **\crv{(-5,1)&(5,-1)}
\POS?(.5)*{}="x"; (-5,5)*{}; "x" **\crv{(-5,1)}; "x"; (5,-5)*{}
**\crv{(5,-1)} \endxy};
(5,-10)*{\xy (-5,-5)*{}; (5,5)*{} **\crv{(-5,1)&(5,-1)}
\POS?(.5)*{\bigcirc}="x"; (-5,5)*{}; "x" **\crv{(-5,1)}; "x"; (5,-5)*{}
**\crv{(5,-1)} \endxy};
(-10,-15)*{}; (-10,-5)*{} **\dir{-}; (10,-5)*{}; (10,5)*{}
**\dir{-}; (-10,5)*{}; (-10,15)*{} **\dir{-};
\endxy};
(15,0)*{-};
(30,0)*{\xy
(5,10)*{\xy (-5,-5)*{}; (5,5)*{} **\crv{(-5,1)&(5,-1)}
\POS?(.5)*{\bigcirc}="x"; (-5,5)*{}; "x" **\crv{(-5,1)}; "x"; (5,-5)*{}
**\crv{(5,-1)} \endxy};
(-5,0)*{\xy (-5,-5)*{}; (5,5)*{} **\crv{(-5,1)&(5,-1)}
\POS?(.5)*{\bigcirc}="x"; (-5,5)*{}; "x" **\crv{(-5,1)}; "x"; (5,-5)*{}
**\crv{(5,-1)} \endxy};
(5,-10)*{\xy (-5,-5)*{}; (5,5)*{} **\crv{(-5,1)&(5,-1)}
\POS?(.5)*{}="x"; (-5,5)*{}; "x" **\crv{(-5,1)}; "x"; (5,-5)*{}
**\crv{(5,-1)} \endxy};
(-10,-15)*{}; (-10,-5)*{} **\dir{-}; (10,-5)*{}; (10,5)*{}
**\dir{-}; (-10,5)*{}; (-10,15)*{} **\dir{-};
\endxy}
\endxy}
\endxy
\]

In symbols, this is:

\begin{align}
Y_{ijk}:= R_{ij}R_{ik}R_{jk} & - R_{jk}R_{ik}R_{ij} \label{RelY} \\
= \bR_{ij}\bR_{ik}\bR_{jk} &- \bR_{jk}\bR_{ik}\bR_{ij} \notag \\
&+ \bR_{ij}\bR_{ik} + \bR_{ij}\bR_{jk} + \bR_{ik}\bR_{jk} - \bR_{jk}\bR_{ik} - \bR_{jk}\bR_{ij} - \bR_{ik}\bR_{ij} \notag \\
= \bR_{ij}\bR_{ik}\bR_{jk} &- \bR_{jk}\bR_{ik}\bR_{ij} \notag \\
&+ [\bR_{ij},\bR_{ik}] + [\bR_{ij},\bR_{jk}] + [\bR_{ik},\bR_{jk}] \notag
\end{align}
where square brackets denote the usual algebra commutator.

There are also commutativities:

\begin{equation}
C_{ij}^{kl} := R_{ij}R_{kl} - R_{kl}R_{ij} = [\bR_{ij}, \bR_{kl}] \label{RelC}
\end{equation}

\newpage

Both $F$ and $K$ are filtered by powers of their augmentation ideals.  We can form the respective associated graded algebras $gr F:= \oplus_{p\geq 0} \IF^p / \IF^{p+1}$ and $gr K:= \oplus_{p\geq 0} \IK^p / \IK^{p+1}$, and there are clearly surjections $\IF^p / \IF^{p+1} \twoheadrightarrow \IK^p / \IK^{p+1}$.  These factor through the quadratic approximation as follows:

\[
\xy
(-10,0)*{\IF^p / \IF^{p+1}}="1";
(15,0)*+{\frac{\IF^p / \IF^{p+1}}{\langle\partial \Rg\rangle}}="2";
(15,-15)*+{\IK^p / \IK^{p+1}}="3";
{\ar@{.>} "1"; "2"};
{\ar@{->} "1"; "3"};
{\ar@{.>} "2"; "3"}
\endxy
\]
where $\langle\partial \Rg\rangle$ is the ideal of relations in $gr F$ generated by $\partial \Rg:= \Q \{\yijk:=Y_{ijk} \mod \IF^3, \ \cijkl:= C_{ij}^{kl} \mod \IF^3\}$ (the notation $\partial \Rg$ will be explained later).

The graded algebra $\pv:= \oplus_{p\geq 0} \frac{\IF^p / \IF^{p+1}}{\langle\partial \Rg\rangle}$ is isomorphic to the quadratic approximation $q(\grK)$ to $gr K$ (see Subsection \ref{TwoCanonAssGradedAlg} below).  It is generated by $\IF/\IF^2 = \Q \{\rij:= \bR_{ij} \mod \IF^2\}$ and is subject to the relations (\ref{6Trels}) and (\ref{AlgCommutRels}):

\begin{align*}
\yijk = [r_{ij},r_{ik}] + & [r_{ij},r_{jk}] + [r_{ik},r_{jk}] = 0,
\\
& \cijkl= [r_{ij},r_{kl}] = 0 \quad \quad \quad
\end{align*}
which come from (\ref{RelY}) and (\ref{RelC}) modulo $\IF^3$.  We thus recover the presentation given in the Introduction for $\pv$.

\subsubsection{Expansions, 1-Formality and their Graded Versions For $\PV$}

It has been shown in \cite{BEER} that $\PfB$, and hence $\PV$, are not 1-formal.  However, in the same paper, Hilbert series for the integral homology of these groups were obtained, from which one can conclude in particular that the rational homology is finite dimensional (and hence of finite type).  Therefore, by the results of Subsection \ref{SecExp1Form}, the concept of quadraticity and graded 1-formality for these groups coincide.

In Theorem \ref{ThmPvBQuadratic} and its corollaries we show that $\PV$ and $\PfB$ are quadratic.  In different language, this shows that the $\PV$ and $\PfB$ have a universal finite type invariant, or are graded 1-formal.

\newpage

\section{Overview of the PVH Criterion}

\subsection{Group Theoretic Background}
\label{GroupTh}

Since the classic setting of the PVH criterion is that of group rings, we identify the attributes of group rings which we rely on and will want to see preserved in our generalized context.

We recall the following basic fact:

\begin{proposition} [See \cite{MKS}, s. 5.15]  If $G$ is given by the short exact sequence
\begin{equation*}
1 \rightarrow N \rightarrow FG \rightarrow G \rightarrow 1
\end{equation*}
where $FG$ is a free group generated by symbols $\{g_p:\ p \in P\}$ and $N$ is a normal subgroup of $FG$ generated by the set  $\{r_q:\ q \in Q\}$, then the rational group ring of $G$ is given by the exact sequence
\begin{equation*}
0 \rightarrow (N-1) \rightarrow \Q FG \rightarrow \QG \rightarrow 0
\end{equation*}
where $(N-1)$ is the two-sided ideal in $\Q FG$ generated by $\{(r_q-1):\ q \in Q\}$.
\end{proposition}

We can clearly restrict the second exact sequence to the exact sequence
\begin{equation}
0 \rightarrow (N-1) \rightarrow I_{FG} \rightarrow I_G \rightarrow 0
\label{IGmSES}
\end{equation}
where $I_{FG}$ and $I_G$ are the augmentation ideals of $\Q FG$ and $\QG$ respectively.

\subsection{Generalized Algebraic Setting}
\label{AlgSetting}

By analogy with the above group case, we take $\pmb{K}$ to be an augmented (unital) algebra over $\Q$ with 2-sided augmentation ideal $\IK$, and $\pmb{F}$ to be the free algebra over $\Q$ with the same generating set as $K$, with 2-sided augmentation ideal $\pmb{\IF}$.

In particular we assume an exact sequence:

\begin{equation*}
0 \longrightarrow \IK \longrightarrow K \overset{\epsilon}{\longrightarrow} \Q \longrightarrow 0
\end{equation*}

By analogy with the ideal $(N-1)$ in the group context, we let $\pmb{M} \subseteq \IF \subseteq F$ be a 2-sided ideal such that:

\begin{align*}
0 \longrightarrow & M \longrightarrow F \longrightarrow K \longrightarrow 0 \\
0 \longrightarrow & M \longrightarrow \IF \longrightarrow \IK \longrightarrow 0
\end{align*}
are exact.  In fact we will see below that we can (and do) assume $M\subseteq \IF^2$ with very little loss of generality.

We suppose $F$ (and $K$) to be generated by a set $\pmb{X}$.  Again by analogy with the group context, we will suppose $\IF$ (and $\IK$) to be the kernel of the algebra homomorphism which sends each $x\in X$ to $1\in \Q$.  Using this convention, we will exhibit an explicit grading on $F$ which induces the filtration by powers of $\IF$.

Specifically, we define $\pmb{\tX}:=\Q \{ \bx:=(x-1): x\in X\}$.  Then $F$ has the graded structure of tensor algebra\footnote{This can be seen as follows.  Essentially by definition, $F=TX$, the tensor algebra over $X$.  But then it is easy to see that the algebra homomorphism which maps $x \mapsto \bx +1$ converts from the $TX$ presentation to the $T\tX$ presentation.} over $\tX$, i.e. $F=T\tX = \oplus_{p\geq 0} \tX^p$ where $\pmb{\tX^p}$ consists of all sums of $p$-fold products of elements of $\tX$.  We obtain a filtration on $F$ by setting $\pmb{\Fgp}:=\oplus_{q\geq p} \tX^q$.  It should be clear that $\Fgp = \IF^p$.

We will henceforth in fact work with the completions $\hat{K}$ of $K$ (and $\hat{F}$ of $F$) with respect to the filtrations by powers of their respective augmentation ideals.  Our reason for doing this is that, by picking a suitable set of generators for $K$ (and $F$) and passing to the completions, we claim that we may arrange that $M\subseteq \IF^2$ (see Subsection \ref{EliminatingLinearRelations}), and we will assume this condition holds in the sequel.  Since we always work with the completions, we will simply denote them $K$ and $F$, without the hat.

\subsection{The Associated Graded Algebra}

$K$ is filtered by powers of $\IK$:

\begin{equation*}
\dots \hookrightarrow \IK^3 \hookrightarrow \IK^2 \hookrightarrow \IK \hookrightarrow \IK^0=K
\label{FiltrationByI}
\end{equation*}

We denote $\pmb{\grK}$ the associated graded of the above filtration.  We have $\grK \cong \bigoplus_p \IK^p / \IK^{p+1}$.  It is clear that $\grK$ is generated as an algebra by its degree one piece $\pmb{V}:= \IK/\IK^2$, a vector space over $\Q$.

\subsection{A Candidate Presentation for $\IK^p$}

In order to understand the quotients $\IK^p/\IK^{p+1}$, we will first seek a presentation for the $\IK^p$, for $p\geq 0$ (we take $\IK^0 = K$).  Note that, essentially by definition, $\IK^p = \IF^p / (M\cap \IF^p)$.  So we wish to determine $(M\cap \IF^p)$.

Let $\{y_q : q\in Q\}\subseteq \IF^2$ be a set of generators for $M$.  Then define $\pmb{\YF}:=\{Y_q: q\in Q\}$ to be a collection of symbols in 1-1 correspondence with the $\{y_q\}$.  Then we define $\pmb{\RF}$ to be the free 2-sided $F$-module generated by $\YF$ (if $\YF$ is empty, take $\RF=\{0\}$).  Now if we define a map $\pmb{\delK}: \RF \rightarrow F$ which maps $Y_q \mapsto y_q$ (and extend $\delK$ as an $F$-module homomorphism to $\RF$) then it is clear that

\begin{equation}
\label{DefOfKAsHomol}
K= \frac{F}{\delK \RF}
\end{equation}

$\RF$ inherits both a graded and a filtered structure from $F$.  Specifically, if we define:

\begin{equation*}
\Kipq := \tX^q. \Q\YF. \tX^{p-q-2}
\end{equation*}
and
\begin{equation}
\Kip := \sum_{q=0}^{p-2} \Kipq \label{RKmDef}
\end{equation}
then $\RF = \oplus_{p \geq 0} \RKp$ gives a graded structure on $\RF$.

We get a filtered structure by defining

\begin{equation*}
\pmb{\RKgp} := \bigoplus_{q\geq p} \Rg_q
\end{equation*}
and we have the filtration of $\RF$:
\begin{equation*}
\dots \hookrightarrow \Rg_{\geq 3} \hookrightarrow \Rg_{\geq 2} = \RF
\end{equation*}

By construction, $\delK(\RKgp) \subseteq (M\cap \IF^p)$.  In fact we will prove in Section \ref{PVHCriterion} that we have the following:

\begin{proposition}
\label{IStability}
The PVH Criterion (to be defined in Theorem \ref{HutchingsCriterion}) is met if and only if, for all $p\geq 2$, we have $\delK(\RKgp) = (M\cap \IF^p)$, and hence $\IK^p \cong \IF^p / \delK \RKgp$.
\end{proposition}

\subsection{A Graded Approximation to $K$}

We may `approximate' $K$ by taking a graded version of the construction (\ref{DefOfKAsHomol}) of $K$, i.e. we define an algebra

\begin{equation}
\label{DefOfA}
\pmb{A}:= \frac{\oplus_{p\geq 0} \Xp}{\oplus_{p\geq 0} \delA \RKp}
\end{equation}
where $\delA$ is a graded version of $\delK$.

More precisely, we have projections:

\begin{equation*}
\pi_p^0: \Xgp \twoheadrightarrow \Xp
\end{equation*}
which are the identity on $\Xp$ and send $\oplus_{q > p} \tX^q$ to 0.  Similarly, we have projections:

\begin{equation*}
\pi_p^1: \RKgp \twoheadrightarrow \RKp
\end{equation*}
which are the identity on $\RKp$ and send $\oplus_{q > p} \Rg_q$ to 0.

Then $\delA$ is defined by:

\begin{equation*}
\delA(z) = \pi_p^0 \circ \delK (z) \text{      for } z\in \RKp
\end{equation*}

\newpage

These definitions are encapsulated in the following diagram, which is commutative and has exact rows:

\[
\xy
(-30,0)*+{0}="1";
(-15,0)*+{\RK_{\geq p+1}}="2";
(0,0)*+{\RKgp}="3";
(15,0)*+{\RKp}="4";
(30,0)*+{0}="5";
(-30,-15)*+{0}="6";
(-15,-15)*+{\IF^{p+1}}="7";
(0,-15)*+{\IF^p}="8";
(15,-15)*+{\Xp}="9";
(30,-15)*+{0}="10";
{\ar "1";"2"}; {\ar^{\iota} "2";"3"};
{\ar^{\pi^1_p} "3";"4"}; {\ar "4";"5"};
{\ar "6";"7"}; {\ar^{\iota} "7";"8"};
{\ar^{\pi^0_p} "8";"9"}; {\ar "9";"10"};
{\ar^{\delK^{res}} "2";"7"}; {\ar^{\delK} "3";"8"};
{\ar^{\delA} "4";"9"}
\endxy
\]
where $\delK^{res}$ denotes the restriction and $\iota$ denotes the inclusions.

\subsection{The Quadratic Approximation}
\label{TwoCanonAssGradedAlg}

In this subsection, we show that the graded algebra $A$ is the `quadratic approximation' $q(\grK)$ of $\grK$, in the sense that  $A$ has the same generators as $\grK$, and has relations generated by the degree 2 relations of $\grK$.

We noted previously that $\grK$ is generated by its degree 1 piece $V=\IK/\IK^2$.  In fact, because we take $M \subseteq \IF^2$, we have

\begin{equation*}
\IK / \IK^2 = \frac{\IF / M}{\IF^2 / (M \cap \IF^2)} \cong \frac{\IF / M}{\IF^2 / M} \cong \IF / \IF^2
\end{equation*}
so that actually we may view $V$ as $V=\IF/\IF^2$.

With that in mind, it should be clear that we have an isomorphism

\begin{align*}
\tX & \overset{\sim}{\longrightarrow} V=\IF/\IF^2 \\
(x-1) & \mapsto (x-1) + \IF^2
\end{align*}
which extends to an isomorphism of graded algebras $\Pi: F = T\tX \overset{\sim}{\longrightarrow} TV$ (where $TV$ is the tensor algebra of $V$ over $\Q$) by the  universal property of tensor algebras.

Now we identify $V \otimes V = \IF/\IF^2 \otimes \IF / \IF^2$, and define $\Rq:= ker(m_K:\IF/\IF^2 \otimes \IF / \IF^2 \rightarrow \IK^2/\IK^3)$.  Here $m_K$ is the composition

\begin{equation*}
\IF/\IF^2 \otimes \IF / \IF^2 \overset{m_F}{\rightarrow} \IF^2 / \IF^3 \overset{p}{\twoheadrightarrow} \IK^2 / \IK^3
\end{equation*}
where the first map $m_F$ is the isomorphism induced from multiplication in $F$, and the second map $p$ is induced from the projection $F \twoheadrightarrow K$.

The quadratic approximation $q(\grK)$ is formally defined as $q(\grK) := T V / \langle \Rq \rangle$, where $\langle R \rangle$ is the two-sided ideal in $T V$ generated by the vector subspace $\Rq \subseteq V \otimes V$.  Thus $A$ and $q(\grK)$ at least have spaces of generators which are isomorphic via the map $\Pi$.  The following lemma effectively tells us that $\delA\Q\YF = R$ and hence that $A$ and $q(\grK)$ have the same relations:

\begin{lemma}
We have $\Rq \cong (M+\IF^3) / \IF^3 \cong \delA \Q\YF$.
\label{RqStructure}
\end{lemma}

\begin{proof}
It suffices to determine the kernel $ker(\IF^2 / \IF^3 \overset{p}{\twoheadrightarrow} \IK^2 / \IK^3)$, and then pull the result back to $V \otimes V$ via the isomorphism $m_F^{-1}$.  We have:

\begin{equation*}
\frac{\IK^2}{\IK^3} = \frac{\IF^2 / M}{\IF^3 / (M \cap \IF^3)} \cong \frac{\IF^2 / M}{(\IF^3 + M) / M} \cong \frac{\IF^2}{\IF^3 + M} \cong \frac{\IF^2 / \IF^3}{(\IF^3 + M) / \IF^3}
\end{equation*}

It follows that $ker\ p$ must be the last denominator, i.e. $(\IF^3 + M) / \IF^3$, which is what we needed.
\end{proof}

It is immediate from the lemma that in fact $\oplus_{p \geq 2} \delA \RKp = \langle R \rangle$, and hence that $A \cong q(\grK)$.

We will denote by $A^m$ the $m$-th graded piece of $A$.  We note that since $A$ has the same generators and the same quadratic relations as $\grK$, there is always a surjection $A \twoheadrightarrow \grK$.  Quadraticity of $\grK$ is thus equivalent to the fact that this surjection is an isomorphism  $A^m \cong \IK^m/\IK^{m+1}$, for all $m$.  We will often use this alternative definition of quadraticity.

\subsection{The PVH Criterion}

We now resume the thread of our development of the PVH Criterion.  Recall that we have the exact, commutative diagram:

\[
\xy
(-30,0)*+{0}="1";
(-15,0)*+{\RK_{\geq p+1}}="2";
(0,0)*+{\RKgp}="3";
(15,0)*+{\RKp}="4";
(30,0)*+{0}="5";
(-30,-15)*+{0}="6";
(-15,-15)*+{\IF^{p+1}}="7";
(0,-15)*+{\IF^p}="8";
(15,-15)*+{\Xp}="9";
(30,-15)*+{0}="10";
{\ar "1";"2"}; {\ar "2";"3"};
{\ar^{\pi^1_p} "3";"4"}; {\ar "4";"5"};
{\ar "6";"7"}; {\ar "7";"8"};
{\ar^{\pi^0_p} "8";"9"}; {\ar "9";"10"};
{\ar^{\delK^{res}} "2";"7"}; {\ar^{\delK} "3";"8"};
{\ar^{\delA} "4";"9"}
\endxy
\]

We extend the right half of the above diagram by adding kernels at the top:

\[
\xy
(0,15)*+{ker\ \delK}="12";
(20,15)*+{ker\ \delA}="13";
(0,0)*+{\RKgp}="3";
(20,0)*+{\RKp}="4";
(0,-15)*+{\IF^p}="8";
(20,-15)*+{\Xp}="9";
{\ar "12";"3"};
{\ar "13";"4"}; {\ar^{\FSyz_p} "12";"13"};
{\ar@{->>}^{\pi^1_p} "3";"4"};
{\ar@{->>}^{\pi^0_p} "8";"9"}; {\ar^{\delK} "3";"8"};
{\ar^{\delA} "4";"9"}
\endxy
\]
where $\FSyz_p$ is the map induced from $\pi^1_p$ on kernels, and we have abbreviated $\delK |_{\RKgp}$ as $\delK$, and $\delA |_{\RKp}$ as $\delA$.\footnote{Note the $\pi_p^{Syz}: \RF \ra [ker\ \delA]_p$ referred to in the Introduction are just the $\FSyz_p$ above extended by $0$ outside of $ker\ \delK|_{\RKgp}$.}

\newpage

Now, with notation as in the above diagram, and with the assumptions in Subsection \ref{AlgSetting}\footnote{The statement about Koszulness, however, relies on results about Koszul algebras whose graded components are finitely generated over the ground ring.  Hence, for purposes of this part of the theorem, we assume the algebra $K$ to be finitely generated, which is sufficient to ensure that $A^m$ is a finite dimensional $\Q$-vector space for all $m$. \label{fnOnFiniteness}}:

\begin{theorem}[PVH Criterion]
$K$ is quadratic if $\FSyz_p$ is surjective for all $p\geq 2$, i.e. informally if `the syzygies of $A$ also hold in $K$'.

If $A$ is Koszul,\footnote{In fact, $A$ need only be 2-Koszul, i.e. its Koszul complex need only be exact up to homological degree 2 inclusive. \label{OnlyNeedDeg3}} then we need only check that this criterion holds for degree 2 and 3 syzygies of $A$.
\label{HutchingsCriterion}
\end{theorem}

This theorem generalizes a result first obtained in \cite{Hutchings}, where $K$ was the group ring of the pure braid group (see also \cite{BNStoi}).  We give the proof in Section \ref{PVHCriterion}.  As was pointed out to me by Alexander Polishchuk, the result also follows from the paper \cite{PosVish}, whenever the algebra $K$ is finitely generated.

\subsection{Checking the PVH Criterion in Degree 2}

The following proposition gives a sufficient, though not necessary, condition for the criterion to hold in degree 2.

\begin{proposition}
\label{HutchingsDeg2}
Let $\{y_q: q \in Q \}$ be a set of generators for $M$ as a two-sided $F$-module.  If the $\{ y_q + \IF^3 : q \in Q \}$ are linearly independent in $R \cong (M+\IF^3)/\IF^3$, then the PVH Criterion is satisfied in degree 2.
\end{proposition}

\begin{proof}
Indeed, $\delA: \Q \YF \rightarrow V^{\otimes 2} \cong \IF^2 / \IF^3$ is then an inclusion, so $ker\ \delA=0$ and $\FSyz$ is automatically surjective.
\label{RkHutchingsDeg2}
\end{proof}

\subsection{How the PVH Criterion is Useful}
\label{HutchingsUsefulness}

Assuming the requirements of Theorem \ref{HutchingsCriterion} are met, we can conclude that $\grK$ is quadratic if, informally, the syzygies of $A$ also hold in $K$.

It is often the case that the syzygies of a quadratic algebra can be determined quite explicitly, using quadratic duality.  Essentially, if the quadratic algebra $A$ is Koszul, then the syzygies are generated by $A^{!3}$ (i.e. the degree $3$ part of the quadratic dual $A^!$ of $A$). Thus the problem of comparing syzygies is reduced to the finite, computable problem of determining a basis for $A^{!3}$ and checking whether the resulting syzygies of $A^3$ also hold in $K$.

In the context of $\PV$, it was shown in \cite{BEER} that $\pv$ is Koszul (a different proof is provided in Section \ref{ProofOfKoszulness} of this paper), so we only need to check the PVH Criterion in degree 2 and 3.

\newpage

One also knows certain syzygies that are satisfied by the group algebra $gr\Q\PV$, particularly the syzygy known as the Zamolodchikov tetrahedron:\footnote{The picture builds on xy-pic templates due to Aaron Lauda -- see \cite{Lau}.}

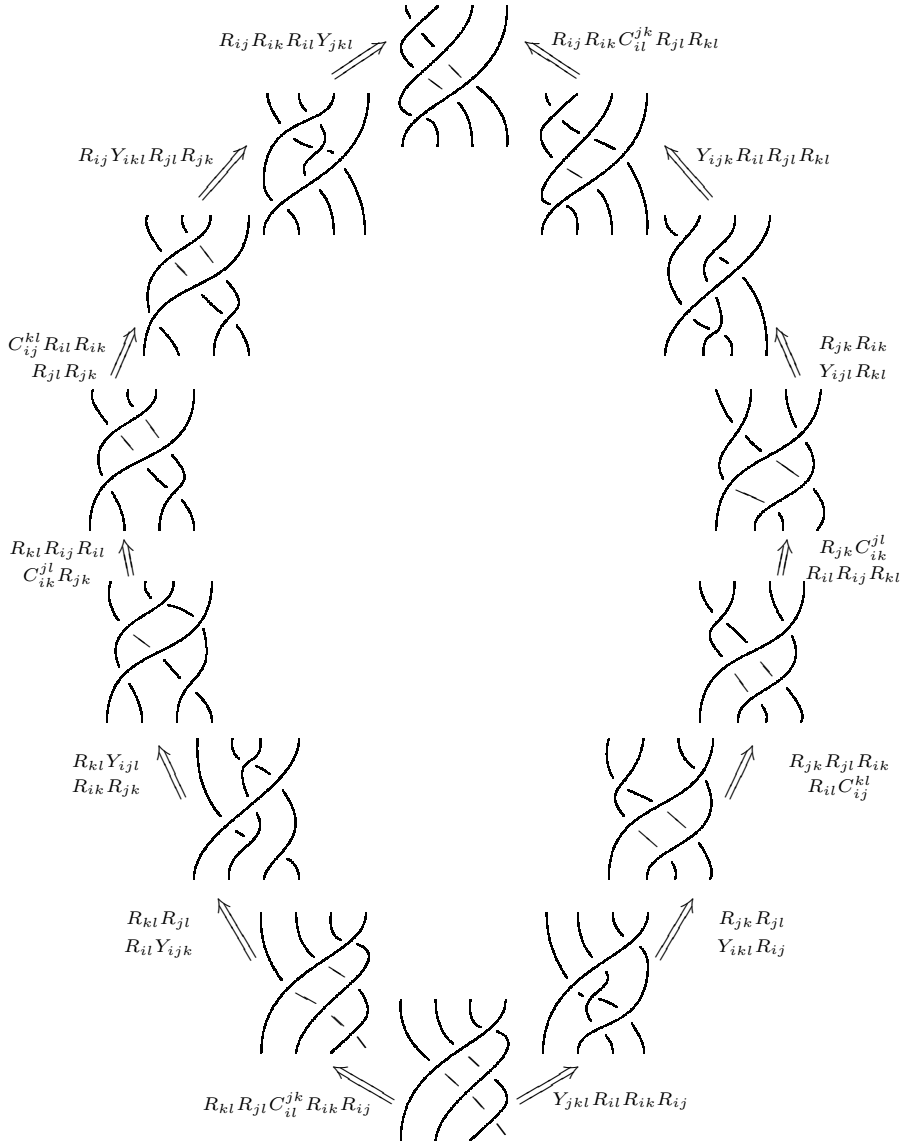
\[
\xy 0;/r.11pc/: 
(0,120)*+{
\xy 
(-15,-20)*{}="T1";
(-5,-20)*{}="T2";
(5,-20)*{}="T3";
(15,-20)*{}="T4";
(-14,20)*{}="B1";
(-5,20)*{}="B2";
(5,20)*{}="B3";
(15,20)*{}="B4";
"T1"; "B4" **\crv{(-15,-7) & (15,-5)}
\POS?(.25)*{\hole}="2x" \POS?(.47)*{\hole}="2y" \POS?(.59)*{
\hole}="2z";
"T2";"2x" **\crv{(-4,-12)};
"T3";"2y" **\crv{(5,-10)};
"T4";"2z" **\crv{(16,-9)};
(-15,-5)*{}="3x";
"2x"; "3x" **\crv{(-18,-10)};
"3x"; "B3" **\crv{(-13,0) & (4,10)}
\POS?(.31)*{\hole}="4x" \POS?(.53)*{\hole}="4y";
"2y"; "4x" **\crv{};
"2z"; "4y" **\crv{};
(-15,10)*{}="5x";
"4x";"5x" **\crv{(-17,6)};
"5x";"B2" **\crv{(-14,12)}
\POS?(.6)*{\hole}="6x";
"6x";"B1" **\crv{(-14,18)};
"4y";"6x" **\crv{(-8,10)};
\endxy
}="T";
(-40,95)*+{
\xy 
(-15,-20)*{}="b1";
(-5,-20)*{}="b2";
(5,-20)*{}="b3";
(14,-20)*{}="b4";
(-14,20)*{}="T1";
(-5,20)*{}="T2";
(5,20)*{}="T3";
(15,20)*{}="T4";
"b1"; "T4" **\crv{(-15,-7) & (15,-5)}
\POS?(.25)*{\hole}="2x" \POS?(.47)*{\hole}="2y" \POS?(.65)*{ \hole}="2z";
"b2";"2x" **\crv{(-5,-15)};
"b3";"2y" **\crv{(5,-10)};
"b4";"2z" **\crv{(14,-9)};
(-15,-5)*{}="3x";
"2x"; "3x" **\crv{(-15,-10)};
"3x"; "T3" **\crv{(-15,15) & (5,10)}
\POS?(.38)*{\hole}="4y" \POS?(.65)*{\hole}="4z";
"T1";"4y" **\crv{(-14,16)};
"T2";"4z" **\crv{(-5,16)};
"2y";"4z" **\crv{(-10,3) & (10,2)} \POS?(.6)*{\hole}="5z";
"4y";"5z" **\crv{(-5,5)};
"5z";"2z" **\crv{(5,4)};
\endxy }="TTL";
(-75,60)*+{
\xy 
(-14,20)*{}="T1";
(-4,20)*{}="T2";
(4,20)*{}="T3";
(15,20)*{}="T4";
(-15,-20)*{}="B1";
(-5,-20)*{}="B2";
(5,-20)*{}="B3";
(15,-20)*{}="B4";
"B1";"T4" **\crv{(-15,5) & (15,-5)}
\POS?(.17)*{\hole}="2x" \POS?(.49)*{\hole}="2y" \POS?(.65)*{ \hole}="2z";
"2x";"T3" **\crv{(-20,10) & (5,10) }
\POS?(.45)*{\hole}="3y" \POS?(.7)*{\hole}="3z";
"2x";"B2" **\crv{(-6,-16)};
"T1";"3y" **\crv{(-16,17)};
"T2";"3z" **\crv{(-5,17)};
"3z";"2z" **\crv{};
"3y";"2y" **\crv{};
"B3";"2z" **\crv{ (5,-5) &(20,-10)}
\POS?(.41)*{\hole}="4z";
"2y";"4z" **\crv{(6,-8)};
"4z";"B4" **\crv{(15,-15)};
\endxy}="TLL";
(-90,10)*+{
\xy 
(-14,20)*{}="T1";
(-4,20)*{}="T2";
(4,20)*{}="T3";
(15,20)*{}="T4";
(-15,-20)*{}="B1";
(-5,-20)*{}="B2";
(5,-20)*{}="B3";
(15,-20)*{}="B4";
"B1";"T4" **\crv{(-15,5) & (15,-5)}
\POS?(.25)*{\hole}="2x" \POS?(.49)*{\hole}="2y" \POS?(.65)*{ \hole}="2z";
"2x";"T3" **\crv{(-20,10) & (5,10) }
\POS?(.45)*{\hole}="3y" \POS?(.7)*{\hole}="3z";
"2x";"B2" **\crv{(-5,-14)};
"T1";"3y" **\crv{(-16,17)};
"T2";"3z" **\crv{(-5,17)};
"3z";"2z" **\crv{};
"3y";"2y" **\crv{};
"B3";"2z" **\crv{ (5,-5) &(20,-10)}
\POS?(.4)*{\hole}="4z";
"2y";"4z" **\crv{(6,-8)};
"4z";"B4" **\crv{(15,-15)};
\endxy }="ML";
(-85,-45)*+{
\xy 
(-14,20)*{}="T1";
(-4,20)*{}="T2";
(4,20)*{}="T3";
(15,20)*{}="T4";
(-15,-20)*{}="B1";
(-5,-20)*{}="B2";
(5,-20)*{}="B3";
(15,-20)*{}="B4";
"B1";"T4" **\crv{(-15,5) & (15,-5)}
\POS?(.25)*{\hole}="2x" \POS?(.49)*{\hole}="2y" \POS?(.8)*{ \hole}="2z";
"2x";"T3" **\crv{(-20,10) & (5,10) }
\POS?(.38)*{\hole}="3y" \POS?(.72)*{\hole}="3z";
"2x";"B2" **\crv{(-5,-14)};
"T1";"3y" **\crv{(-16,13)};
"T2";"3z" **\crv{(-5,17)};
"3z";"2z" **\crv{(5,14)};
"3y";"2y" **\crv{};
"B3";"2z" **\crv{ (5,-5) &(18,-10)}
\POS?(.34)*{\hole}="4z";
"2y";"4z" **\crv{(6,-8)};
"4z";"B4" **\crv{(15,-15)};
\endxy}="BLL";
(-60,-90)*+{
\xy 
(-14,20)*{}="T1";
(-4,20)*{}="T2";
(4,20)*{}="T3";
(15,20)*{}="T4";
(-15,-20)*{}="B1";
(-5,-20)*{}="B2";
(5,-20)*{}="B3";
(15,-20)*{}="B4";
"B1";"T4" **\crv{(-15,-5) & (15,5)}
\POS?(.38)*{\hole}="2x" \POS?(.53)*{\hole}="2y" \POS?(.7)*{\hole}="2z";
"T1";"2x" **\crv{(-15,5)};
"2y";"B2" **\crv{(10,-10) & (-6,-10)} \POS?(.45)*{\hole}="4x";
"2z";"B3" **\crv{ (15,0)&(15,-10) & (6,-16)} \POS?(.7)*{\hole}="5x";
"T3";"2y" **\crv{(5,10)& (-6,18) }
\POS?(.5)*{\hole}="3x";
"T2";"3x" **\crv{(-5,15)};
"3x";"2z" **\crv{(7,11)};
"2x";"4x" **\crv{(-3,-7)};
"4x";"5x" **\crv{};
"5x";"B4" **\crv{(15,-15)};
\endxy }="BBLL";
(-40,-140)*+{
\xy 
(15,20)*{}="T1";
(5,20)*{}="T2";
(-5,20)*{}="T3";
(-15,20)*{}="T4";
(15,-20)*{}="B1";
(5,-20)*{}="B2";
(-5,-20)*{}="B3";
(-15,-20)*{}="B4";
"T1"; "B4" **\crv{(15,7) & (-15,5)}
\POS?(.25)*{\hole}="2x" \POS?(.45)*{\hole}="2y" \POS?(.67)*{\hole}="2z";
"T2";"2x" **\crv{(4,12)};
"T3";"2y" **\crv{(-5,10)};
"T4";"2z" **\crv{(-16,9)};
(15,5)*{}="3x";
"2x"; "3x" **\crv{(18,10)};
"3x"; "B3" **\crv{(13,0) & (-4,-10)}
\POS?(.25)*{\hole}="4x" \POS?(.56)*{\hole}="4y";
"2y"; "4x" **\crv{};
"2z"; "4y" **\crv{};
(15,-10)*{}="5x";
"4x";"5x" **\crv{(17,-6)};
"5x";"B2" **\crv{(14,-12)}
\POS?(.6)*{\hole}="6x";
"6x";"B1" **\crv{};
"4y";"6x" **\crv{};
\endxy }="BBBL";
(40,95)*+{
\xy 
(-15,-20)*{}="T1";
(-5,-20)*{}="T2";
(5,-20)*{}="T3";
(15,-20)*{}="T4";
(-14,20)*{}="B1";
(-5,20)*{}="B2";
(5,20)*{}="B3";
(15,20)*{}="B4";
"T1"; "B4" **\crv{(-15,-7) & (15,-5)}
\POS?(.25)*{\hole}="2x" \POS?(.47)*{\hole}="2y" \POS?(.69)*{
\hole}="2z"; 
"T2";"2x" **\crv{(-4,-12)};
"T3";"2y" **\crv{(5,-10)};
"T4";"2z" **\crv{(16,-9)};
(-15,-8)*{}="3x"; 
"2x"; "3x" **\crv{(-16,-12)}; 
"3x"; "B3" **\crv{(-13,0) & (4,10)}
\POS?(.3)*{\hole}="4x" \POS?(.59)*{\hole}="4y"; 
"2y"; "4x" **\crv{};
"2z"; "4y" **\crv{(7,4)}; 
(-15,10)*{}="5x";
"4x";"5x" **\crv{(-17,6)};
"5x";"B2" **\crv{(-14,12)}
\POS?(.6)*{\hole}="6x";
"6x";"B1" **\crv{(-14,18)};
"4y";"6x" **\crv{(-8,10)};
\endxy
}="TTR";
(75,60)*+{
\xy 
(14,-20)*{}="T1";
(4,-20)*{}="T2";
(-4,-20)*{}="T3";
(-15,-20)*{}="T4";
(15,20)*{}="B1";
(5,20)*{}="B2";
(-5,20)*{}="B3";
(-15,20)*{}="B4";
"B1";"T4" **\crv{(15,5) & (-15,-5)}
\POS?(.38)*{\hole}="2x" \POS?(.53)*{\hole}="2y" \POS?(.7)*{\hole}="2z";
"T1";"2x" **\crv{(15,-5)};
"2y";"B2" **\crv{(-10,10) & (6,10)}
\POS?(.45)*{\hole}="4x";
"2z";"B3" **\crv{ (-15,0)&(-15,10) & (-6,16)}
\POS?(.7)*{\hole}="5x";
"T3";"2y" **\crv{(-5,-10)& (6,-18) }
\POS?(.5)*{\hole}="3x";
"T2";"3x" **\crv{(5,-15)};
"3x";"2z" **\crv{(-7,-11)};
"2x";"4x" **\crv{(3,7)};
"4x";"5x" **\crv{};
"5x";"B4" **\crv{(-15,15)};
\endxy }="TRR";
(90,10)*+{
\xy 
(14,-20)*{}="T1";
(4,-20)*{}="T2";
(-4,-20)*{}="T3";
(-15,-20)*{}="T4";
(15,20)*{}="B1";
(5,20)*{}="B2";
(-5,20)*{}="B3";
(-15,20)*{}="B4";
"B1";"T4" **\crv{(13,-3) & (-15,5)}
\POS?(.26)*{\hole}="2x" \POS?(.48)*{\hole}="2y" \POS?(.78)*{ \hole}="2z";
"2x";"T3" **\crv{(23,-10) & (-5,-10) }
\POS?(.37)*{\hole}="3y" \POS?(.7)*{\hole}="3z";
"2x";"B2" **\crv{(5,14)};
"T1";"3y" **\crv{(19,-14)};
"T2";"3z" **\crv{(5,-17)};
"3z";"2z" **\crv{};
"3y";"2y" **\crv{};
"B3";"2z" **\crv{ (-5,5) &(-20,3)}
\POS?(.31)*{\hole}="4z";
"2y";"4z" **\crv{(-5,4)};
"4z";"B4" **\crv{(-12,11)};
\endxy
}="MR";
(85,-45)*+{
\xy 
(14,-20)*{}="T1";
(4,-20)*{}="T2";
(-4,-20)*{}="T3";
(-15,-20)*{}="T4";
(15,20)*{}="B1";
(5,20)*{}="B2";
(-5,20)*{}="B3";
(-15,20)*{}="B4";
"B1";"T4" **\crv{(15,-5) & (-15,5)}
\POS?(.25)*{\hole}="2x" \POS?(.49)*{\hole}="2y" \POS?(.65)*{ \hole}="2z";
"2x";"T3" **\crv{(20,-10) & (-5,-10) }
\POS?(.45)*{\hole}="3y" \POS?(.7)*{\hole}="3z";
"2x";"B2" **\crv{(5,14)};
"T1";"3y" **\crv{(16,-17)};
"T2";"3z" **\crv{(5,-17)};
"3z";"2z" **\crv{};
"3y";"2y" **\crv{};
"B3";"2z" **\crv{ (-5,5) &(-20,10)}
\POS?(.4)*{\hole}="4z";
"2y";"4z" **\crv{(-6,8)};
"4z";"B4" **\crv{(-15,15)};
\endxy
}="BRR";
(60,-90)*+{
\xy 
(14,-20)*{}="T1";
(4,-20)*{}="T2";
(-4,-20)*{}="T3";
(-15,-20)*{}="T4";
(15,20)*{}="B1";
(5,20)*{}="B2";
(-5,20)*{}="B3";
(-15,20)*{}="B4";
"B1";"T4" **\crv{(13,-3) & (-15,5)}
\POS?(.18)*{\hole}="2x" \POS?(.49)*{\hole}="2y" \POS?(.69)*{ \hole}="2z";
"2x";"T3" **\crv{(23,-10) & (-5,-10) }
\POS?(.45)*{\hole}="3y" \POS?(.7)*{\hole}="3z";
"2x";"B2" **\crv{(5,14)};
"T1";"3y" **\crv{(16,-17)};
"T2";"3z" **\crv{(5,-17)};
"3z";"2z" **\crv{};
"3y";"2y" **\crv{};
"B3";"2z" **\crv{ (-7,5) &(-20,7)}
\POS?(.51)*{\hole}="4z";
"2y";"4z" **\crv{(-6,4)};
"4z";"B4" **\crv{(-17,12)};
\endxy}="BBRR";
(40,-140)*+{
\xy 
(15,20)*{}="b1";
(5,20)*{}="b2";
(-5,20)*{}="b3";
(-14,20)*{}="b4";
(14,-20)*{}="T1";
(5,-20)*{}="T2";
(-5,-20)*{}="T3";
(-15,-20)*{}="T4";
"b1"; "T4" **\crv{(15,7) & (-15,5)}
\POS?(.25)*{\hole}="2x" \POS?(.47)*{\hole}="2y" \POS?(.65)*{ \hole}="2z";
"b2";"2x" **\crv{(5,15)};
"b3";"2y" **\crv{(-5,10)};
"b4";"2z" **\crv{(-14,9)};
(15,5)*{}="3x";
"2x"; "3x" **\crv{(15,10)};
"3x"; "T3" **\crv{(15,-15) & (-5,-10)}
\POS?(.38)*{\hole}="4y" \POS?(.65)*{\hole}="4z";
"T1";"4y" **\crv{(14,-16)};
"T2";"4z" **\crv{(5,-16)};
"2y";"4z" **\crv{(10,-3) & (-10,-2)} \POS?(.6)*{\hole}="5z";
"4y";"5z" **\crv{(5,-5)};
"5z";"2z" **\crv{(-5,-4)};
\endxy }="BBBR";
(0,-165)*+{
\xy 
(15,20)*{}="T1";
(5,20)*{}="T2";
(-5,20)*{}="T3";
(-15,20)*{}="T4";
(15,-20)*{}="B1";
(5,-20)*{}="B2";
(-5,-20)*{}="B3";
(-15,-20)*{}="B4";
"T1"; "B4" **\crv{(15,7) & (-15,5)}
\POS?(.25)*{\hole}="2x" \POS?(.45)*{\hole}="2y" \POS?(.6)*{\hole}="2z";
"T2";"2x" **\crv{(4,12)};
"T3";"2y" **\crv{(-5,10)};
"T4";"2z" **\crv{(-16,9)};
(15,5)*{}="3x";
"2x"; "3x" **\crv{(18,10)};
"3x"; "B3" **\crv{(13,0) & (-4,-10)}
\POS?(.31)*{\hole}="4x" \POS?(.53)*{\hole}="4y";
"2y"; "4x" **\crv{};
"2z"; "4y" **\crv{};
(15,-10)*{}="5x";
"4x";"5x" **\crv{(17,-6)};
"5x";"B2" **\crv{(14,-12)}
\POS?(.6)*{\hole}="6x";
"6x";"B1" **\crv{};
"4y";"6x" **\crv{};
\endxy }="B";
(-20,130)*{}="X2";
(-35,120)*{}="X1"; 
{\ar@{=>}^{R_{ij}R_{ik}R_{il}Y_{jkl}\quad } "X1";"X2"}; 
(20,130)*{}="X2";
(35,120)*{}="X1"; 
{\ar@{=>}_{\quad R_{ij}R_{ik}C_{il}^{jk}R_{jl}R_{kl}} "X1";"X2"}; 
(-60,100)*{}="X2"; 
(-73,85)*{}="X1"; 
{\ar@{=>}^{R_{ij}Y_{ikl}R_{jl}R_{jk}} "X1";"X2"}; 
(60,100)*{}="X2"; 
(73,85)*{}="X1"; 
{\ar@{=>}_{Y_{ijk}R_{il}R_{jl}R_{kl}} "X1";"X2"}; 
(-92,48)*{}="X2"; 
(-98,34)*{}="X1"; 
{\ar@{=>} "X1";"X2"}; 
(-114,43)*{\scriptstyle{C_{ij}^{kl}R_{il}R_{ik}}};
(-112,35)*{\scriptstyle{R_{jl}R_{jk}}};
(92,48)*{}="X2"; 
(98,34)*{}="X1"; 
{\ar@{=>} "X1";"X2"}; 
(114,43)*{\scriptstyle{R_{jk}R_{ik}}};
(114,35)*{\scriptstyle{Y_{ijl}R_{kl}}};
(-95,-13)*{}="X2"; 
(-93,-23)*{}="X1"; 
{\ar@{=>} "X1";"X2"}; 
(-114,-16)*{\scriptstyle{R_{kl}R_{ij}R_{il}}};
(-114,-23)*{\scriptstyle{C_{ik}^{jl}R_{jk}}};
(95,-13)*{}="X2"; 
(93,-23)*{}="X1"; 
{\ar@{=>} "X1";"X2"};
(114,-15)*{\scriptstyle{R_{jk}C_{ik}^{jl}}};
(114,-23)*{\scriptstyle{R_{il}R_{ij}R_{kl}}};
(-85,-72)*{}="X2";
(-78,-87)*{}="X1";
{\ar@{=>} "X1";"X2"}; 
(-100,-77)*{\scriptstyle{R_{kl}Y_{ijl}}};
(-100,-84)*{\scriptstyle{R_{ik}R_{jk}}};
(85,-72)*{}="X2";
(78,-87)*{}="X1";
{\ar@{=>} "X1";"X2"};
(110,-77)*{\scriptstyle{R_{jk}R_{jl}R_{ik}}};
(110,-84)*{\scriptstyle{R_{il}C_{ij}^{kl}}};
(-68,-116)*{}="X2";
(-58,-133)*{}="X1";
{\ar@{=>} "X1";"X2"}; 
(-85,-122)*{\scriptstyle{R_{kl}R_{jl}}};
(-85,-130)*{\scriptstyle{R_{il}Y_{ijk}}};
(68,-116)*{}="X2";
(58,-133)*{}="X1";
{\ar@{=>} "X1";"X2"};
(85,-122)*{\scriptstyle{R_{jk}R_{jl} }};
(85,-130)*{\scriptstyle{Y_{ikl}R_{ij}}};
(-35,-164)*{}="X2";
(-18,-174)*{}="X1";
{\ar@{=>}^{R_{kl}R_{jl} C_{il}^{jk} R_{ik}R_{ij} \quad} "X1";"X2"};
(35,-164)*{}="X2";
(18,-174)*{}="X1";
{\ar@{=>}_{\quad \ Y_{jkl} R_{il}R_{ik} R_{ij}} "X1";"X2"};
\endxy \]

The notation will be clarified in Subsection \ref{SectionKoszulTerminol}.

In the second part of this paper, we will find a basis for the quadratic dual algebra $\pvq$, and in particular for $\pvqiii$.  We will then check `by hand' that the corresponding degree 3 syzygies of $A$ are also satisfied by $K$.  These consist primarily of syzygies which correspond to the `Zamolodchikov' syzygy alluded to above (this is explained in Subsection \ref{sectionGlobSyz}).  This will allow us to conclude that $\grPV \cong \pv$.

\section{Postponed Proofs}

\subsection{Eliminating Linear Relations}
\label{EliminatingLinearRelations}

We mentioned in Subsection \ref{AlgSetting} that we work with the completions $\hat{K}$ of $K$ (and $\hat{F}$ of $F$) because, by picking a suitable set of generators for $K$ (and $F$) and passing to the completions, we may arrange that $M\subseteq \IF^2$.

To prove this claim, we let $\{\xp: p\in P\}$ be a set of generators for the algebra $K$, so that $\{\bxp:=(\xp-1): p\in P\}$ is a set of generators for $\IK$ as a left- or right-sided ideal in $K$.  The images of the $\{\bxp: p\in P\}$ in the vector space $\IKh$ generate that space, so the images of some subset $\{\bxp: p\in S \subseteq P\}$ form a basis.  Thus the $\{\bxp: p\in P - S\}$ may be expressed as linear combinations of the $\{\bxp: p\in S\}$ modulo elements of $\IK^2$.  More generally, we may replace any polynomial involving the $\{\bxp: p\in P - S\}$ by a polynomial involving only $\{\bxp: p\in S\}$, modulo elements in higher powers of $\IK$.  It therefore follows that the $\{\bxp: p\in S\}$ generate the completion, and we may drop the $\{\bxp: p\in P - S\}$ from our list of generators.

We note in particular that in the case where $K$ is the group algebra of some group $G$, the generators of $K$ as an algebra would normally include, not only the group generators, but also their inverses.  Moreover, the relations ideal $M$ would include relations derived from the group laws for the generators (recall that $F$ is the free algebra on $X$, not the free group algebra on $X$).  Thus if $a$ is a generator of the group, and $b$ its inverse, we have the group law $ab=1$ which gives, under the substitution $a\mapsto \bar{a}+1$, $b\mapsto \bar{b}+1$, where $\bar{a}:=(a-1)$ and $\bar{b}:=(b-1)$, the relation $\bar{a}+\bar{b}+\bar{a}\bar{b}=0$, which is not in $\IF^2$.  However using the relation $\bar{b}=-\bar{a}-\bar{a}\bar{b}$ we can replace all occurrences of $\bar{b}$ by $-\bar{a}$, provided we are working in the completion of $K$.  So in the case of group algebras we will take as generators only the group generators and we omit the group law relations from $M$.

Coming back to the case of a general $K$, note that $\IFh$ and $\IKh$ are now vector spaces with bases having the same number of elements, and hence are isomorphic.  We can use this observation to see that $M\subseteq \IF^2$.

\newpage

Indeed, it is also clear that:

\begin{align*}
\IKh &= \frac{\IF/M}{\IF^2 / (M \cap \IF^2)} \\
&= \frac{\IF/M}{(\IF^2 + M) /M} \\
&= \frac{\IF}{\IF^2 + M} \\
&= \frac{\IFh}{(\IF^2 + M)/\IF^2} \\
&= \frac{\IFh}{M/(M\cap \IF^2)}
\end{align*}
so we must have $M/(M\cap \IF^2)=0$, i.e. $M\subseteq \IF^2$.

\subsection{Proof of Proposition \ref{IStability} and Theorem \ref{HutchingsCriterion}}
\label{PVHCriterion}

Recall that we have the following exact, commutative diagram:

\[
\xy
(-30,0)*+{0}="1";
(-15,0)*+{\RK_{\geq p+1}}="2";
(0,0)*+{\RKgp}="3";
(15,0)*+{\RKp}="4";
(30,0)*+{0}="5";
(-30,-15)*+{0}="6";
(-15,-15)*+{\IF^{p+1}}="7";
(0,-15)*+{\IF^p}="8";
(15,-15)*+{\Xp}="9";
(30,-15)*+{0}="10";
{\ar "1";"2"}; {\ar "2";"3"};
{\ar^{\pi^1_p} "3";"4"}; {\ar "4";"5"};
{\ar "6";"7"}; {\ar "7";"8"};
{\ar^{\pi^0_p} "8";"9"}; {\ar "9";"10"};
{\ar^{\delK^{res}} "2";"7"}; {\ar^{\delK} "3";"8"};
{\ar^{\delA} "4";"9"}
\endxy
\]
where $\delK^{res}$ denotes the restriction.

We extend the above diagram by adding a row of kernels at the top, and a row of cokernels at the bottom:

\[
\xy
(-20,15)*+{ker\ \delK^{res}}="11";
(0,15)*+{ker\ \delK}="12";
(20,15)*+{ker\ \delA}="13";
(-40,0)*+{0}="1";
(-20,0)*+{\RK_{\geq p+1}}="2";
(0,0)*+{\RKgp}="3";
(20,0)*+{\RKp}="4";
(40,0)*+{0}="5";
(-40,-15)*+{0}="6";
(-20,-15)*+{\IF^{p+1}}="7";
(0,-15)*+{\IF^p}="8";
(20,-15)*+{\Xp}="9";
(40,-15)*+{0}="10";
(-20,-30)*+{\frac{\IF^{p+1}}{\delK^{res} \RK_{\geq p+1}}}="14";
(0,-30)*+{\frac{\IF^p}{\delK \RKgp}}="15";
(20,-30)*+{A^p}="16";
{\ar "1";"2"}; {\ar "2";"3"};
{\ar^{\pi^1_p} "3";"4"}; {\ar "4";"5"};
{\ar "6";"7"}; {\ar "7";"8"};
{\ar^{\pi^0_p} "8";"9"}; {\ar "9";"10"};
{\ar^{\delK} "2";"7"}; {\ar^{\delK} "3";"8"};
{\ar^{\delA} "4";"9"};
{\ar "11";"2"}; {\ar "12";"3"};
{\ar "13";"4"};
{\ar "7";"14"}; {\ar "8";"15"};
{\ar "9";"16"};
{\ar "11";"12"}; {\ar^{\FSyz_p} "12";"13"};
{\ar^{\nu_p} "14";"15"}; {\ar "15";"16"}
\endxy
\]
where we have abbreviated $\delK |_{\RKgp}$ as $\delK$, and $\delA |_{\RKp}$ as $\delA$.

\newpage

By the Snake Lemma, the following sequence is exact:

\begin{lemma}
\label{HexLemma2}
\begin{multline}
ker\ \delK \xrightarrow{\FSyz_p} ker\ \delA
\rightarrow \frac{\IF^{p+1}}{\delK \RK_{\geq p+1}} \xrightarrow{\nu_p} \frac{\IF^p}{\delK \RKgp} \longrightarrow A^p \longrightarrow  0 \label{HutchingsLES}
\end{multline}
\end{lemma}

Also, as is clear from the long exact sequence, we have:

\begin{lemma}
For every $p\geq 2$, $\FSyz_p$ is surjective if and only if $\nu_p: \frac{\IF^{p+1}}{\delK \RK_{\geq p+1}} \rightarrow \frac{\IF^p}{\delK \RKgp}$ is injective.
\end{lemma}

\begin{proof}[Proof of Proposition \ref{IStability}]

It is clear that the $\nu_p: \frac{\IF^p}{\delK \RKgp} \rightarrow \frac{\IF^{p-1}}{\delK \RK_{\geq p-1}}$ are injective for all $p\geq 2$ if and only if the compositions $\frac{\IF^p}{\delK \RKgp} \rightarrow \nu_p(\frac{\IF^p}{\delK \RK_{\geq p}}) \rightarrow \dots \rightarrow \IK^p$ are also injective for all $p\geq 2$.  Since these compositions are always surjective, injectivity is equivalent to isomorphism.  This proves Proposition \ref{IStability}.
\end{proof}

\begin{proof}[Proof of Theorem \ref{HutchingsCriterion}]
The first claim in Theorem \ref{HutchingsCriterion} follows from the long exact sequence (\ref{HutchingsLES}) and Proposition \ref{IStability}, which imply that

\begin{equation*}
A^p \cong \frac{\IF^p}{\delK \RKgp} / \frac{\IF^{p+1}}{\delK \RK_{\geq p+1}} \cong \IK^p/\IK^{p+1}
\end{equation*}
whenever $\FSyz_p$ is surjective.

We deal with the restriction to degrees 2 and 3 for the Koszul case in the next subsection.
\end{proof}

\subsection{Quadratic Duality and the Role of Koszulness}

In this subsection we briefly review the theory of quadratic algebras to the extent needed to prove the final claim in Theorem \ref{HutchingsCriterion}, and to cover material that will be needed later (but skipping proofs).  The reader who is not familiar with this theory can find a quick overview in \cite{Froberg} or \cite{Hille}, or more extensive treatment in \cite{Pol} and \cite{Kraehmer}; the original source is \cite{Priddy}.\footnote{As noted in footnote \ref{fnOnFiniteness}, we rely on results about Koszul algebras which have only been developed for graded algebras whose graded components are finitely generated over the ground ring.  Hence, wherever we rely on Koszulness of $A$, we assume the algebra $K$ to be finitely generated.  This is sufficient to ensure that $A^m$ is finitely generated over $\Q$.}

We start with the quadratic algebra $A$ which is given by $A= T V / \langle \Rq \rangle$ (in the notation of Subsection \ref{TwoCanonAssGradedAlg}).  The quadratic dual algebra $A^!$ is defined as $A^!:= TV^* / \langle \Rqp \rangle$, where $V^*$ is the linear dual vector space and $\Rqp \subseteq V^*\otimes V^*$ is the annihilator of $\Rq$.

One indication of the usefulness of the concept of quadratic duality is that the degree 2 part of the dual algebra catalogues the relations of the original algebra (this is true for all quadratic algebras).  More generally, the Koszul complex provides a readily computable `candidate' resolution for $A$, which is an actual resolution precisely when $A$ is Koszul.   In particular the degree 3 part of the dual provides at least a candidate catalogue of the relations among the relations of the original algebra (i.e. syzygies) - and more generally, the degree $m$ part of the dual provides a candidate catalogue of the relations among relations among ... (($m-1)$ times) of the original algebra (which we will call the level $m$ syzygies).  Moreover, there are specific maps from the degree $m$ part of the dual into the space of level $m$ syzygies.  The statement that a quadratic algebra is Koszul is equivalent to the statement that the dual algebra not only provides a candidate catalogue of the syzygies of all levels, but an actual, complete catalogue of those syzygies.  For purposes of this paper, it is only the level 3 syzygies that are important.

More specifically, if we define $\Diti: A^{!2*} \rightarrow V\otimes V$ as the dual to multiplication $V^* \otimes V^* \rightarrow A^{!2}$, then in fact $\Diti$ is an isomorphism:

\begin{align}
\Diti: \ \ & A^{!2*} \overset{\sim}{\rightarrow} \ \Rq \label{RKq2Iso}
\end{align}
Thus $A^{!2*}$ catalogues the degree 2 relations of $A$ and the map $\Diti$ sends a basis of $A^{!2*}$ to a basis of $\Rq$ (see (\ref{yijk}) and (\ref{cijkl}) below, in the case of $\pvq$).

In the same vein, $A^{!3*}$ catalogues all relations between relations of $A$, in degree three\footnote{If we assume that $\delA: \Q\YF \rightarrow R$ is injective (and hence the PVH Criterion is satisfied in degree 2) then level 3 syzygies must have at least degree 3 in the generators of $A$.  Given a level 3, degree 3 syzygy, we can also get level 3 syzygies of higher degree by pre- or post-multiplying all terms in the syzygy by monomials in the generators, although level 3 syzygies of higher degree need not all arise in this way (except when the algebra is Koszul).} - in other words, $A^{!3*} \cong (\Rq  \otimes  V \cap V  \otimes  \Rq)$ (see \cite{Pol}, proof of Theorem 4.4.1).  More specifically, if $\Diiti$ is dual to the multiplication: $A^{!2} \otimes V^* \twoheadrightarrow A^{!3}$, then the map

\begin{equation}
(\Diti \negthinspace \otimes 1) \circ \Diiti: \ \ A^{!3*} \hookrightarrow \ \Rq \negthinspace  \otimes \negthinspace  V  \label{X31}
\end{equation}
is actually an isomorphism $A^{!3*} \overset{\sim}{\rightarrow} (\Rq \negthinspace  \otimes \negthinspace  V \cap V \negthinspace  \otimes \negthinspace  \Rq)$.

Similarly, if $\Ditii$ is dual to the multiplication: $V^* \otimes A^{!2} \rightarrow A^{!3}$, the map:

\begin{equation}
(1 \negthinspace  \otimes \Diti) \circ \Ditii: \ \ A^{!3*} \hookrightarrow \ V \negthinspace  \otimes \negthinspace  R \label{X32}
\end{equation}
is an isomorphism $A^{!3*} \overset{\sim}{\rightarrow} (\Rq \otimes  V \cap V \otimes \Rq)$.  By (co-)associativity, $(\Diti \negthinspace \otimes 1) \circ \Diiti = (1 \negthinspace  \otimes \Diti) \circ \Ditii$ are the same map, but it will be useful to keep both notations on hand in the sequel.

The following theorem, which follows from \cite{Pol}, Theorem 2.4.1 (p.29) and Proposition 1.7.2 (p.16), to which the reader is referred for proofs, explains precisely the role Koszulness in determining the syzygies of $A$.  We adopt the notation $\xmi := V^{\otimes i} \otimes R \otimes V^{\otimes m-i-2}$.

\begin{theorem}
Koszulness\footnote{As per footnote \ref{OnlyNeedDeg3}, $A$ need only be 2-Koszul, i.e. its Koszul complex need only be exact up to homological degree 2 inclusive.} of the algebra $A$ implies exactness of the sequence:
\begin{equation}
\bigoplus_{i<j} (\xmi \cap \xmj) \overset{\delSyz}{\longrightarrow} \bigoplus_i \xmi \overset{\kappa}{\longrightarrow} V ^{\otimes m}
\label{DResolution}
\end{equation}
\end{theorem}

In (\ref{DResolution}), the direct sums are external, and the maps are induced from the following diagram:

\[
\xy
(-16,0)*+{\xmi \cap \xmj}="1";
(0,13)*{\quad\ \xmi}="2";
(0,-13)*{\quad\ \xmj}="3";
(16,0)*+{V^{\otimes m}}="4";
(-11,8)*{\scriptstyle (+1)}; (-10,-8)*{\scriptstyle (-1)}
{\ar "1";"2"} {\ar "1";"3"} {\ar "2";"4"} {\ar "3";"4"}
\endxy
\]
where the left diagonals are multiplication by the indicated factors, and the right diagonals are the inclusions.

Note that we can decompose $\bigoplus_{i<j} (\xmi \cap \xmj)$ as follows:

\begin{equation*}
\bigoplus_{i<j} (\xmi \bigcap \xmj) = \bigoplus_i (\xmi \cap \Rq_{m,i+1}) \oplus \bigoplus_{i+1<j} (\xmi \cap \xmj)
\end{equation*}

The syzygies $\bigoplus_{i+1<j} (\xmi \cap \xmj)$ are `trivial' in the sense that they arise from the obvious fact that non-overlapping relations commute.  This fact remains true at the global level, so that these `trivial' syzygies also trivially satisfy the PVH Criterion.

The more interesting syzygies are the $(\xmi \cap \Rq_{m,i+1}) $.  From the review given above, we have $(\xmi \cap \Rq_{m,i+1}) \cong V^{\otimes i} \otimes (R\otimes V \cap V\otimes R) \otimes V^{\otimes m-i-3} \cong V^{\otimes i} \otimes A^{!3*} \otimes V^{\otimes m-i-3}$.  This makes clear that the PVH Criterion need only be checked up to degree 3 in the Koszul case.

\section{The Quadraticity of $\PV$}

\subsection{Overview}

We now turn to $\PV$.  Our goal is to establish the following theorem:

\begin{theorem}
\label{ThmPvBQuadratic}
$\PV$ is quadratic.
\end{theorem}

As pointed out in Section 8.5 of \cite{BEER}, Theorem \ref{ThmPvBQuadratic} implies the truth of their Conjecture 8.6, as we see next.

\begin{corollary}
\label{CorQTrCohomology}
$H^*(\PV) \cong \pvq$ as algebras.
\end{corollary}

Note that there are natural homomorphisms $\PfB \rightarrow \PV \rightarrow \PfB$, with the composition being the identity (this is pointed out in Section 2.3 of \cite{BEER}).  The second map sends all generators $R_{ij}$ to themselves, and the first sends $R_{ij}$ to itself whenever $i<j$.  It follows that $\PfB$ is a split quotient of $\PV$ (and similarly $\pf$ is a split quotient of $\pv$ by essentially the same reasoning).  Hence quadraticity of $\PV$ implies:

\begin{corollary}
\label{CorPfBQuadratic}
$\PfB$ is quadratic.
\end{corollary}
This confirms a result originally announced in \cite{BEER} (in which $\PfB$ is referred to as $T \negthinspace r_n$).  Again as pointed out in Section 8.5 of \cite{BEER}, Corollary \ref{CorPfBQuadratic} implies the following (which is Theorem 8.5 of \cite{BEER}):

\begin{corollary}
\label{CorTrCohomology}
$H^*(\PfB) \cong \pfq$ as algebras.
\end{corollary}

Because we prove the quadraticity of $\PV$ using the PVH Criterion, we go through the following steps:

\begin{itemize}
\item Check that the preliminary requirements (as per Subsection \ref{AlgSetting}) for applying the PVH Criterion are met.
\item Find the infinitesimal syzygies.  We will use the fact that $\pv$ is Koszul, and that accordingly the infinitesimal syzygies are essentially given by ${\pvqiii}^*\cong {\pvqiii}$.  (The Koszulness of $\pvq$ was first established in \cite{BEER}, and we give an alternative proof in subsection \ref{ProofOfKoszulness}).  After finding a basis for $\pvq$, and in particular for $\pvqiii$, we will see that finding the infinitesimal syzygies becomes a fairly straightforward calculation.
\item Find the global syzygies corresponding to the Zamolodchikov tetrahedron, and compute the induced infinitesimal syzygies.
\item Check that global syzygies induce all of the infinitesimal syzygies, confirming that the PVH Criterion is met.
\end{itemize}

\subsection{Terminology and Preliminary Requirements for PVH Criterion}
\label{SectionKoszulTerminol}

We denote by $\qpv$ and $\Q F$ the rational group ring of $\PV$ and the rational free group ring on the same generators, respectively.  Their respective augmentation ideals are denoted $\IK$ and $\IF$.

Consistent with the discussion in Subsection \ref{EliminatingLinearRelations}, we take $\qpv$ to be completed with respect to the filtration by powers of the augmentation ideal, so that we can eliminate the inverses of group generators, and the linear relations corresponding to the group laws, from our presentation for $\qpv$ as an algebra.

We now derive a presentation for $\pv$, even though it was given in substantially similar form in Subsection \ref{SubsecPvBQuadApp}.

Given the presentation for $\PV$ in Section \ref{Intro}, the augmentation ideal $\IK$ is generated as a 2-sided $\qpv$-module by the set $\tX:= \{\bR_{ij}:= (R_{ij} - 1):  1\leq i \ne j \leq n\}$.  It is straightforward to check that the elements of $\tX$ (modulo $\IK^2$) are linearly independent (i.e. $\Q\tX \cap \IK^2 =0$), and hence in fact form a basis of $V=\IK/\IK^2$.  The $\bR_{ij} \text{ mod } \IK^2$ correspond to the generators $\{\rij\}$ for $\pv$ from the presentation (\ref{6Trels}).

From the relations (\ref{relations1}) and (\ref{relations2}) for $\PV$, the ideal $M\subseteq \IF$ of relations for $K$ is the 2-sided ideal $M$ in $F$ generated by
\begin{align*}
\Yijk' & := R_{ij} R_{ik} R_{jk} R_{ij}^{-1} R_{ik}^{-1} R_{jk}^{-1} -1\\
{\Cijkl}' & := R_{ij} R_{kl} R_{ij}^{-1} R_{kl}^{-1} -1
\end{align*}
for $1\leq i,j,k,l \leq n$, and $i,j,k,l$ all distinct.  Equivalently, $M$ is generated (as 2-sided $F$-ideal) by
\begin{align}
\Yijk & := R_{ij} R_{ik} R_{jk} - R_{jk} R_{ik} R_{ij}  \label{relators1}\\
{\Cijkl} & := R_{ij} R_{kl} - R_{kl} R_{ij} \label{relators2}
\end{align}

As per Lemma \ref{RqStructure}, the relations in $\pv$ are generated by $\Rq \cong (M+\IF^3) / \IF^3$.  Thus to obtain $\Rq$, we make the substitution $R_{ij} \mapsto (\bR_{ij} + 1)$ throughout the $\Yijk$ and ${\Cijkl}$, and drop all terms of degree 3 (or higher) in the $\bR_{ij}$.   We obtain the quadratic relators $\{\yijk; \cijkl\}$ for $\pv$ (see (\ref{6Trels}) and (\ref{AlgCommutRels})), up to replacing the $\{\bR_{ij}\}$ by the $\{\rij\}$.

Since $\PV$ is a finitely presented group, the requirements for applicability of the PVH Criterion, as we have developed it, essentially reduce to (see Subsection \ref{AlgSetting}) checking that the ideal of relations $M \subseteq \IF$ actually satisfies $M \subseteq \IF^2$. In turn this amounts to checking that the relators obtained above for $\Rq$ (i.e. (\ref{6Trels})) are all quadratic in the $\rij$, which is clearly true. We note that this is essentially due to the fact that the relators for $\PV$ all have degree 0 in each of the generators of $\PV$, so that after performing the substitution $R_{ij} \mapsto (\bR_{ij} + 1)$ and expanding in terms of the $\bR_{ij}$, all constant terms and terms linear in the $\bR_{ij}$ cancel out.

We can also easily check that $\PV$ satisfies the degree 2 PVH Criterion, i.e. that the generators (\ref{relators1}) and (\ref{relators2}) for $M$ satisfy the requirement that $\{\Yijk + \IF^3, {\Cijkl} + \IF^3\}$ are linearly independent in $(M+\IF^3)/\IF^3$.  Equivalently we have to check that the $\{\yijk; \cijkl\}$ are linearly independent.  There are several ways to do this - one slightly fancy way to do it is to use the isomorphism
\begin{equation*}
\Diti: {\pvqii}^* \overset{\sim}{\rightarrow} \ \Rq
\end{equation*}
which we recalled in (\ref{RKq2Iso}), and note that $\Diti$ takes a basis of $\pvqii$ precisely to the relators $\{\yijk; \cijkl\}$ of $\pv$ (we compute this in (\ref{yijk}) and (\ref{cijkl}) below).

\subsection{Finding the Infinitesimal Syzygies}

As a preliminary matter we recall the definition of $\pvq$ and exhibit its relations.  As noted in Subsection \ref{TwoCanonAssGradedAlg}, $\pv$ is defined as $\pv= T V / \langle \Rq \rangle$ (where $V = \IK / \IK^2$ and $R$ were obtained in Subsection \ref{SectionKoszulTerminol}).  The quadratic dual algebra $\pvq$ is defined as $\pvq:= TV^* / \langle \Rqp \rangle$, where $V^*$ is the linear dual vector space and $\Rqp \subseteq V^*\otimes V^*$ is the annihilator of $\Rq$.

From these definitions, one readily finds that $\pvq$ is the exterior algebra generated by the set $\{\rijq: 1 \leq i \ne j \leq n\}$, subject to the relations:

\begin{align}
\rijq \wedge \rikq & = \rijq \wedge \rjkq - \rikq \wedge \rkjq \label{vrel} \\
\rikq \wedge \rjkq & = \rijq \wedge \rjkq - \rjiq \wedge \rikq \label{arel} \\
\rijq \wedge \rjiq & = 0 \label{qAS}
\end{align}
where the indices $i,j,k$ are all distinct.

\subsubsection{A Basis for $\pv^!$}
\label{sectionBasis}

In this section we will identify a basis for the dual algebra $\pv^!$.  We state the result for all degrees, although we only actually need $\pvqiii$.

We note that monomials in $\pvq$ may be interpreted as directed graphs, with vertices given by the integers $[n] := \{1, \dots , n\}$, and edges consisting of all ordered pairs $(i,j)$ such that $\rij$ is in the monomial.\footnote{\label{rijfootnote}Strictly speaking $\pvq$ is generated by dual generators $\{\rij^*\}$.  To simplify the notation, we write $\rij$ instead of $\rij^*$.}   We thus get a graphical depiction of the above relations:

\[
\xy
(-20,0)*{
\xy
(0,0)*{i}="1";
(-5,5)*{j}="2";
(5,5)*{\ k}="3";
{\ar "1";"2"}; {\ar "1";"3"};
\endxy}="1";
(-10,0)*{=};
(0,0)*{
\xy
(0,0)*{i}="1";
(-5,5)*{j}="2";(-6,6)*{};
(5,5)*{\ k}="3";
{\ar "1";"2"}; {\ar "2";"3"};
\endxy}="2";
(10,0)*{-};
(20,0)*{
\xy
(0,0)*{i}="1";
(-5,5)*{j}="2";
(5,5)*{\ k}="3";(7,6)*{};
{\ar "1";"3"}; {\ar "3";"2"};
\endxy}="3";
\tag{Pruning V}
\endxy
\]

\[
\xy
(-20,0)*{
\xy
(0,0)*{k}="1";
(-5,-5)*{i}="2";
(5,-5)*{j}="3";
{\ar "2";"1"}; {\ar "3";"1"};
\endxy}="1";
(-10,0)*{=};
(0,3)*{
\xy
(0,0)*{k}="1";
(-5,-5)*{i}="2";(-6,6)*{};
(5,-5)*{j}="3";
{\ar "2";"3"}; {\ar "3";"1"};
\endxy}="2";
(10,0)*{-};
(20,3)*{
\xy
(0,0)*{k}="1";
(-5,-5)*{i}="2";
(5,-5)*{j}="3";(7,6)*{};
{\ar "3";"2"}; {\ar "2";"1"};
\endxy}="3";
\tag{Pruning A}
\endxy
\]

\[
\xy
(-20,0)*{\xy
(-5,0)*{i};(5,0)*{j};
(-3,1)*{}="1";(3,1)*{}="2";
(-3,-1)*{}="3";(3,-1)*{}="4";
{\ar "1";"2"}; {\ar "4";"3"};
\endxy};
(0,0)*{=}; (10,0)*{0};
\tag{No Loop}
\endxy
\]

We note that there is a sign indeterminacy in the graphs, in that for instance the LHS of (Pruning V) can equally refer to $\pm \rij \wedge \rik$.  We will only use the graphs when the signs are immaterial.

\newpage

\begin{theorem}
\label{QuadDualBasis}
The algebra $\pvq$ has a basis consisting exactly of the monomials corresponding to `chain gangs', i.e. unordered partitions of $[n]$ into ordered subsets.
\end{theorem}

\begin{corollary}
The degree $k$ component of $\pvq$ has dimension $L(n,n-k)$, where the `Lah number' $L(n,n-k)$ denotes the number of unordered partitions of $[n]$ into $(n-k)$ ordered subsets.
\label{LahBasis}
\end{corollary}
\begin{proof}
Clear from the theorem, since it is easy to see that a chain gang on the index set $[n]$ with $(n-k)$ chains must have exactly $k$ arrows (and correspond to a basis monomial of degree $k$).
\end{proof}

We note that it was already proved in \cite{BEER} that the dimensions of the graded components of $\pvq$ are given by the Lah numbers (although a basis for $\pvq$ was not provided).

We postpone the proof of the theorem until Subsection \ref{PFBasis}.  However, the idea of the proof is straightforward, i.e. show that a basis is given by all monomials whose graphical representation has no \pmb{A-joins} or \pmb{V-joins} (by which we mean the diagrams in the LHS of the relations (Pruning A) and (Pruning V), respectvely) and no loops:

\begin{itemize}
\item One first shows that if a tree has an A-join or a V-join, we can replace it by a sum of trees in which the particular join is replaced by an oriented segment of length 2, using either (Pruning A) or (Pruning V).  Eventually we are left with a sum of oriented chains.
\item One must then show that these oriented chains are linearly independent.
\item Next one shows that all monomials whose graph contains a loop (oriented or not) are $0$:  it turns out that loops of length greater than 2 can be reduced progressively to loops of length 2, and then the resulting graph is $0$ either by (No Loops) or by anti-commutativity.
\end{itemize}

\begin{remark} We will see that directed chains of length 3 are in a 1-1 correspondence with certain (level 3) syzygies of the global algebra $\A$ - specifically one Zamolodchikov tetrahedron for each ordering of a particular choice of 4 of the $n$ strands in $\PV$. An arrow from index `$i$' to index `$j$' means strand `$i$' remains above strand `$j$' throughout the syzygy.  Although not relevant for our purposes, this correspondence between oriented chains of length $m$ and level $m$ syzygies holds for syzygies of all levels. These higher level global syzygies correspond to generalizations of the Zamolodchikov tetrahedron, and most likely correspond in some sense to generators of the cohomology of $\PV$. \label{RkCohomol}
\end{remark}

\begin{remark} If in the basis given above one includes only generators $\rij$ with $i<j$, we obtain a basis for the algebra $\pfq$.  This basis is different from the basis given in \cite{BEER}.  The basis given here is more useful for purposes of applying the PVH criterion, because of the fact that directed chains of length 3 correspond to syzygies arising from the Zamolodchikov tetrahedron.
\end{remark}

\begin{remark} If, as in the previous remark, we again consider the implied basis for $\pfq$, we see that the (No Loop) relation and the exclusion of monomials whose graph contains a loop are irrelevant.  We are left with a rule that says that a basis of $\pfq$ is given by all monomials whose graph does not contain an A-join or a V-join.  The exclusion of A-joins and V-joins amounts to specifying a quadratic Gr\"obner basis for the ideal of relations in $\pfq$.  By a theorem of \cite{Yuz}, this gives a proof that the algebra $\pfq$ (and its dual $\pf$) is Koszul.  Unfortunately the given basis for $\pvq$, as opposed to $\pfq$, does not prove Koszulness, since the no-loop exclusion corresponds to Gr\"obner basis elements of arbitrarily high degree (i.e. of degree equal to the length of the loop).  In subsection \ref{ProofOfKoszulness} we give an alternative basis for $\pvq$, from which the Koszulness of $\pvq$ can be deduced.
\label{RkReNonKoszulness}
\end{remark}

\subsubsection{The Infinitesimal Syzygies}
\label{sectionInfSyz}

One can readily compute that the isomorphism $\Diti$ acts on basis elements of ${\pvqii}^*$ as follows:\footnote{Instead of writing $\rij^{**}$ for generators of $\mathfrak{pvb}_n^{!*}$, we write $\rij$.}

\begin{align}
\Diti: \ & \rij\w \rjk \mapsto \yijk = [\rij,\rik] + [\rij,\rjk] + [\rik,\rjk] \label{yijk} \\
& \rij \w \rkl \mapsto \cijkl = [\rij,\rkl] \label{cijkl}
\end{align}
Indeed one can in fact view (Pruning A) and (Pruning V) as giving the only elements of $V^*$ that do not multiply freely in $\pvqii$, and then (since $\Diti$ is dual to the product in $\pvqii$) the above result is immediate.

As noted following (\ref{X31}) and (\ref{X32}), the map $(1 \negthinspace  \otimes \Diti) \circ \Ditii = (\Diti \negthinspace \otimes 1) \circ \Diiti$ is an isomorphism $A^{!3*} \overset{\sim}{\rightarrow} (\Rq \negthinspace  \otimes \negthinspace  V \cap V \negthinspace  \otimes \negthinspace  \Rq)$.  From (\ref{DResolution}) and the discussion thereafter, we see that the space

\begin{equation}
(\Diti \negthinspace \otimes 1) \circ \Diiti ({\pvqiii}^*) \oplus (-1)(1 \negthinspace  \otimes \Diti) \circ \Ditii ({\pvqiii}^*)
\label{X31minus32}
\end{equation}
gives us the degree 3 infinitesimal syzygies, which in turn generate all the infinitesimal syzygies.

It is easy to see that there are three types of basis element in ${\pvqiii}^*$, corresponding to three types of chain gang with three edges:
\begin{itemize}
\item $\rij\w\rjk\w\rkl$ with $i,j,k,l$ all distinct;
\item $\rij\w\rjk\w\rst$ with $i,j,k,s,t$ all distinct;
\item $\rij\w\rkl\w\rst$ with $i,j,k,l,s,t$ all distinct.
\end{itemize}

\newpage

We first deal with the first type of basis element.  We will show that in this case the first component of (\ref{X31minus32}) is given by:

\begin{align}
\label{RightSyzygy}
(\Diti \otimes 1) \circ \Diiti (\rij\w\rjk\w\rkl) & = -(\Diti \otimes 1)(\rij\w\rjk) \otimes (-\ril-\rjl-\rkl) \\
& + (\Diti \otimes 1)(\rij\w\rjl) \otimes (-\rik-\rjk+\rkl)  \notag \\
& -(\Diti \otimes 1)(\rik\w\rkl) \otimes (-\rij+\rjk+\rjl)  \notag \\
& + (\Diti \otimes 1)(\rjk\w\rkl) \otimes(\rij+\rik+\ril)  \notag \\
& - (\Diti \otimes 1)(\rij\w\rkl) \otimes (\rik+\ril+\rjk+\rjl)  \notag \\
& + (\Diti \otimes 1)(\rik\w\rjl) \otimes(\rij+\ril-\rjk+\rkl)  \notag \\
& + (\Diti \otimes 1)(\ril\w\rjk) \otimes(-\rij-\rik+\rjl+\rkl) \notag
\end{align}
We defer a more detailed justification of the above calculation to Subsection \ref{ProductCalcs}.  Now using (\ref{yijk}) and (\ref{cijkl}), we get:

\begin{align}
(\Diti \otimes 1) \circ \Diiti & (\rij\w\rjk\w\rkl)  = \label{rijkl} \\
& -\yijk \otimes (-\ril-\rjl-\rkl) + \yijl \otimes (-\rik-\rjk+\rkl) \notag \\
& -\yikl \otimes (-\rij+\rjk+\rjl) + \yjkl \otimes(\rij+\rik+\ril) \notag \\
& - \cijkl \otimes (\rik+\ril+\rjk+\rjl) + \cikjl \otimes(\rij+\ril-\rjk+\rkl) \notag \\
& - \ciljk \otimes(\rij+\rik-\rjl-\rkl) \notag
\end{align}

In (\ref{rijkl}) the tensor products are the tensor products in the tensor algebra $TV$, so we drop them.  $(1\otimes \Diti) \circ \Ditii (\rij\w\rjk\w\rkl)$ is the same, but with the tensor components flipped.

We will see that the resulting infinitesimal syzygies are induced (via the map $\FSyz$) from global syzygies in the next subsection.

This leaves the two remaining types of degree 3 basis element in ${\pvqiii}^*$.  It is fairly straightforward to compute that they correspond, respectively, to the relations:
\begin{equation*}
\yijk \rst = \rst \yijk
\end{equation*}
and
\begin{equation*}
\cijkl \rst = \rst \cijkl
\end{equation*}
which are clearly satisfied also at the global level.

\subsection{Global Syzygies and the PVH Criterion}

\subsubsection{The Global Syzygies}
\label{sectionGlobSyz}

We now display (a sum of) elements of $\Rg_{\geq 3}$ which specify a syzygy of $K=\Q\PV$ and which project via $\FSyz_3$ (see notation in Theorem \ref{HutchingsCriterion}) to the infinitesimal syzygy arising from (\ref{rijkl}) in $\pv$.  The syzygy corresponds to the standard syzygy in the braid group (i.e. the Zamolodchikov tetrahedron pictured in Subsection \ref{HutchingsUsefulness}), which can be written:

\begin{multline*}
Y_{jkl} R_{il} R_{ik} R_{ij} + R_{jk} R_{jl} Y_{ikl} R_{ij} + R_{jk} R_{jl} R_{ik} R_{il} C_{ij}^{kl} \\
+ R_{jk} C_{ik}^{jl} R_{il} R_{ij} R_{kl} + R_{jk} R_{ik} Y_{ijl} R_{kl} + Y_{ijk} R_{il} R_{jl} R_{kl} + R_{ij} R_{ik} C_{il}^{jk} R_{jl}R_{kl} \\
-R_{ij} R_{ik} R_{il} Y_{jkl} - R_{ij} Y_{ikl} R_{jl} R_{jk} - C_{ij}^{kl} R_{il} R_{ik} R_{jl} R_{jk} \\
-R_{kl} R_{ij} R_{il} C_{ik}^{jl} R_{jk} - R_{kl} Y_{ijl} R_{ik} R_{jk} - R_{kl} R_{jl} R_{il} Y_{ijk} - R_{kl} R_{jl} C_{il}^{jk} R_{ik} R_{ij}
\end{multline*}
where by abuse of notation we have used the same notation $\Yijk$ and $\Cijkl$ as in (\ref{relators1}) and (\ref{relators2}), but strictly speaking $\YF = \{\Yijk, \Cijkl\}$ is just a set of symbols which generate the relators in $K$ via the map $\delK$:
\begin{align}
\Yijk & \mapsto R_{ij} R_{ik} R_{jk} - R_{jk} R_{ik} R_{ij} \label{YandC} \\
{\Cijkl} & \mapsto R_{ij} R_{kl} - R_{kl} R_{ij} \notag
\end{align}

The calculation may be explained as follows.  The illustration in Subsection \ref{HutchingsUsefulness} shows 14 braids $\{B_i\}_{i=1,\dots, 14}$ around its perimeter.  These are linked by arrows labeled by various multiples of the $\Yijk$ and $\Cijkl$.  If we attach the labels $B_1, B_2, \dots$ starting at the bottom braid and proceeding clockwise around the perimeter, the arrows correspond to differences $(B_2-B_1), \dots, (B_8-B_7)$ up the left side of the diagram, and to differences $(B_{14}-B_1), \dots, (B_8-B_9)$ around the right side.  If we label the differences in accordance with (\ref{YandC}), we get the labeling in the picture. But clearly the telescopic sums on the left and right both give $B_8-B_1$, so we get a syzygy which we wrote down above.

This syzygy induces an infinitesimal syzygy which we obtain by substituting $R_{ij} \mapsto (\overline{R}_{ij} + 1)$ and dropping all but the lowest degree terms in the $\overline{R}_{ij}$ (this corresponds to applying the map $\FSyz_3$).  After reorganizing, we get the following element of $ker\ \delA$:

\begin{equation*}
Zam^R - Zam^L
\end{equation*}
where $Zam^R \in \YF \tX$, $Zam^L\in \tX \YF$, and

\begin{multline*}
Zam^R = Y_{jkl} (\bR_{ij}+ \bR_{ik}+ \bR_{il}) - Y_{ikl} (- \bR_{ij} + \bR_{jk} +\bR_{jl})  + Y_{ijl} (-\bR_{ik} -\bR_{jk} + \bR_{kl}) \\
- Y_{ijk} (-\bR_{il}- \bR_{jl}- \bR_{kl}) \\
- C_{ij}^{kl} (\bR_{ik}+ \bR_{il} + \bR_{jk}+ \bR_{jl}) + C_{ik}^{jl} (\bR_{ij} + \bR_{il} -\bR_{jk} + \bR_{kl}) \\
- C_{il}^{jk} (\bR_{ij}+ \bR_{ik}-  \bR_{jl}-  \bR_{kl})
\end{multline*}

and $Zam^L$ is the same except with the order of the products reversed.

Subject to identifying $T\tX$ and $TV$ (where $V=\Q\{\rij\}$) by means of the map $\bR_{ij} \mapsto \rij$, the map $\delA$ sends the $Y_{ijk}$ to the $\yijk$ in (\ref{6Trels}) and the $C_{ij}^{kl}$ to $\cijkl$ in (\ref{AlgCommutRels}).  By inspection, we see that

\begin{equation*}
(\delA (Zam^R),-\delA (Zam^L)) \in R\otimes V \oplus V\otimes R
\end{equation*}
coincides with the degree 3 infinitesimal syzygy in the space (\ref{X31minus32}) corresponding to (\ref{rijkl}).  Hence we have confirmed that all of the infinitesimal syzygies are covered by global syzygies.

\subsection{Proof of the Basis for $\pvq$}
\label{PFBasis}

We will follow the outline of the proof provided in Subsection \ref{sectionBasis}.

We will say that a pair of vertices in a forest graph is \pmb{unordered} if there is not an oriented sequence of edges from one of them to the other.  We define the \pmb{defect} of a tree as the number of unordered pairs of vertices in the graph, and the defect of a forest as the sum of the defects of its components.

Then chain gangs (unordered partitions of $[n]$ into ordered subsets) are exactly the forests with $0$ defect.  Moreover, in the pruning moves the A- and V-joins have defect 1, while the remaining terms have defect 0.

We will refer to a relation formed by adding to each of the terms in either (Pruning A) or (Pruning V) exactly the same additional edges, without ever forming a loop, as a \pmb{multiple} of the original relation.  Note that the graphs representing the multiple need not be connected.  The defect function has the following `multiplicativity' property on forests:

\begin{lemma}
In each multiple of either (Pruning A) or (Pruning V), the term which is built from the term in the original relation containing a join has defect strictly larger than the other terms.
\label{multiplicativity}
\end{lemma}

The proof is deferred to the end of this subsection.

\begin{proof}[Proof of Theorem \ref{QuadDualBasis}]

We follow the plan of proof given following the statement of Theorem \ref{QuadDualBasis}.

The multiplicativity property of the defect makes clear that all forests can be expressed in terms of (sums of) chain gangs:  If a forest contains an A- or V-join, then using either (Pruning A) or (Pruning V) we can replace it with a sum of forests with strictly lower defect.  Iterating, we get a forest with 0 defect, i.e. a chain gang.

Now we show that chain gangs are linearly independent modulo the relations in $\pvq$.  The proof is a variation on the standard diamond lemma proof, which we briefly recall.  By a \pmb{reduction}, we mean specifically replacing the LHS of either pruning relation, or a multiple thereof, by the RHS.

We will show that reducing a defect to 0 will produce the same chain gangs regardless of the sequence of reductions chosen, by induction on the size of the defect.  This is clearly true when we start with a forest with defect 1, since there is only one way to reduce such a forest.

Suppose the claim is true for all forests of defect $\leq m$.  Let us consider a forest of defect $m+1$, and suppose there are two possible reductions, called (a) and (b).  Then applying either (a) or (b) gives a (sum of) new forests, which we call A and B respectively, each of defect $\leq m$.

Suppose (a) and (b) (or the pruning relations of which they are multiples) involve changes to pairs of edges that do not overlap.  Then it is still possible to apply reduction (a) to B, and reduction (b) to A.  Doing so, we obtain the same forest C of defect $\leq m-1$, since the result of applying non-overlapping reductions clearly does not depend on the order they are applied.

Alternatively, suppose (a) and (b) (or the pruning relations of which they are multiples) involve changes to pairs of edges that do overlap.  We will see that we can find further reductions (a'), (a'') and (b'), (b'') such that applying the sequence (a)-(a')-(a'') or (b)-(b')-(b'') leads to the same (sum of) forests C, of defect $\leq m-2$.\footnote{In fact the reductions (a') and (b') may really involve two reductions, applicable to different terms.}

Either way, we know by induction that all reduction sequences from A give the same results, and similarly for B, and since they have a common reduction sequence going through C, we see that A and B both give the same (sum of) forests of defect 0.  Hence all reductions of the original forest must give the same (sum of) chain gangs.

We now deal with the case where reductions (a) and (b) involve pairs of edges that overlap, and exhibit the reductions (a'), (a'') and (b'), (b'').  By inspection of the A- and V-joins, the following three types of overlap can arise (up to sign):

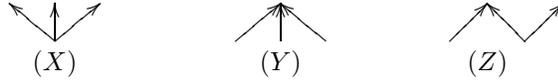
\[
\xy
(-30,0)*{\xy
(0,0)*{}="1";
(-6,5)*{}="2";
(0,5)*{}="3";
(6,5)*{}="4";
(0,-3)*{(X)};
{\ar "1";"2"}; {\ar "1";"3"}; {\ar "1";"4"};
\endxy}="1";
(0,0)*{%
\xy
(0,0)*{}="1";
(-6,-5)*{}="2";
(0,-5)*{}="3";
(6,-5)*{}="4";
(0,-8)*{(Y)};
{\ar "2";"1"}; {\ar "3";"1"}; {\ar "4";"1"};
\endxy}="2";
(30,0)*{%
\xy
(-5,-5)*{}="1";
(0,0)*{}="2";
(5,-5)*{}="3";
(10,0)*{}="4";
(0,-8)*{(Z)};
{\ar "1";"2"}; {\ar "3";"2"}; {\ar "3";"4"};
\endxy}="3";
\endxy
\]
In each case we have only shown the edges involved in the reductions.

Case (Z) is dealt with as follows (a star over a wedge $\overset{*}{\wedge}$ indicates the join which is being reduced - hence to make the following more legible we have dropped the ${}^*$ from elements $\rijq \in \pvq$):

\begin{align*}
\rij \sw \rkj \w \rkl & = \rik \w \rkj \sw \rkl + \rij \w \rki \sw \rkl \\
& = \rik \w \rkj \w \rjl - \rik \w \rkl \w \rlj - \rij \sw \ril \w \rki + \rkl \w \rli \w \rij \\
& = \rik \w \rkj \w \rjl - \rik \w \rkl \w \rlj - \rki\w\rij\w\rjl +\rki\w\ril\w\rlj \\
& \quad \quad \quad - \rkl \w \rli \w \rij
\end{align*}
while on the other hand
\begin{align*}
\rij \w \rkj \sw \rkl & = \rik\w\rkj\sw\rkl +\rki\sw\rkl\w\rij \\
& =  \rik\w\rkj\w\rjl -\rik\w\rkl\w\rlj + \rki\w\ril\sw\rij - \rkl\w\rli\w\rij\\
& = \rik\w\rkj\w\rjl - \rik\w\rkl\w\rlj + \rki\w\ril\w\rlj - \rki\w\rij\w\rjl  \\
& \quad \quad \quad - \rkl\w\rli\w\rij
\end{align*}

Since the results of the two calculations are the same, we see that regardless of which join we reduce first, there is a further sequence of reductions that leads to the same (signed sum of) trees, which is what we needed.\footnote{As per the previous footnote, note that reductions (a') and (b'), indicated by the stars in the RHS of the first lines, actually involve two reductions, applicable to separate terms.}

Cases (X) and (Y) are dealt with similarly - we simply note that

\begin{multline*}
\rij\sw\rik\w\ril = -\rij\w\rjl\w\rlk + \rij\w\rjk\w\rkl + \ril\w\rlj\w\rjk \\
+ \rik\w\rkl\w\rlj - \rik\w\rkj\w\rjl - \ril\w\rlk\w\rkj = \rij\w\rik\sw\ril
\end{multline*}
and
\begin{multline*}
\ril\sw\rjl\w\rkl = \rij\w\rjk\w\rkl - \rik\w\rkj\w\rjl +\rki\w\rij\w\rjl \\
- \rji\w\rik\w\rkl +\rjk\w\rki\w\ril - \rkj\w\rji\w\ril = \ril\w\rjl\sw\rkl
\end{multline*}
and leave the details of the calculations for the reader.

The next step in the proof is to show that all graphs with loops are $0$.  For loops containing a V-join, we can use (Pruning V) to reduce loops of length greater than 2 to (sums of) loops of shorter length:

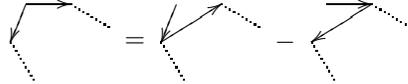
\[
\xy
(-20,0)*{%
\xy
(-2,0)*{}="1";
(0,5)*{}="2";
(6,5)*{}="3";
(1,-5)*{}="4";
(11,2)*{}="5";
{\ar "2";"1"}; {\ar "2";"3"};
{\ar@{.} "1";"4"}; {\ar@{.} "3";"5"};
\endxy}="1";
(-10,0)*{=};
(0,0)*{%
\xy
(-2,0)*{}="1";
(0,5)*{}="2";
(6,5)*{}="3";
(1,-5)*{}="4";
(11,2)*{}="5";
{\ar "2";"1"}; {\ar "1";"3"};
{\ar@{.} "1";"4"}; {\ar@{.} "3";"5"};
\endxy}="2";
(10,0)*{-};
(20,0)*{%
\xy
(-2,0)*{}="1";
(0,5)*{}="2";
(6,5)*{}="3";
(1,-5)*{}="4";
(11,2)*{}="5";
{\ar "2";"3"}; {\ar "3";"1"};
{\ar@{.} "1";"4"}; {\ar@{.} "3";"5"};
\endxy}="3";
\endxy
\]

The case of loops containing an A-join, rather than a V-join, is similar. We continue reducing the length of the loops until we are down to loops of length 2, which are $0$ either by (No Loops) or by anti-commutativity.

Thus, if we follow the above procedure, we can reduce all loops to $0$.  It is also clear from the above that even if we followed a different sequence of pruning moves we would never reduce loops to a sum of diagrams including trees, since a pruning move can never break a loop.
\end{proof}

We have completed the proof of Theorem \ref{QuadDualBasis}, subject to proving multiplicativity of the defect function, i.e. Lemma \ref{multiplicativity}.  We do this now.

By a `\pmb{vertex in a relation}' we will mean a vertex which is an endpoint of at least one edge in the graphs corresponding to the terms of the relation.  It is fairly clear this is a well-defined notion (and in particular that the number of vertices in a relation is constant over all terms in a relation).

By the `\pmb{join term}' in a multiple of a pruning relation, we mean the term that was built by adding edges to the term in the original pruning relation which contained an A- or V-join.

\begin{proof}[Proof of Lemma \ref{multiplicativity}]
We proceed by induction on the number of vertices in a relation.  We claim that if (x) and (y) are vertices in the new relation, and there is a directed chain of edges from (x) to (y) in the join term, then there is also a directed chain of edges from (x) to (y) (in the same direction) in the other terms of the relation.  Hence:

\begin{enumerate}
\item When we form a multiple of (Pruning A) or (Pruning V) by adding edges, in that multiple each vertex is no more ordered (with respect to other vertices) in the join-term than in the non-join terms.
\item However, in each relation, there is at least one pair of vertices which is unordered in the join-term, but is ordered in the other terms, namely the unordered pair in the original pruning relation.
\end{enumerate}
So we can conclude that the join-term in the new relation must have strictly highest defect.

The above claim is easily verified in the original relations (Pruning A) and (Pruning V).  We now assume the claim has been proved whenever there are up to $m$ edges in a relation; we take a relation with $m$ edges and add a further edge. There are three cases.

\pmb{Case I: Two New Vertices.}  If the added edge forms a separate component in the new graphs, then clearly the defect is unchanged in all terms of the relation.

\pmb{Case II:  One Old, One New Vertex.}  So let us suppose that the added edge has one vertex (a) already in the relation, and one new vertex (b).  It is clear that the orderliness of pairs of vertices not including (b) is unchanged.

Now suppose that (c) is any other vertex in the relation.  If (b) and (c) are ordered in the new join term, say with a directed chain from (b) to (c), this chain must go through (a) since vertex (b) was not previously the endpoint of any edge.  Thus there was also a directed chain from (a) to (c) in the join-term of the old relation, hence by induction there were directed chains from (a) to (c) in the non-join terms in the old relation.  It follows that there is also a directed chain from (b) to (c) in the new non-join terms.  The case of a directed chain from (c) to (b) in the new join-term is similar.

If (b) and (c) are unordered in the new join term, there is nothing to prove.  Letting (c) range over all other vertices in the relation proves Case II.

\pmb{Case III:  Two Old Vertices.}  All that is left is to consider the case where the new edge links two existing vertices in the relation.  Because we assume that the added edge does not create a loop, it follows that the edge must be linking two formerly disconnected components of the graphs underlying the relation.  We assume the new edge links existing vertices (a) and (b).  It is clear that the orderliness of pairs of vertices already within the same component in the old relation is unchanged.  So we take (c) and (d) to be to be any two vertices in the component of (a) and (b) respectively. We can assume without loss of generality that either (a) $\ne$ (c) or (b) $\ne$ (d) (because if (a) = (c) and (b) = (d) then that pair is joined by the new edge and hence ordered in all terms of the relation).

The reasoning is similar to Case II.  If (c) and (d) are ordered in the new join term, say with a directed chain from (c) to (d), this chain must go through (a) and (b) since we assume there are no loops.  Thus the new edge must be oriented (a) to (b); moreover, there must also have been directed chains from (c) to (a) and from (b) to (d) in the join-term of the old relation.  By induction, there were directed chains from (c) to (a) and from (b) to (d) in the non-join terms in the old relation.  It follows that there is also a directed chain from (c) to (d) in the new non-join terms.  The case of a directed chain from (d) to (c) in the new join-term is similar.

Finally, if (c) and (d) are unordered in the new join term, there is nothing to prove.  Letting (c) and (d) range over all other vertices in the relation proves Case III.
\end{proof}

\subsection{Justification of the Co-Product Formulas}
\label{ProductCalcs}

We give here a summary of the action of the product $m^!: \pvqii \otimes V^* \rightarrow \pvqiii$ in terms of the `directed chains' basis for the respective spaces.  The verifications are routine and we will leave them to the reader.

\begin{align*}
\rijq\w\rjkq \otimes \rilq & \mapsto \rilq\w\rljq\w\rjkq - \rijq\w\rjlq\w\rlkq + \rijq\w\rjkq\w\rklq \\
\rijq\w\rjkq \otimes \rjlq & \mapsto -\rijq\w\rjlq\w\rlkq + \rijq\w\rjkq\w\rklq \\
\rijq\w\rjkq \otimes \rklq & \mapsto \rijq\w\rjkq\w\rklq \\
\rijq\w\rjkq \otimes \rliq & \mapsto \rliq\w\rijq\w\rjkq \\
\rijq\w\rjkq \otimes \rljq & \mapsto -\rilq\w\rljq\w\rjkq + \rliq\w\rijq\w\rjkq \\
\rijq\w\rjkq \otimes \rlkq & \mapsto \rijq\w\rjlq\w\rlkq - \rilq\w\rljq\w\rjkq + \rliq\w\rijq\w\rjkq
\end{align*}
and
\begin{align*}
\rijq\w\rklq \otimes \rikq & \mapsto -\rijq\w\rjkq\w\rklq + \rikq\w\rkjq\w\rjlq - \rikq\w\rklq\w\rljq   \\
\rijq\w\rklq \otimes \rkiq & \mapsto \rklq\w\rliq\w\rijq - \rkiq\w\rilq\w\rljq + \rkiq\w\rijq\w\rjlq   \\
\rijq\w\rklq \otimes \rilq & \mapsto -\rijq\w\rjkq\w\rklq + \rikq\w\rkjq\w\rjlq - \rikq\w\rklq\w\rljq   \\
& -  \rkiq\w\rijq\w\rjlq + \rkiq\w\rilq\w\rljq\\
\rijq\w\rklq \otimes \rliq & \mapsto \rklq\w\rliq\w\rijq  \\
\rijq\w\rklq \otimes \rjkq & \mapsto -\rijq\w\rjkq\w\rklq  \\
\rijq\w\rklq \otimes \rkjq & \mapsto \rklq\w\rliq\w\rijq - \rkiq\w\rilq\w\rljq + \rkiq\w\rijq\w\rjlq  \\
& + \rikq\w\rklq\w\rljq - \rikq\w\rkjq\w\rjlq \\
\rijq\w\rklq \otimes \rjlq & \mapsto -\rijq\w\rjkq\w\rklq + \rikq\w\rkjq\w\rjlq - \rkiq\w\rijq\w\rjlq  \\
\rijq\w\rklq \otimes \rljq & \mapsto \rklq\w\rliq\w\rijq - \rkiq\w\rilq\w\rljq + \rikq\w\rklq\w\rljq \\
\end{align*}

We leave to the reader the routine but tedious verification that the formula (\ref{RightSyzygy}) follows.

\subsection{Proof of the Koszulness of $\pvq$}
\label{ProofOfKoszulness}

As indicated in Remark \ref{RkReNonKoszulness}, the basis given in Theorem \ref{QuadDualBasis} will not in itself lead to a proof of Koszulness because the explicit exclusion of monomials whose graphical representation contains a loop corresponds to Gr\"obner basis elements of arbitrarily high degree.  In contrast, standard theorems on Gr\"obner bases only tell us that (under mild assumptions) algebras with quadratic Gr\"obner bases are Koszul.

So we will exhibit a different basis for $\pvq$, consisting of all monomials not containing certain length two subwords, which corresponds to the specification of a quadratic Gr\"obner basis for $\pvq$.  We will see that, by a result of Yuzvinsky \cite{Yuz} (see also \cite{ShelYuz}), $\pvq$(and hence also $\pv$) is Koszul.

To begin with, given any finite subset $I \subseteq \N$ (which we order numerically), we will define two kinds of graph with vertices indexed by $I$ - we will call these Down graphs and Up graphs.  We will then show how to combine Down and Up graphs to get graphs (which we will call Up-Down graphs) which correspond to a new basis for $\pvq$, of the desired form (i.e. corresponding to the specification of a quadratic Gr\"obner basis for $\pvq$).

We will also see that the Down and Up graphs, respectively, catalogue bases for:

\begin{itemize}
\item the algebra $\pfq$, which is quadratic dual to the quadratic approximation for the pure flat braid group $\PfB$ ( i.e. the quadratic dual to the universal enveloping algebra of the triangular Lie algebra $\mathfrak{tr_n}$ in \cite{BEER}); and
\item the algebra $\pbq$, quadratic dual to the quadratic approximation for the pure braid group $\PB$.
\end{itemize}
However, we do not have a coherent explanation for why bases for $\pfq$ and $\pbq$ should fit together in this way to produce bases of $\pvq$.  See Remark \ref{RemarkPVBasProduct} below.

\subsubsection{Down Graphs and $\pfq$}

A \pmb{Down tree} on the index set $I = \{i_1, \dots ,i_m \} \subseteq \N$ (with smallest index $i_1$) consists of a `tuft' of directed edges $\{(i_2,i_1), \dots, (i_m,i_1) \}$.  (The graph is non-planar in that all orderings of the edges incident to a particular vertex are considered equivalent.)  This corresponds to allowing all trees built with directed edges with \emph{decreasing} indices (i.e. edges $(i,j)$ with $i>j$) by

\begin{itemize}
\item allowing the following subgraphs:

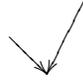
\[
\xy
(0,0)*{}="1";
(-5,5)*{}="2";
(5,10)*{}="3";
{\ar "2";"1"}; {\ar "3";"1"};
\endxy
\]

\item excluding the following three subgraphs:

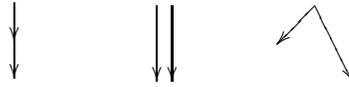
\[
\xy
(-20,0)*{
\xy
(0,10)*{}="1";
(0,5)*{}="2";
(0,0)*{}="3";
{\ar "1";"2"}; {\ar "2";"3"};
\endxy};
(0,0)*{
\xy
(-1,0)*{}="1";
(-1,-10)*{}="2";
(1,0)*{}="3";
(1,-10)*{}="4";
{\ar "1";"2"};{\ar "3";"4"};
\endxy};
(20,0)*{
\xy
(0,0)*{}="1";
(-5,-5)*{}="2";
(5,-10)*{}="3";
{\ar "1";"2"}; {\ar "1";"3"};
\endxy};
\endxy
\]
\end{itemize}
where in all cases the relative heights of the endpoints indicate the relative ordering of the indices (in particular, the middle subgraph has a doubled edge: $\{(i,j),(i,j)\}$).  We declare by way of convention that a Down tree on on index set with one element is the empty graph.  An example of a Down graph is the following:

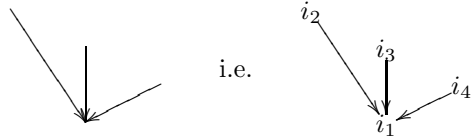
\[
\xy
(-20,0)*{
\xy
(0,0)*{}="1";
(-10,15)*{}="2";
(0,10)*{}="3";
(10,5)*{}="4";
{\ar "2";"1"};{\ar "3";"1"}; {\ar "4";"1"};
\endxy};
(0,0)*{\text{i.e.}};
(20,0)*{\xy
(0,0)*{i_1}="1";
(-10,15)*{i_2}="2";
(0,10)*{i_3}="3";
(10,5)*{i_4}="4";
{\ar "2";"1"};{\ar "3";"1"}; {\ar "4";"1"};
\endxy};
\endxy
\]
where $i_2>i_3>i_4>i_1$.

Note that because of the last two types of excluded graph, we needn't have explicitly restricted ourselves to trees, as these exclusions prevent the formation of (ordered or unordered) loops in the graph  (recall also that Down graphs are built only with directed edges with decreasing indices).  Thus we have an exclusion rule, of degree 2 in the number of edges, which effectively eliminates loops.  In particular the obstacle to proving Koszulness due to the presence of non-quadratic Gr\"obner basis elements no longer arises.

We now define a Down forest on an index set $I$ partitioned as $I=S_1 \sqcup \dots \sqcup S_u$ to be the union of the Down trees on the subsets $S_i$.

\begin{remark} The monomials corresponding to Down forests induced by unordered partitions of $[n]=\{1,2, \dots, n\}$ (using the correspondence explained in subsection \ref{sectionBasis}) form a basis for the algebra $\pfq$ (as a skew-commutative algebra \footnote{See \cite{Mikha} for more on such bases.}).  Indeed, it is easy to see that Down forests are in bijective correspondence with the `reduced monomials with disjoint supports' which were proved in \cite{BEER}, Proposition 4.2, to form a basis of $\pfq = U(\mathfrak{tr_n})^!$ (with the minor difference that the edges in \cite{BEER} had increasing indices).  Also, the above excluded subgraphs correspond to the excluded monomials implied by the Gr\"obner basis given in \cite{BEER}, Corollary 4.3, for $U(\mathfrak{tr_n})^!$ (subject to always writing generators with increasing indices, using the relation $r_{i,j}=-r_{j,i}$).  The fact that these Gr\"obner basis elements are quadratic allowed \cite{BEER} to conclude that $\pfq$ is Koszul.
\label{Rkpfbbasis}
\end{remark}

\subsubsection{Up Graphs and $\pbq$}

An \pmb{Up tree} on the index set $I = \{i_1, \dots ,i_m \} \subseteq \N$ (with $i_1 < \dots < i_m$) consists of all trees built with directed edges with \emph{increasing} indices by

\begin{itemize}
\item allowing the following two subgraphs:

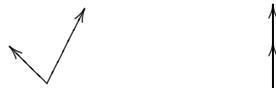
\[
\xy
(-15,0)*{
\xy
(0,0)*{}="1";
(-5,5)*{}="2";
(5,10)*{}="3";
{\ar "1";"2"}; {\ar "1";"3"};
\endxy};
(15,0)*{
\xy
(0,11)*{}="1";
(0,6)*{}="2";
(0,0)*{}="3";
{\ar "2";"1"}; {\ar "3";"2"};
\endxy};
\endxy
\]

\item excluding the following two subgraphs:

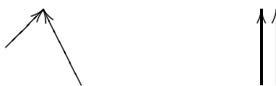
\[
\xy
(-15,0)*{
\xy
(0,0)*{}="1";
(-5,-5)*{}="2";
(5,-10)*{}="3";
{\ar "2";"1"}; {\ar "3";"1"};
\endxy};
(15,0)*{
\xy
(-1,0)*{}="1";
(-1,-10)*{}="2";
(1,0)*{}="3";
(1,-10)*{}="4";
{\ar "2";"1"};{\ar "4";"3"};
\endxy};
\endxy
\]
\end{itemize}
where, again, in all cases the relative heights of the endpoints indicate relative ordering of the indices. Furthermore, the graphs are again non-planar in that all orderings of the edges incident to a particular vertex are considered equivalent; also, we again declare by way of convention that a Up tree on an index set with one element is the empty graph.  An example of an Up tree is the following:

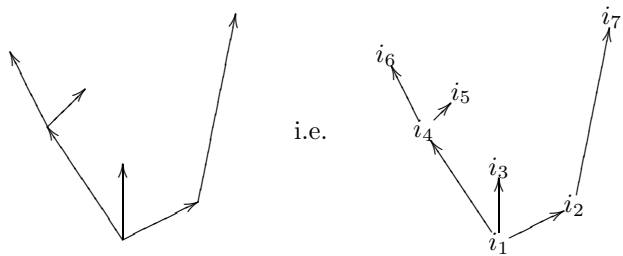
\[
\xy
(-25,0)*{
\xy
(0,0)*{}="1";
(-10,15)*{}="2";
(0,10)*{}="3";
(10,5)*{}="4";
(-15,25)*{}="5";
(-5,20)*{}="6";
(15,30)*{}="7";
{\ar "1";"2"};{\ar "1";"3"}; {\ar "1";"4"};{\ar "2";"5"};{\ar "2";"6"};{\ar "4";"7"};
\endxy};
(0,0)*{\text{i.e.}};
(25,0)*{
\xy
(0,0)*{i_1}="1";
(-10,15)*{i_4}="2";
(0,10)*{i_3}="3";
(10,5)*{i_2}="4";
(-15,25)*{i_6}="5";
(-5,20)*{i_5}="6";
(15,30)*{i_7}="7";
{\ar "1";"2"};{\ar "1";"3"}; {\ar "1";"4"};{\ar "2";"5"};{\ar "2";"6"};{\ar "4";"7"};
\endxy};
\endxy
\]
where $i_1< \dots <i_7$.

As with Down graphs we needn't have explicitly restricted ourselves to trees, since one effect of the excluded subgraphs is to prevent the formation of (ordered or unordered) loops in the graph.  Again, the obstacle to proving Koszulness due to the presence of non-quadratic Gr\"obner basis elements has been avoided.

We now define an Up forest as a union of Up trees (with disjoint index sets).

\begin{proposition}
The Up trees on a given index set $I$ with $m$ elements (in which all indices belong to at least one edge) are in bijective correspondence with the cyclic orderings of $[m]=\{1, \dots, m\}$, or equivalently the orderings of $I$ starting with the smallest index.  This number is clearly $(m-1)!$.
\end{proposition}
\begin{proof}
It is fairly easy to see that Up trees are what is called `recursive' - i.e. non-planar rooted trees with vertices labeled by distinct numbers, where the labels are strictly increasing as move in the direction of the arrows. It is a classical result that there are $(m-1)!$ of these on an index set of size $m$.  One way to see it is to place the root at the bottom of the picture with the edges pointing up, and order the children of each node by increasing size toward the left (we can do this since the trees are non-planar, i.e. the children of each node are unordered): see the sample Up tree above.  Now thicken all the edges into ribbons (which are kept flat to the plane with no twisting).  Finally, starting at the root go along the outside edge of the ribbon graph in a clockwise direction writing down each index the first time it is reached.  The result is an ordering of the $m$ indices starting with the smallest (in the case of the sample Up tree above we get $(i_1,i_4,i_6,i_5,i_3,i_2,i_7$), and there are clearly $(m-1)!$ of these.  It is easy to see that this procedure gives the required bijection.
\end{proof}
\begin{corollary}  The Up forests on an index set $I$ are in bijective correspondence with the unordered partitions of $I$ into cyclically ordered subsets.
\end{corollary}

\begin{remark} The monomials corresponding to Up forests induced by unordered partitions of $[n]=\{1,2, \dots, n\}$) into cyclically ordered subsets form a basis for the algebra $\pbq$.  Indeed, it is easy to see that Up forests are in bijective correspondence with the basis elements for $\pbq$ given in \cite{Yuz}, see also \cite{Arnold} and \cite{ShelYuz}.  Also, the above excluded subgraphs correspond to the excluded monomials implied by the Gr\"obner basis given in \cite{Yuz}.  The fact that these Gr\"obner basis elements are quadratic allowed \cite{ShelYuz} to conclude that $\pbq$ is Koszul.
\label{Rkpbbasis}
\end{remark}

\subsubsection{Up-Down Graphs and $\pvq$}

To define Up-Down graphs we first need the concept of an ordered 2-step partition (essentially due to \cite{BEER}\footnote{See the proof of Corollary 4.6, (iii).  Our \emph{ordered} 2-step partitions differ from their `2-step partitions' in that our underlying sets $S_i$ are cyclically ordered and theirs are unordered.}).  Namely given $n\in \N$ and $[n]:=\{1,2,\dots,n\}$, first take an unordered partition of $[n]$ as $[n]=S_1 \sqcup \dots \sqcup S_l$ where the sets $S_i$ are cyclically ordered (and let $m_i$ denote the minimal element of $S_i$).  Second, take an unordered partition of the set $\mathcal{M}:=\{m_i: i=1, \dots, l\}$ of minimal elements into distinct unordered subsets, $\mathcal{M}=M_1 \sqcup \dots \sqcup M_k$.

Now suppose given such an ordered 2-step partition of $[n]$.  First we form the Down forest on the index set $\mathcal{M}$ with the given partition.  Next we form the Up tree on each of the (cyclically ordered) sets $S_i$.  The resulting graph is called an \pmb{Up-Down forest} on the index set $[n]$.  An Up-Down forest is uniquely determined by a particular ordered 2-step partition, and conversely.

\begin{theorem}
\label{ThmKoszulBasis}
The monomials corresponding to Up-Down graphs on ordered 2-step partitions of $[n]=\{1,\dots ,n\}$ form a basis of $\pvq$.
\end{theorem}

This theorem will follow from the following:

\begin{proposition}
\label{PropKoszulBasis}
The algebra $\pvq$ has a basis consisting of all monomials whose graph does not contain any of the following as subgraphs (again the relative heights of the endpoints indicate relative ordering of the indices, and the graphs are non-planar, so that the all edges (incoming or outgoing) incident to a particular vertex may be represented in any order without changing the graph):

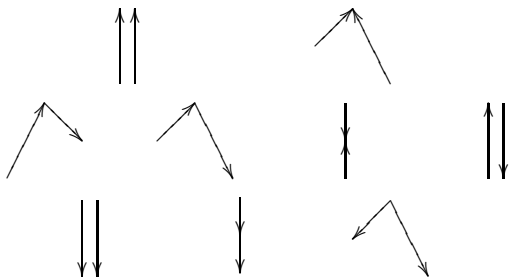
\[
\xy
(0,0)*{
\xy
(-1,0)*{}="1";
(-1,-10)*{}="2";
(1,0)*{}="3";
(1,-10)*{}="4";
{\ar "2";"1"};{\ar "4";"3"};
\endxy};
(30,0)*{
\xy
(0,0)*{}="1";
(-5,-5)*{}="2";
(5,-10)*{}="3";
{\ar "2";"1"}; {\ar "3";"1"};
\endxy};
\endxy
\]

\[
\xy
(-20,0)*{
\xy
(-5,-10)*{}="1";
(0,0)*{}="2";
(5,-5)*{}="3";
{\ar "1";"2"}; {\ar "2";"3"};
\endxy};
(0,0)*{
\xy
(-5,-5)*{}="1";
(0,0)*{}="2";
(5,-10)*{}="3";
{\ar "1";"2"}; {\ar "2";"3"};
\endxy};
(20,0)*{
\xy
(0,0)*{}="1";
(0,-5)*{}="2";
(0,-10)*{}="3";
{\ar "1";"2"}; {\ar "3";"2"};
\endxy};
(40,0)*{
\xy
(-1,0)*{}="1";
(-1,-10)*{}="2";
(1,0)*{}="3";
(1,-10)*{}="4";
{\ar "2";"1"};{\ar "3";"4"};
\endxy};\endxy
\]

\[
\xy
(-20,0)*{
\xy
(-1,0)*{}="1";
(-1,-10)*{}="2";
(1,0)*{}="3";
(1,-10)*{}="4";
{\ar "1";"2"};{\ar "3";"4"};
\endxy};
(0,0)*{
\xy
(0,10)*{}="1";
(0,5)*{}="2";
(0,0)*{}="3";
{\ar "1";"2"}; {\ar "2";"3"};
\endxy};
(20,0)*{
\xy
(0,0)*{}="1";
(-5,-5)*{}="2";
(5,-10)*{}="3";
{\ar "1";"2"}; {\ar "1";"3"};
\endxy};
\endxy
\]
\end{proposition}

Note that the first row consists exactly of the excluded subgraphs for Up graphs.  Thus if, in a graph corresponding to a basis monomial for $\pvq$, we look only at the subgraph of upward arrows, we see that this subgraph must be an Up graph (and all Up graphs may arise).

The second row of excluded subgraphs features `mixed' subgraphs, in that they each involve both an up arrow and a down arrow.  It is clear that the non-excluded (= permitted) mixed subgraphs must be the following:

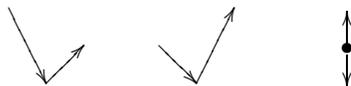
\[
\xy
(-20,0)*{
\xy
(-5,10)*{}="1";
(0,0)*{}="2";
(5,5)*{}="3";
{\ar "1";"2"}; {\ar "2";"3"};
\endxy};
(0,0)*{
\xy
(-5,5)*{}="1";
(0,0)*{}="2";
(5,10)*{}="3";
{\ar "1";"2"}; {\ar "2";"3"};
\endxy};
(20,0)*{
\xy
(0,0)*{}="1";
(0,-5)*{\bullet}="2";
(0,-10)*{}="3";
{\ar "2";"1"}; {\ar "2";"3"};
\endxy};
\endxy
\]

As is readily seen, the effect of these excluded and non-excluded mixed subgraphs is to ensure that different Up trees (which use only up arrows) which are connected to each other by down arrows are in fact only connected to each other by down arrows between their minimal elements.

Finally, the last row of excluded subgraphs involve only downward pointing arrows, and consist precisely of the excluded subgraphs for Down graphs.  Thus a graph which excludes all the subgraphs listed in the proposition will be an Up-Down graph, and conversely.  Thus Theorem \ref{ThmKoszulBasis} follows from Proposition \ref{PropKoszulBasis}.

\begin{remark} It is not hard to see that the Up-Down graphs on $[n]$ consist exactly of the red-black graphs corresponding to the monomials referred to in Proposition 4.5 of \cite{BEER}.\footnote{The Down and Up graphs correspond respectively to red and black graphs in the terminology used in the definition of 2-step partition immediately prior to Proposition 4.5 of \cite{BEER}.}  These monomials are shown in that proposition to form a basis of a certain algebra $QA_n^0$ related to $\pvq$:  namely, after making a certain change of basis to $\pvq = U(\mathfrak{qtr_n})^!$, \cite{BEER} show that a certain filtration is defined on $\pvq$.  Then $QA_n^0$ is the quadratic approximation to the associated graded of $\pvq$ with respect to that filtration.  The given basis for $QA_n^0$ is then used to find the Hilbert series and to prove the Koszulness of $QA_n^0$, which in turn lead to the Hilbert series and Koszulness of $\pvq$.  It is interesting that the same collection of (Up-Down) graphs can be used to index a basis of $\pvq$ itself and show directly that it is Koszul, as we shall see next.
\end{remark}

\begin{proof}[Proof of the Proposition]

To begin with we linearly order the generators $\{\rij: 1\leq i \ne j \leq n\}$ of $\pvq$ using the numerical order of the indices, i.e. $\rij > \rkl \iff (i>k) \text{ or } (i=k \text{ and } j>l)$.\footnote{Strictly speaking $\pvq$ is generated by dual generators $\{\rij^*\}$.  As per footnote \ref{rijfootnote}, and to simplify the notation, we write $\rij$ instead of $\rij^*$.}  Then, given a wedge product of generators, we first order the generators in the product in increasing order, and then we linearly order such monomials first by length and then lexicographically (we also agree that 0 has length 0, so that $0<u$ for all non-zero $u$).  This ordering (which we refer to as the lexicographical ordering) is multiplicative in the sense that if $u,v,w$ are wedge products such that $u > v$ and $uw \ne 0$ then $uw >vw$.

We wish to define a set $\Eii$ of `illegal' degree 2 monomials, consisting of those degree 2 monomials which can be expressed as linear combinations of `smaller' monomials (with respect to the lexicographical ordering) using the defining relations of $\pvq$.  Unfortunately, the set $\Eii$ cannot be read off directly from the relations in the form (\ref{vrel}), (\ref{arel}) and (\ref{qAS}) as some of these have the same maximal terms.

However one readily finds that those relations can be put in the following equivalent form (where $1\leq i<j<k \leq n$):

\begin{align}
\rik\w\rjk &= \rij\w\rjk - \rji\w\rik \label{Gbasis1} \\
\rkj\w\rji &= \rji\w\rik - \rji\w\rjk - \rji\w\rki \notag \\
\rki\w\rkj &= \rki\w\rij - \rji\w\rik + \rji\w\rjk + \rji\w\rki  \notag \\
\rik\w\rkj &= \rij\w\rjk - \rij\w\rik  \notag \\
\rjk\w\rki &= \rji\w\rik - \rji\w\rjk  \notag \\
\rij\w\rkj &= \rij\w\rjk - \rij\w\rik - \rki\w\rij \label{Gbasis6}
\end{align}
as well as the relations (\ref{qAS}).  Each relation now has a distinct maximal term, and these have been collected on the LHS above.  Thus $\Eii$ consists of the union of the sets:

\begin{align*}
& \{ \rjk\w\rik,\ \rkj\w\rji,\ \rkj\w\rki: 1\leq i<j<k \leq n \} \\
& \{\rik\w\rkj,\ \rjk\w\rki,\ \rij\w\rkj,\  1\leq i<j<k \leq n \} \\
& \{\rij\w\rji,\ \rij\w\rij: 1\leq i \ne j \leq n\}
\end{align*}

These monomials are readily seen to correspond with the excluded diagrams of the Proposition.  The Proposition will be proved if we can show that the set $\overline{S}$ of monomials which do not contain any of the excluded 2-letter monomials $\Eii$ (even after re-ordering of the generators forming the monomial) comprise a basis for $\pvq$.

The proof of this fact is in the following two steps:
\begin{itemize}
\item show that the set $\overline{S}$ generates $\pvq$; and
\item show that $\overline{S}$ has the same number of elements in each degree as the basis for $\pvq$ given in Theorem \ref{QuadDualBasis} (which implies that the elements of $\overline{S}$ are linearly independent, and hence form a basis).
\end{itemize}

The fact that $\overline{S}$ generates $\pvq$ is easy, since if we have a monomial which contains (possibly after reordering its factors) an excluded 2-letter monomial, we can replace the monomial by a sum of terms in which the excluded 2-letter monomial is replaced by a smaller, legal 2-letter monomial.  It is clear that all of these terms are strictly smaller than the original monomial with respect to the lexicographical ordering, because of the multiplicative property of that ordering.  Hence, repeating if necessary, we must eventually reach a sum of terms none of which contains an excluded 2-letter submonomial, even after reordering of its factors - i.e. a sum of terms belonging to $\overline{S}$.

The fact that $\overline{S}$ has the same number of elements in each degree as the basis for $\pvq$ given in Theorem \ref{QuadDualBasis} is also straightforward.  Let us consider again the procedure described above for creating Up-Down graphs:
\begin{itemize}
\item First, take an unordered partition of $[n]$ into some number $l\leq n$ of cyclically ordered subsets (and form the unique Up graphs determined by the cyclically ordered subsets) - the number of ways of doing this is $s(n,l)$, where $s(-,-)$ denotes (unsigned) Stirling numbers of the first kind.  It is easy to see that the resulting Up forests have $(n-l)$ arrows, so that the resulting monomials have degree $(n-l)$.  We let $m_i$ denote the minimal element of cycle $C_i$ for $i-1,\dots,l$.
\item Second, take an unordered partition of $\mathcal{M}:=\{m_i:i=1,\dots,l\}$ as $\mathcal{M}=M_1  \sqcup \dots \sqcup M_k$, where the $M_i$ are unordered, and form the unique Down graph determined by this partition of $\mathcal{M}$.  The number of ways of doing this is $S(l,k)$, where $S(-,-)$ denotes (unsigned) Stirling numbers of the second kind.  It is easy to see that the resulting Down forests have $(l-k)$ arrows, so that the resulting monomials have degree $(l-k)$.
\end{itemize}

It is clear that the resulting Up-Down graph will have $(n-k)=(n-l)+(l-k)$ arrows, and hence will correspond to a degree $(n-k)$ monomial.

Thus if $\bar{S}^{n-k}$ denotes the monomials in $\bar{S}$ of degree $(n-k)$ we find:

\begin{equation*}
\text{dim } \bar{S}^{n-k} = \sum_{l=k}^n s(n,l) S(l,k) = L(n,k) = \text{dim } A^{!(n-k)}
\end{equation*}

For the last equality we used Corollary \ref{LahBasis}, and for the second-last equality we used the so-called Lah-Stirling identity:
\begin{equation*}
L=s S
\end{equation*}
where $L$, $s$ and $S$ are infinite-dimensional lower-triangular matrices whose $(n,k)$-th entries are, respectively, $L(n,k)$ (Lah numbers of Corollary \ref{LahBasis}), $s(n,k)$ and $S(n,k)$.  (See \cite{Riordan}, Problem 16(d) p. 44.)

This completes the proof.
\end{proof}

\begin{corollary}
\label{GrobnerBasis}
The algebra $\pvq$ (and hence also $\pv$) is Koszul.
\end{corollary}
\begin{proof}
The fact that the monomials $\overline{S}$ not containing any of the 2-letter monomials $\Eii$ form a basis for $\pvq$ means that the equations (\ref{Gbasis1})-(\ref{Gbasis6}) and (\ref{qAS}) (whose leading terms are the $\Eii$) constitute a Gr\"obner basis for $\pvq$ (as a skew-commutative algebra - see \cite{Mikha}).  This Gr\"obner basis is quadratic, and hence by a result of \cite{Yuz} \footnote{Theorem 6.16.}, $\pvq$ is Koszul.
\end{proof}

\begin{remark} $\pvq$ as a `Product' of the families $\pfq$ and $\pbq$
\label{RemarkPVBasProduct}
\end{remark}

Given the correspondence between Down forests and $\pfq$, and between Up forests and $\pbq$, identified in Remarks \ref{Rkpfbbasis} and \ref{Rkpbbasis}, Theorem \ref{ThmKoszulBasis} suggests that the family of all $\pvq$ (parametrized by $n$) may be some kind of `product' of the families of the $\pfq$ and $\pbq$.  Indeed, one could express the Lah-Stirling identity above in the form:

\begin{equation*}
\text{dim } \mathfrak{pvb}_n^{!n-k} = L(n,k) = \sum_l s(n,l) S(l,k) = \sum_l \text{dim }  \mathfrak{pb}_n^{!(n-l)}\ \text{dim } \mathfrak{pfb}_l^{!(l-k)}
\end{equation*}

As pointed out in \cite{BEER}, $\pb$ may be viewed as a quotient of $\pv$ by $\pf$.  However, this does not explain why one might be able to view $\pvq$ as the kind of product of the families $\pbq$ and $\pfq$ suggested by the Lah-Stirling identity.

\section{Final Remarks}

\subsection{Other Groups}

One could seek to apply the PVH Criterion to determine whether other groups are quadratic.  One group that comes to mind is the pure cactus group $\Gamma$, as developed for instance in \cite{EHKR}.  As a preliminary step, it would be useful to find a presentation for $\Gamma$, and to show that the quadratic approximation to the associated graded of $\Q\Gamma$ with respect to the filtration by powers of the augmentation ideal (i.e. the universal enveloping algebra of the holonomy Lie algebra of $\Gamma$) is Koszul (at least up to homological degree 2).

There are many other groups to which one could seek to apply the PVH Criterion.  These include: Torelli groups of surfaces (see for instance \cite{HaMa}, questions 8.1 and 8.3); pure braid groups of surfaces (see for instance \cite{CEE}); pure braid groups of Coxeter groups (see for instance \cite{tDieck} and \cite{Cher}) and virtual virtual braid groups of Coxeter groups (see \cite{Thiel}).  It would also be interesting to define a notion of pure virtual braid groups of surfaces and determine whether they are quadratic.  In the case of some of the above groups, quadraticity has already been established by transcendental means, but obtaining a purely combinatorial proof through the use of the PVH Criterion would again be of interest. (I am indebted to G. Massuyeau, P. Etingof and E. Rains for suggesting many of these possible applications of the PVH Criterion.)

\subsection{Generalizing the PVH Criterion}

As mentioned in the Introduction, the PVH Criterion arguably lives naturally in a broader context than we have explored here, such as perhaps augmented algebras over an operad (or the related `circuit algebras' of \cite{DBN-WKO}).

In a different direction, one could try to generalize the criterion to deal with filtrations of an algebra by powers of an ideal other than an augmentation ideal.  A particular case of this deals with groups that exhibit a `fibering'.  For instance the virtual braid group $vB_n$ fits into an exact sequence:

\begin{equation*}
1 \rightarrow \PV \rightarrow vB_n \rightarrow S_n \rightarrow 1
\end{equation*}
where $S_n$ is the symmetric group.  (Similar sequences exist for the braid group and the cactus group.)  In such cases it is more interesting to consider the ideal corresponding to the kernel of the induced homomorphism $\Q vB_n \rightarrow \Q S_n$, rather than the augmentation ideal of $\Q vB_n$.  The extension of the PVH Criterion to cover these particular ideals should not be too difficult, but dealing with more general ideals could be interesting.


\begin{thebibliography}{99}

%
\bibitem{ABCKT}
Amoros, J., Burger, M., Corlette, K., Kotschick, D., Toledo, D.: Fundamental Groups of K\"ahler Manifolds, AMS, Providence, Rhode Island (1996)
%
\bibitem{Arnold}
Arnold, V.I.,: The cohomology ring of the colored braid group, Math. Notes Acad. Sci. USSR, 5, 138-140 (1969)
%
\bibitem{Bard}
Bardakov, V.: The virtual and universal braids, Fundamenta Mathematicae, 184, 1-18 (2004)
%
\bibitem{DBN-WKO}
Bar-Natan, D., Dancso, Z.: Finite type invariants of w-knotted objects: from Alexander to Kashiwara and Vergne (in preparation), available at http://www.math.toronto.edu/$\sim$drorbn/papers/WKO/WKO.pdf
%
\bibitem{BNStoi}
Bar-Natan, D., Stoimenow, A.: The fundamental theorem of Vassiliev invariants, Geometry and Physics, Lecture Notes in Pure and Applied Math., 184. Marcel-Dekker, New York (1997)
%
\bibitem{BEER}
Bartholdi, L., Enriquez, B., Etingof, P., Rains, E.: Groups and Lie algebras corresponding to the Yang-Baxter equations, J. Algebra, 305:2, 742-764 (2006)
%
\bibitem{Bez}
Bezrukavnikov, R.: Koszul DG-algebras arising from configuration spaces, Geom. and Funct. Analysis, 4:2, 119-135 (1994)
%
\bibitem{CEE}
Calaque, D., Enriquez, B., Etingof, P.: Universal KZB equations I: the elliptic case, appearing in Yu. Tschinkel, Yu. Zarhin, eds., Algebra, Arithmetic and Geometry In Honour of Yu. Manin, Volume I, Progress in Mathematics, 269 (2009)
%
\bibitem{Cenkl}
Cenkl, B., Porter, R.: Malcev's completion of a group and differential forms, J. Diff. Geom., 15, 531-542 (1980)
%
\bibitem{Cher}
Cherednik, I.: Quantum Knizhnik-Zamolodchikov equations and affine root systems, Comun. Math. Phys., 150, 109-136 (1992)
%
\bibitem{EHKR}
Etingof, P., Henriques, A., Kamnitzer, J., Rains, E.: The cohomology ring of the real locus of the moduli space of stable curves of genus 0 with marked points, Ann. Math., 171:2, 731-777 (2010)
%
\bibitem{Froberg}
Fr\"oberg, R.: Koszul algebras, in Advances in Commutative Ring Theory, Proc. Fez
Conf. 1997, Lecture Notes in Pure and Applied Mathematics, 205. Marcel-Dekker, New York (1999)
%
\bibitem[GPV]{GPV}
Goussarov, M., Polyak, M., Viro, O.: Finite type invariants of classical and virtual knots, Topology, 39:5, 1045-1068 (2000)
%
\bibitem{HaMa}
Habiro, K., Massuyeau, G.: Symplectic Jacobi diagrams and the Lie algebra of homology cylinders, J. Topology, 2:3, 527-569 (2009)
%
\bibitem{HilStam}
Hilton, P.J., Stammbach, U.: A Course in Homological Algebra, 2nd. Ed. Springer-Verlag, New York (1997)
%
\bibitem{Hille}
Hille, L.: Selected topics in representation theory 3 -- Quadratic algebras and Koszul algebras, online publication available at http://www.mathematik.uni-bielefeld.de/~sek/select/hille3.pdf
%
\bibitem{Hutchings}
Hutchings, M.: Integration of singular braid invariants and graph cohomology, Trans. A.M.S., 350:5, 1791-1809 (1998)
%
\bibitem{Kohno}
Kohno, T.: S\'erie de Poincar\'e-Koszul associ\'ee aux groupes de tresses pures, Invent. Math., 82:1, 57-75 (1985)
%
\bibitem{Kraehmer}
Kraehmer, U.: Notes on Koszul algebras, online publication available at http://www.maths.gla.ac.uk/~ukraehmer/connected.pdf
%
\bibitem{Lambe}
Lambe, L.A.: Two exact sequences in rational homotopy theory, Proc. A.M.S., 96:2, 360-364 (1986)
%
\bibitem{Lau}
Lauda, A.: Braid tutorial, http://www.math.columbia.edu/~lauda/xy/ \\* braidtutorial
%
\bibitem{Lin}
Lin, X.-S.: Power series expansions and invariants of links, AMS/IP Studies in Advanced Mathematics,  2:1, 184-202 (1997)
%
\bibitem{MKS}
Magnus, W., Karrass, A., Solitar, D.: Combinatorial Group Theory, 2nd. Rev. Ed. Dover Publications,  Mineola (1976)
%
\bibitem{MarMc}
Margalit, D., McCammond, J.: Geometric presentations for the pure braid group, J. Knot Theory Ramif., 18 1-20 (2009)
%
\bibitem{Mikha}
Mikhalev, A.A., Vasilieva, E.A.: Standard bases of ideals of free supercommutative polynomial algebras. In: Lie Algebras, Rings and Related Topics, 108-125. Springer-Verlag,  Hong Kong (2000)
%
\bibitem{PapSu}
Papadima, A., Suciu, A.I.: Chen Lie algebras, IMRN Int'l Math. Res. Notices, 21, 1057-1086 (2004)
%
\bibitem{Pol}
Polishchuk, A., Positselski, L.: Quadratic Algebras, AMS, Providence (2005)
%
\bibitem{Polyak}
Polyak, M.: On the algebra of arrow diagrams, Lett. Math. Phys., 51, 275-291 (2000)
%
\bibitem{PosVish}
Positselski, L., Vishik, A.: Koszul duality and Galois cohomology, Math. Res. Lett., 2, 771-781 (1995)
%
\bibitem{Priddy}
Priddy, S.B.: Koszul resolutions, Trans. A.M.S., 152, 39-60 (1970)
%
\bibitem{Quillen}
Quillen, D.G.: On the associated graded ring of a group ring, J. Algebra, 10, 411-418 (1968)
%
\bibitem{Quillen2}
Quillen, D.G.: Rational homotopy theory, Ann. of Math., 90:2, 205-295 (1969)
%
\bibitem{Riordan}
Riordan, J.: Introduction to Combinatorial Analysis, Wiley, New York (1958)
%
\bibitem{ShelYuz}
Shelton, B., Yuzvinsky, S.: Koszul algebras from graphs and hyperplane arrangements, J. London Math. Soc., 56:3, 477-490 (1997)
%
\bibitem{Sullivan}
Sullivan, D.: Infinitesimal computations in topology, Publ. I.H.E.S., 47, 269-331 (1977)
%
\bibitem{Thiel}
Thiel, A.-L.: Virtual braid groups of Type B and weak categorification, J. Knot Theory Ramifications, 21:2 (2012)
%
\bibitem{tDieck}
tom Dieck, T.: Symmetrische Br\"ucken und Knotentheorie zu den Dynkin-Diagrammen vom Typ B, J. Reine Angew. Math., 451, 71-88 (1994)
%
\bibitem{Yuz}
Yuzvinsky, S.: Orlik-Solomon algebras in algebra and topology, Russ. Math. Surv., 56:2, 293-364 (2001)
%
\end{thebibliography}
\end{document}